\documentclass{article}
\usepackage[utf8]{inputenc}
\usepackage[margin=1.5in]{geometry}
\usepackage{amsmath}
\usepackage{amsthm, ifthen, amsfonts, amssymb,
srcltx, amsopn, xcolor, enumerate, mathabx, ulem}
\usepackage{mathtools}
\usepackage{enumerate}
\usepackage[colorlinks=true, pdfstartview=FitV, linkcolor=blue, citecolor=blue, urlcolor=blue]{hyperref}
\usepackage[nameinlink]{cleveref}
\usepackage{url}
\usepackage{graphicx}
\usepackage{subcaption}
\usepackage{quiver}
\usepackage{tikz-cd}
\usepackage{overpic}

\newtheorem{thm}{Theorem}[section]
\newtheorem{prop}[thm]{Proposition}
\newtheorem{lem}[thm]{Lemma}
\newtheorem{cor}[thm]{Corollary}
\theoremstyle{definition}
\newtheorem{defn}[thm]{Definition}
\newtheorem{remark}[thm]{Remark}

\newtheorem{claim}[thm]{Claim}

\newtheorem{ques}{Question}
\newenvironment{claimproof}[1][Proof of Claim]{%
  \begin{proof}[#1]%
}{%
  \end{proof}%
}

\DeclarePairedDelimiter\ceil{\lceil}{\rceil}

\newcommand{\ov}{\overline}
\newcommand{\cay}{\mathrm{Cay}(G,  S)}
\newcommand{\pres}{\ensuremath{\langle S \mid R\rangle} }

\newcommand{\gpres}{\ensuremath{\langle S\mid \bigcup_{i=1}^\infty R_i \rangle}}
\newcommand{\npres}{\ensuremath{\langle S\cup T \mid \bigcup_{i=1}^\infty Q_i\rangle}}
\newcommand{\rint}[1]{\ensuremath{\rho_{\mathrm{rel}, #1}}}

\newcommand*\mc[1]{\mathcal{#1}}
\newcommand{\N}{\mathbb{N}}

\newcommand{\R}{\mathbb{R}}
\newcommand{\Z}{\mathbb{Z}}
\newcommand{\mb}{\partial_*}

\DeclarePairedDelimiter\abs{\lvert}{\rvert}%
\newcommand{\dist}{\textup{\textsf{d}}}
\newcommand{\diam}{\textup{\textsf{diam}}}
\newcommand{\Cay}{\text{Cay}}

\newcounter{scomments}

\newcounter{ccomments}

\newcounter{dcomments}

\newcounter{tcomments}

\newcommand{\eps}{\varepsilon}

\title{Connections between the topology of the Morse boundary, the Morse local-to-global property and acylindrical hyperbolicity}

\author{Carolyn Abbott and Stefanie Zbinden}
\date{}

\begin{document}

\maketitle

\begin{abstract}
We relate the topology of the Morse boundary of a group to geometric and algorithmic properties of the group. In particular, we show that a group has $\sigma$-compact Morse boundary if and only if it is Morse local-to-global. We also provide tools such as the geodesic Morse local-to-global property to show that groups are (not) Morse local-to-global. Our strategy generalizes tools from small cancellation theory, such as the intersection of relators, to arbitrary finitely generated groups. Further, we introduce a class of groups akin to graded small-cancellation groups and show that, for groups in this class, a geodesic is Morse if and only if its intersection with relators grows sublinearly in the length of the relators.

We use this to construct the first example of a non-virtually cyclic Morse local-to-global group with an infinite-order Morse element that is not acylindrically hyperbolic. 
\end{abstract}

\section{Introduction}
Since the introduction of hyperbolic and relatively hyperbolic groups by Gromov \cite{Gromov-hyp}, many different notions of negative curvature for groups have been introduced. For instance, being (weakly) acylindrically hyperbolic, having contracting geodesics, admitting a loxodromic WPD element in an action on a hyperbolic space, having hyperbolically embedded subgroups, being a non-product hierarchically hyperbolic group, or having the Morse local-to-global property without being Morse limited \cite{O:acylindrical, H:cohom, S:random_walk, DGO:rotating_families, BF:cohomMCG, BBF:cohom, BHS2, russellsprianotran:thelocal}.   
Many of these properties are known to  imply or not imply others \cite{BBF:cohom, S:hypemb, BBFS, HHS1}, and several have been shown to be equivalent to acylindrical hyperbolicity by Osin in his landmark paper \cite{O:acylindrical}.

One common feature among all the above notions of negative curvature is the existence of Morse geodesics. In a hyperbolic space, all geodesics are Morse, a property that, in fact, characterizes hyperbolic spaces \cite{Gromov-hyp}. In light of this, Morse geodesics are often thought of as the ``hyperbolic-like" directions in a space, and negatively or non-positively curved spaces that are not hyperbolic typically have a combination of Morse and non-Morse geodesics. The collection of Morse geodesics can be organized using the Morse boundary of the space \cite{C:Morse}. Modeled on the Gromov boundary of a hyperbolic space, the Morse boundary consists of equivalence classes of Morse geodesic rays in (a Cayley graph of) a finitely generated group. The condition that a group contains a Morse ray is equivalent to having non-empty Morse boundary. However, a non-empty Morse boundary alone is not sufficient to ensure that the group satisfies any of the above kinds of negative curvature. For example, Osin, Olshanskii, and Sapir construct torsion groups that Fink shows contain Morse geodesics \cite{OlshankiiOsinSapir:lacunary, fink:morse}.

\begin{ques}\label{ques:connection?}
    Are there topological conditions on the Morse boundary that detect any of the above generalizations of hyperbolicity?
\end{ques}

We give a positive answer to this question by connecting the topology of the Morse boundary to the geometry of the group and to its language theory.

\begin{thm}\label{thm:main1}
    For a finitely generated group $G = \pres$, the following are equivalent.
    \begin{enumerate}[(1)]
        \item \label{item:StrSigmaCpt} $G$ has strongly $\sigma$--compact Morse boundary
        \item \label{item:IPSC} $G$ does not satisfy the general IPSC condition.
        \item \label{item:GeodMLTG} $G$ is geodesic Morse local-to-global.
        \item \label{item:MLTG} $G$ is Morse local-to-global.
        \item \label{item:SigmaCpt} $G$ has $\sigma$--compact Morse boundary.
        \item \label{item:RegularLang} Morse geodesics in $\cay$ form regular languages. That is, for every Morse gauge $M$, there exists a Morse gauge $M'$ and regular language $L_M$ consisting of $M'$--Morse geodesics $\cay$ which contains all $M$--Morse geodesics in $\cay$.
    \end{enumerate}
\end{thm}

The Morse boundary of a group is \textit{(strongly) $\sigma$-compact} if it is the union (direct limit) of countably many compact sets; see Definition~\ref{def:stronglysigmacpt}.  Compact subsets of the Morse boundary correspond to closed subsets of Morse strata \cite{CH:stable_asdim, CD:stable, cordes2024corrigendum}, so having a $\sigma$--compact Morse boundary means that there is a countable collection of Morse gauges $M_i$ such that every Morse geodesic is $M_i$--Morse for some $i$.

The \emph{IPSC condition} is a technical condition introduced in \cite{Zbinden:small-cancellation} for small cancellation groups that controls intersection patterns of relators. In this paper, we introduce \textit{projective geodesics} in an arbitrary finitely generated group, which generalize the notion of intersection of relators; see Definition~\ref{def:ProjQuadrangle}. The \emph{general IPSC} condition is an analogue of the IPSC condition for arbitrary finitely generated groups that controls projective geodesics; see Definition~\ref{def:generalIPSC}.

Hyperbolic groups are characterized by a local-to-global property: every local quasi-geodesic is a global quasi-geodesic \cite{Gromov-hyp}. \emph{Morse local-to-global groups}, defined by Russell, Spriano, and Tran \cite{russellsprianotran:thelocal}, are characterized by the same property for local quasi-geodesics that are also locally Morse, that is, those that locally look like a quasi-geodesic in a hyperbolic space; see Definition~\ref{def_morse-local-to-global}. A group is \emph{geodesic Morse local-to-global} (Definition~\ref{def:geodesicMLTG}) if all \textit{geodesics} which are locally Morse are globally Morse.  In practice, the geodesic Morse local-to-global property is often easier to check than the Morse local-to-global property. 

Being Morse local-to-global has several strong consequences, including a combination theorem for stable subgroups, a growth gap and excellent algorithmic properties \cite{russellsprianotran:thelocal, CordesRussellSprianoZalloum:regularity, DrutuSprianoZbinden:weak}. Moreover, non-elementary Morse local-to-global groups always admit a stable free subgroup \cite[Corollary~I]{russellsprianotran:thelocal}. Stable subgroups of a finitely generated group are a natural generalization of quasi-convex subgroups of hyperbolic groups. Stable subgroups that are virtually $\mathbb F_2$ appear in many of the notions of negative curvature described above, for example, as the subgroup generated by two independent loxodromic WPD elements in an acylindrically hyperbolic group.  

\begin{cor}\label{cor:sigma-cmpact-f2} If the Morse boundary of a finitely generated group $G$ is $\sigma$-compact and has cardinality at least $3$, then $G$ has a stable $\mathbb F_2$ subgroup.
\end{cor}

Corollary~\ref{cor:sigma-cmpact-f2} fits into a body of work showing that the Morse boundary detects algebraic and subgroup properties of groups \cite{KarrerMiraftabZbinden:subgroups, CD:stable}. Corollary~\ref{cor:sigma-cmpact-f2} suggests that when the Morse boundary of $G$ is $\sigma$--compact, it should be possible to use the action of the $G$ on its Morse boundary to play ping-pong. Hence we expect the study of the dynamics of the action on the Morse boundary in this context to be profitable.

Theorem~\ref{thm:main1} is known for $C'(1/9)$--small-cancellation groups \cite{HeSprianoZbinden:sigma-compact}. Specifically, \eqref{item:MLTG}$\implies$\eqref{item:RegularLang} and \eqref{item:RegularLang}$\implies$\eqref{item:StrSigmaCpt} are shown in \cite{CordesRussellSprianoZalloum:regularity} and \cite{HeSprianoZbinden:sigma-compact}, respectively, for all finitely generated groups, and \eqref{item:StrSigmaCpt}$\implies$\eqref{item:SigmaCpt} follows from the definition. The implications \eqref{item:SigmaCpt}$\iff$\eqref{item:IPSC} and \eqref{item:IPSC}$\implies$\eqref{item:GeodMLTG}$\implies$\eqref{item:MLTG} are shown in \cite{Zbinden:small-cancellation} and \cite{HeSprianoZbinden:sigma-compact}, respectively, for $C'(1/9)$--small-cancellation groups. This is illustrated in Figure~\ref{fig:known-and-new-equivalences}. The proofs in the small cancellation setting  rely heavily on the geometric restrictions coming from being a small-cancellation group, specifically on the description of Morse geodesics in such groups as those whose intersection with relators is sublinear \cite[Corollary~4.14]{arzhantseva2019negative}. 

In this paper, we generalize these geometric restrictions from small cancellation groups to arbitrary finitely generated groups.  One key tool, mentioned above, is the notion of a \textit{projective loop} and a \textit{projective geodesic}; see Definition~\ref{def:ProjQuadrangle}.  We also define the \emph{intersection function} of a geodesic, which measures the length of subsegments of the geodesic that are projective geodesics in a projective loop; see Definition~\ref{def:IntFcn}.
Using \cite[Theorem~1.4]{ACHG:contraction_morse_divergence}, we prove the following key lemma.

\begin{lem}[Key lemma]\label{lem:key-lemma1}
     A geodesic in a geodesic metric space is Morse if and only if its intersection function is sublinear.
\end{lem}

In light of Lemma~\ref{lem:key-lemma1}, projective loops should be thought of as analogues of relators from small-cancellation groups, and projective geodesics roughly capture intersections with relators. 
With these generalizations and the key lemma in hand, the proofs of \eqref{item:SigmaCpt}$\implies$\eqref{item:IPSC}$\implies$\eqref{item:GeodMLTG} for small-cancellation groups can be adapted to the setting of finitely generated groups.

For \eqref{item:GeodMLTG}$\implies$\eqref{item:MLTG}, a main ingredient is that any locally Morse locally quasi-geodesic path that stays ``somewhat'' close to the geodesic between its endpoints has to stay ``very close''; see Lemma~\ref{lem:Morse_for_combing_lines}. Finally, for \eqref{item:SigmaCpt}$\implies$\eqref{item:StrSigmaCpt}, we carefully study the concatenation of quasi-geodesics.  Our main tool is the following lemma.

\begin{lem}%[Key lemma 2]
\label{lem:key-lemma2}
     Let $G$ be a non-hyperbolic finitely generated group with non-empty Morse boundary, and let $X = \cay$ for some finite generating set $S$ of $G$. There exists a constant $C$ such that for any sequence $(\gamma_n)_n$ of geodesic segments in $X$, there exists a $C$--quasi-geodesic $\eta$ such that
    \begin{enumerate}
        \item for all $n\in \N$, there exists a translate of $\gamma_n$ which is a subsegment of $\eta$; and
        \item if $\gamma_n$ is $M$--Morse for all $n\in \N$, then $\eta$ is $M'$--Morse for a Morse gauge $M'$ only depending on $M$.\label{prop:morse-carries}
    \end{enumerate}
\end{lem}

\subsection{Connections to acylindrical hyperbolicity}

We next investigate the connection between the Morse local-to-global property and acylindrical hyperbolicity. Not all Morse local-to-global groups are acylindrically hyperbolic: there are straightforward obstructions, such as the group being virtually cyclic or not containing a Morse element. 
We show that, even avoiding these obstructions, acylindrical hyperbolicity does not follow from being Morse local-to-global, answering a question of Russell, Spriano, and Tran \cite[Question~5]{russellsprianotran:thelocal}.

\begin{thm}\label{thm:main2}
    There exists a finitely generated, non-virtually cyclic Morse local-to-global group with an infinite-order Morse element that is not acylindrically hyperbolic.
\end{thm}

The construction of the group in Theorem~\ref{thm:main2} is inspired by the authors' construction of a small cancellation group with a Morse element that cannot be loxodromic in any action on a hyperbolic space \cite{AZ}.  However, small cancellation groups are always acylindrically hyperbolic, so the construction from \cite{AZ} cannot produce a Morse local-to-global group that is not acylindrical hyperbolic. 

In this paper, we define a generalization of the small cancellation condition, which we call \textit{expanding graded small cancellation} (see Definition~\ref{def:ExpandingGSC}), that is closely related to graded small cancellation as defined by Olshanskii, Osin, and Sapir \cite{OlshankiiOsinSapir:lacunary}. The group $G$ we construct to prove Theorem~\ref{thm:main2} is an expanding graded small cancellation group.

We generalize tools introduced by Arzhantseva, Cashen, Gruber, and Hume \cite{arzhantseva2019negative} to understand Morse geodesics in expanding graded small-cancellation groups. In particular, we define an analogue of the intersection function in small cancellation groups that measures the length of the intersection of a geodesic with relators, which we call the \textit{relator intersection function}; see Definition~\ref{defn:relator-intersection-function}.

\begin{prop}\label{prop:morse-geodesics}
   Let $G = \pres = \gpres$ be a balanced expanding graded small cancellation group. A geodesic in $\cay$ is Morse if and only if its relator intersection function is sublinear. 
\end{prop}
    
The term \textit{balanced} in the statement of the proposition is a technical condition; see the discussion after Proposition~\ref{prop:connection-intersecting-relator-intersecting} for the definition.  Using Proposition~\ref{prop:morse-geodesics}, we show that the group $G$ contains an infinite-order Morse element. Moreover, we show that its Morse boundary is $\sigma$--compact by giving a sequence of explicit sublinear functions bounding the relator intersection function of each Morse geodesic. This is sufficient to conclude that $G$ is Morse local-to-global by Theorem~\ref{thm:main1}(4)$\iff$(5).

To show the group $G$ we build is not acylindrically hyperbolic, we show that it has the 
stronger property that \textit{no} element is a loxodromic isometry in an action of $G$ on a hyperbolic space. This property was introduced by Balasubramanya, Fournier-Facio, and Genevois and is called property (NL) \cite{balasubramanya2022property}. 

\begin{thm}\label{thm:MLTGNL}
    There exists a finitely generated, non-virtually cyclic Morse local-to-global group with an infinite order Morse element that has property (NL). 
\end{thm}

\begin{figure}
    \centering

    \begin{overpic}[width=\textwidth]{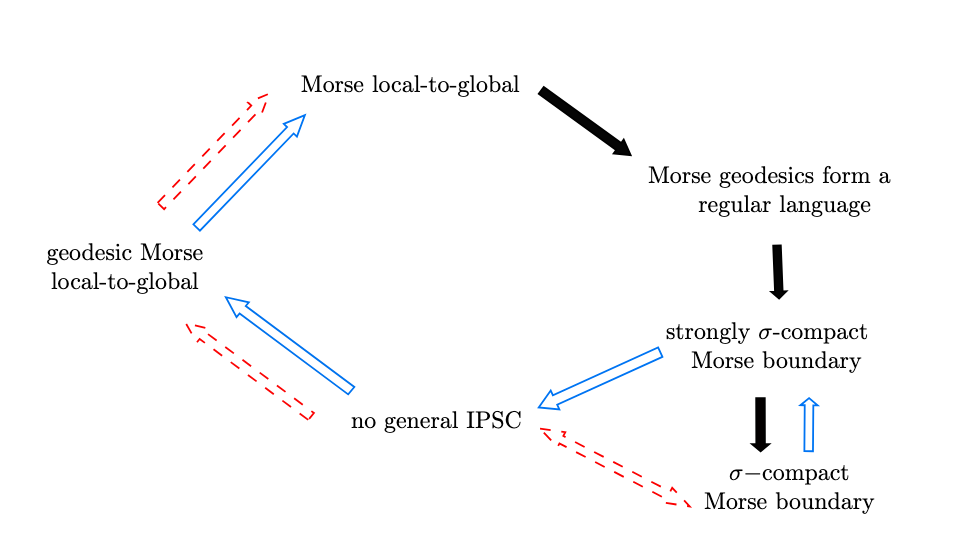}
    \put(40,45){\eqref{item:MLTG}}
    \put(23,30){\eqref{item:GeodMLTG}}
    \put(95,38){\eqref{item:RegularLang}}
    \put(92,21){\eqref{item:StrSigmaCpt}}
    \put(93,7){\eqref{item:SigmaCpt}}
    \put(44,10){\eqref{item:IPSC}}
    \put(10,43){\cite{HeSprianoZbinden:sigma-compact}}
    \put(14,17){\cite{HeSprianoZbinden:sigma-compact}}
    \put(53,6.5){\cite{Zbinden:small-cancellation}}
    \put(73,13){\text{def.}}
    \put(64,47){\cite{CordesRussellSprianoZalloum:regularity}}
    \put(83,29.5){\cite{HeSprianoZbinden:sigma-compact}}
    \put(28,37){Sec.~\ref{sec:MLTG}}
    \put(31,23){Sec.~\ref{sec:GeodMLTG}}
    \put(54.5,20.5){Sec.~\ref{sec:notIPSC}}
    \put(86,13){Sec.~\ref{sec:CptToStrCpt}}

     \end{overpic}
     \caption{Diagram of implications. The black solid arrows are implications already known for all finitely generated groups, the dotted red arrows are implications known for $C'(1/9)$--small-cancellation groups and the blue arrows are implications we prove in this paper for all finitely generated groups.}
    \label{fig:known-and-new-equivalences}
\end{figure}
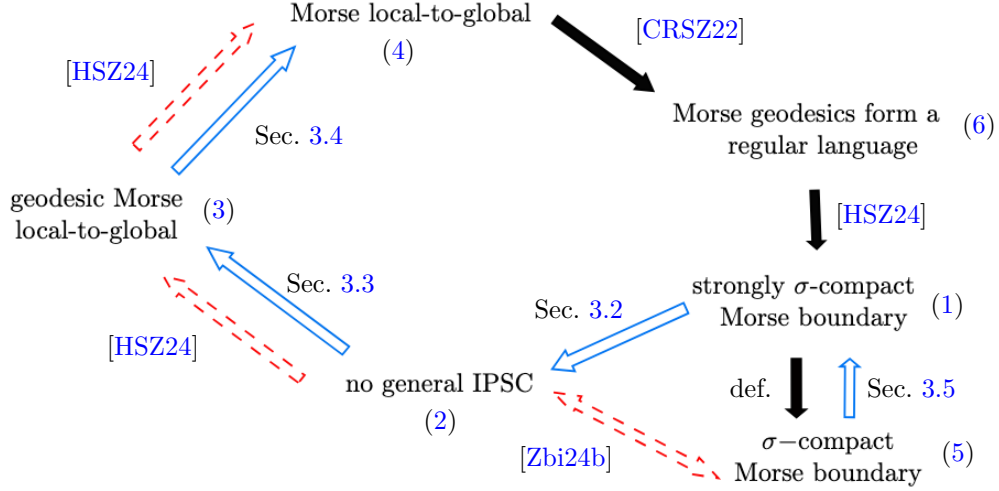

\subsection{Further directions} 

Theorem~\ref{thm:main1} provides new tools to check whether groups are Morse local-to-global.  In this paper, these tools were used to probe the connection between the Morse local-to-global property and acylindrical hyperbolicity. Another interesting direction would be to investigate the connection with the Dehn function of a group. 

\begin{ques}
    Can the tools developed in this paper be used to show that all groups with quadratic Dehn function are Morse local-to-global? 
\end{ques}

To generalize the proof of Theorem~\ref{thm:main1} from small-cancellation groups to finitely generated groups, we introduce projective loops, which exhibit  behavior similar to relators in small-cancellation groups. One can ask how far-reaching this correspondence between projective loops and relators is.   We say that a group admits a \textit{projective loop presentation} if all relators are projective loops.

\begin{ques}
    Which groups admit a projective loop presentation?  Which admit a \textit{finite} projecive loop presentation?
\end{ques}
 
Given the strong algebraic, geometric, and algorithmic consequences of admitting a small-cancellation presentation, one may expect similar consequences for groups admitting a finite projective loop presentation.   
For example:

\begin{ques}
    If $G$ admits a finite projective loop presentation, is $G$ Morse local-to-global? Does $G$ have solvable word problem?
\end{ques}

Projective loop presentations have geometric connections to the group, and so it is natural to ask whether it is possible to algorithmically compute such a presentation. While such a computation is likely difficult to do in general, perhaps it becomes more tractable if one starts with a presentation that already has good algorithmic properties, such as a Dehn presentation. 

\begin{ques}
    Are there classes of groups for which one can algorithmically compute a projective loop presentation? Can one compute a projective loop presentation from a Dehn presentation?   
\end{ques}

\paragraph{Outline.}
We discuss the preliminaries in Section~\ref{sec:prelim}. Section~\ref{sec:circle} is dedicated to the proof of Theorem~\ref{thm:main1}; see also Figure~\ref{fig:known-and-new-equivalences} for an illustration. We define the general IPSC condition in Section~\ref{sec:IPSC}, and the remainder of the section is devoted to proving the remaining equivalences in Theorem~\ref{thm:main1}. In particular, we prove \eqref{item:StrSigmaCpt}$\implies$ \eqref{item:IPSC} in Section~\ref{sec:notIPSC}; \eqref{item:IPSC}$\implies$\eqref{item:GeodMLTG} in Section~\ref{sec:GeodMLTG}; \eqref{item:GeodMLTG}$\implies$\eqref{item:MLTG} in Section~\ref{sec:MLTG}; and \eqref{item:SigmaCpt}$\implies$\eqref{item:StrSigmaCpt} in Section~\ref{sec:CptToStrCpt}.  In Section~\ref{sec:nonahgroup}, we introduce expanding graded small-cancellation groups and develop tools surrounding these groups. We then construct a particular expanding graded small-cancellation group that has $\sigma$-compact Morse boundary but is not acylindrically hyperbolic in Section~\ref{sec:nonAHgp}, proving Theorem~\ref{thm:main2}.

\subsection*{Acknowledgements} 

We would like to thank Rémi Coulon and Chris Cashen for helpful discussions.  The first author was partially supported by NSF grant DMS-2340341. The second author is supported by the Postdoc Mobility grant \#P500PT\_230322 of the Swiss National Science Foundation, the European Union (ERC, SATURN, 101076148), and the Deutsche Forschungsgemeinschaft (DFG, German Research Foundation) under Germany's Excellence Strategy - EXC-2047/1 - 390685813.

\section{Preliminaries}\label{sec:prelim}
Given a metric space $X$ and a subset $Y\subseteq X$, we denote the closed $r$--neighbourhood of $Y$ by $\mc N_r(Y)$. 

A map $f\colon (X,d_X)\to (Y,d_Y)$ of metric spaces is a \emph{$(\lambda,c)$--quasi-isometric embedding} for constants $\lambda \geq 1$ and $c\geq 0$ if for all $x,y\in X$, we have that
$$\lambda^{-1}\dist_X(x,y)-c\le \dist_Y(f(x),f(y))\le \lambda \dist_X(x,y)+c.$$ 
A \emph{quasi-isometry} is a quasi-isometric embedding which is also \emph{coarsely surjective}, meaning that $Y\subseteq \mc N_R(f(X))$ for some constant $R\ge 0$. 

A \emph{$(\lambda,c)$--quasigeodesic} in $X$ is a $(\lambda,c)$--quasi-isometric embedding of an interval into $X$. When the constants $\lambda$ and $c$ are the same, we simply call such path a $\lambda$-quasigeodesic. A \emph{geodesic} is a $(1,0)$-quasigeodesic, that is, an isometric embedding of an interval.

A metric space is \emph{geodesic} (resp. $(\lambda, c)$-quasigeodesic) if any two points are connected by a geodesic (resp. $(\lambda, c)$-quasigeodesic). For $\delta\geq 0$, a geodesic metric space $X$ is \emph{$\delta$--hyperbolic} if, for every three points $x,y,z\in X$, we have $[x,y] \subseteq \mc{N}_\delta( [x,z] \cup [z,y])$; we say that geodesic triangles in a $\delta$-hyperbolic space are $\delta$-\emph{slim}. If the particular choice of $\delta$ is not important, we simply say that $X$ is \emph{hyperbolic}.

Let $X$  be a geodesic metric space. For points $x, y\in X$, we denote by $[x, y]$ a fixed choice of geodesic from $x$ to $y$. By an abuse of notation, we denote the image $\mathrm{Im}(p)$ of a path $p\colon I \to X$   by $p$ and its initial and terminal endpoints by $p^-$ and $p^+$, respectively. Given $x, y\in p$, we denote by $[x, y]_p$ a choice of subsegment $p[s, t]$ such that $p(s) = x$ and $p(t) = y$ and $s\leq t$. We denote the length of (the domain of) $p$ by $\abs{p}$.

If $S$ denotes a finite set of formal variables, then ${ S}^{-1}$ denotes its formal inverses, and $\ov{ S}$ denotes the symmetrised set $  S \cup { S}^{-1}$. A word $w$ over $ S$ (respectively $\ov{ S}$) is a finite sequence of elements in $ S$ (respectively $\ov{ S}$). By an abuse of notation, we sometimes allow words to be infinite.

A word $w$ over $\ov{S}$ is \emph{cyclically reduced} if it is reduced and all its cyclic shifts are reduced. Given a set $ R$ of cyclically reduced words, we denote by $\overline{R}$ the cyclic closure of $ R\cup  R^{-1}$. If $ R = \{w\}$ we sometimes denote $\ov{{R}}$ by $\ov w$. Given a cyclically reduced word $r$, we say that $w$ is a \emph{cyclic subword} of $r$ if $w$ is a subword of a word in $\ov{r}$.

Let $\Cay(G,S)$ be the Cayley graph of $G$ with respect to the generating set $S$.  By a path $p$ in $\Cay(G,S)$, we always mean a combinatorial path.  The label of $p$ is a word in $\ov S$; note that the label of $p$ also represents a word in $G$.

 Given a cycle $K$, we say a 1--Lipschitz map $C\colon K \to X$ is an \textit{embedded cycle}. If $G=\pres$ acts by isometries on a metric space $X$, then we can view each $r\in R$ as an embedded cycle in $X$ in the following way.  The orbit map $G\to X$ sending $g$ to $gx_0$ for some fixed basepoint $x_0\in X$ allows us to view $r=s_1s_2\dots s_k$ as a sequence of points $x_0,s_1x_0, s_1s_2x_0,\dots, s_1s_2\dots s_{k-1}x_0$, which we connect, in order, with geodesics in $X$.  Let $M=\max\{d_X(x_0,sx_0)\mid s\in S\}$, and modify $\cay$ so that each edge has length $M$. Then the orbit map restricted to $r$ is a 1--Lipschitz map.

\subsection{The Morse local-to-global property}\label{sec:mltg}

In this section, we define the Morse local-to-global property, a property introduced in \cite{russellsprianotran:thelocal}.

A function $M\colon \R_{\geq 1}\to \R_{\geq 0}$ is  a \emph{Morse gauge} if it is non-decreasing and continuous.
A quasi-geodesic $\gamma$ is  \emph{$M$-Morse} for some Morse gauge $M$ if every $Q$--quasi-geodesic $\lambda$ with endpoints $\gamma(s)$ and $\gamma(t)$ stays in the closed $M(Q)$--neighbourhood of $\gamma[s, t]$. A quasi-geodesic is  \emph{Morse} if it is $M$--Morse for some Morse gauge $M$.  It is clear from the definition that every subpath of an $M$--Morse quasi-geodesic is again $M$--Morse.  An element $a$ of a finitely generated group $G$ is \textit{Morse} if the subset $\{a^i\}_{i\in \mathbb Z}$ in some (equivalently, any) Cayley graph of $G$ is a Morse quasi-geodesic.

We say a path $p \colon I\to X$ \emph{$L$--locally satisfies a property $(P)$} if for every $s, t \in I$ with $\abs{t - s} \leq L$ the subpath $p[s, t]$ has the property $(P)$. The following property, introduced by Russell, Spriano, and Tran in \cite{russellsprianotran:thelocal}, generalizes the property of Gromov hyperbolic spaces that every local quasi-geodesic is a global quasi-geodesic.  

\begin{defn}[Morse local-to-global]\label{def_morse-local-to-global}
    A metric space $X$ satisfies the \emph{Morse local-to-global} (MLTG) property if the following holds. For any constant $Q\geq 1$ and Morse gauge $M$, there exists a scale $L$, a constant $Q'\geq 1$ and a Morse gauge $M'$ such that every path that is $L$--locally an $M$--Morse $Q$--quasi-geodesic is an $M'$--Morse $Q'$--quasi-geodesic. 
\end{defn}

We say that a MLTG group is \emph{elementary} if its virtually cyclic or does not contain any Morse ray.

We now recall some basic facts about (local) quasi-geodesics in metric spaces.  
\begin{lem}[{\cite[Lemma~2.6]{DrutuSprianoZbinden:weak}}]\label{lem:reverse_inclusion_QG_nbhd}
Let $\gamma_i$, $i=1,2$, be two $C$--quasi-geodesic segments with endpoints at distance $d$. Then for all $\mu \geq d$ there exists $\mu'$ such that if $\gamma_1 \subseteq \mc {N}_\mu(\gamma_2)$  then $\gamma_2\subseteq \mc {N}_{\mu'}(\gamma_1)$.
\end{lem}

The following is a well-known fact about Morse quasi-geodesic and follows from \cite[Lemma~2.8 ix)]{Z:manifold}.  We note that the result in \cite{Z:manifold} is stated for quasi-geodesic triangles, but the proof generalizes to quasi-geodesic quadrangles by dividing the quadrangle into two triangles.

\begin{lem}\label{ref:triangles}
    Let $X$ be a geodesic metric space, let $M$ be a Morse gauge, and let $C\geq 1$ be a constant. There exists a Morse gauge $M''$ such that the following holds for all $C$--quasi-geodesic quadrangles $\Delta = (\gamma_1, \gamma_2, \gamma_3, \gamma_4)$, that is, for all quadruples of $C$--quasi-geodesics $\gamma_1, \ldots \gamma_4$ such that $\gamma_i^+ = \gamma_{i+1}^-$ for all $1\leq i \leq 4$ (here we define $\gamma_5^- = \gamma_1^-$). If $\gamma_1, \gamma_2$ and $\gamma_3$ are $M$--Morse, then $\gamma_4$ is $M'$--Morse.
\end{lem}

For geodesic quadrangles, the following lemma follows from \cite[Lemma~2.2]{C:Morse}. A quadrangle is \textit{$\delta$--slim} if each point on one of the sides is contained in the $\delta$--neighborhood of one of the other three sides.
\begin{lem}\label{ref:delta}
    Let $X$ be a geodesic metric space, let $M$ be a Morse gauge, and let $C\geq 1$ be a constant.  There exists $\delta\geq 0$ such that any $C$--quasi-geodesic quadrangle $\Delta=(\gamma_1,\gamma_2,\gamma_3,\gamma_4)$ such that each $\gamma_i$ is $M$--Morse is $\delta$--slim.
\end{lem}

\begin{proof}
    We show the result holds for triangles with some constant $\delta'$, and then by dividing $\Delta$ into two triangles we obtain the result with $\delta=2\delta'$.  

    Let $(\alpha_1,\alpha_2,\alpha_3)$ be a $C$--quasi-geodesic triangle with $M$--Morse sides, and let $\beta_i$ be a geodesic with the same endpoints as $\alpha_i$.  Then $\beta_i$ is contained in the $M$--neighborhood of $\alpha_i$ and hence is $M'$--Morse for some $M'$ depending only on $M$ and $C$ \cite[Lemma~2.8 x]{Z:manifold}.  By \cite[Lemma~2.2]{C:Morse}, there is a constant $\delta''$ depending only on $M'$ such that the geodesic triangle $(\beta_1,\beta_2,\beta_3)$ is $\delta''$--slim.  By Lemma~\ref{lem:reverse_inclusion_QG_nbhd}, there is some $\mu$ depending on $C$ and $M$ such that each $\beta_i$ is contained in the $\mu$--neighborhood of $\alpha_i$.  It follows that the quasi-geodesic triangle $(\alpha_1,\alpha_2,\alpha_3)$ is $(\delta + 2\max\{\mu,M\})$--slim. 
\end{proof}

The next two lemmas give critera for when the concatenations of certian quasi-geodesics is a quasi-geodesic.  The first is well-known; a similar version appears in, for example, \cite{qingrafi:quasiredirecting}.

\begin{lem}\label{lem:Concat_with_npp}
    Let $p,q\in X$, and let $\gamma\colon [0,T]\to X$ be a $(K,C)$--quasi-geodesic.  Let $t,t'\in [0,T]$ be such that $\gamma(t)$ and $\gamma(t')$ are points on $\gamma$ closest to $p$ and $q$, respectively.  Let $\alpha$ be a geodesic from $p$ to $\gamma(t)$ and $\beta$ a geodesic from $\gamma(t')$ to $q$.  Then 
    \begin{enumerate}
        \item the path $\alpha * \gamma[t,T]$ is a $(3K,C)$--quasi-geodesic; and
        \item if $\abs{t-t'}\geq 3K(d(p,\gamma(t)) + d(q,\gamma(t'))$, then $\alpha * \gamma[t,t'] * \beta$ is a $(3K,C)$--quasi-geodesic. 
    \end{enumerate}
\end{lem}

The next lemma states that local quasi-geodesics contained in a uniform neighbourhood of a geodesic are, in fact, global quasi-geodesics.
\begin{lem}[{\cite[Lemma~2.14]{russellsprianotran:thelocal}}]\label{lemma:close-to-geodesic-implies-quasi-geodesic}
Let $\gamma\colon I \to X $ be an $L$--locally $Q$--quasi-geodesic and $C \geq 0$. Let $s, t\in I$. If $L > Q(3C+Q+2)$ and $\gamma[s,t]$ is contained in the $C$--neighbourhood of a geodesic from $\gamma(s)$ to $\gamma(t)$, then $\gamma[s, t]$ is a $Q'$--quasi-geodesic where $Q'$ depends only on $Q$ and $C$. 
\end{lem}

Further, a quasi-geodesic that stays close to a local Morse quasi-geodesic is itself locally Morse. 
\begin{lem}[{\cite[Lemma~2.12]{HeSprianoZbinden:sigma-compact}}]\label{lem:close-to-local-implies-local}
    Let $M$ be a Morse gauge and let $Q\geq 1$ be a constant. There exists a Morse gauge $M'$ such that the following holds. Let $L'\geq 0$ be a scale. There exists a scale $L$ such that any $Q$--quasi-geodesic in the $Q$--neighbourhood of an $L$--locally $M$--Morse $Q$--quasi-geodesic is $L'$--locally $M'$--Morse.
\end{lem}

He, Spriano, and the second author introduced the following version of the MLTG property \cite{HeSprianoZbinden:sigma-compact}.
\begin{defn}[Geodesic MLTG]\label{def:geodesicMLTG}
A geodesic metric space $X$ satisfies the \emph{geodesic MLTG} property if for every Morse gauge $M$, there exists a constant $L$ and a Morse gauge $M'$ such that every geodesic which is $L$--locally $M$--Morse is $M'$--Morse.
\end{defn}

\subsection{The Morse boundary}

Defined by Cordes \cite{C:Morse}, the \emph{Morse boundary} $\partial_*X$ of a geodesic metric space $X$ is the set of all Morse geodesic rays, up to bounded Hausdorff distance.  A discussion of the topology on $\partial_*X$ can be found in \cite{C:Morse}. The Morse boundary is a quasi-isometry invariant, and thus we can define the Morse boundary  of a finitely generated group $G$ to be $\partial_*G=\partial_* \Cay(G,S)$ for some (equivalently, any) finite generating set $S$ of $G$.  

A topological space $Y$ is \textit{$\sigma$--compact} if it is the union of countably many compact subsets. In the case of the Morse boundary of a proper geodesic metric space, this is equivalent to the following definition, by {\cite[Lemma~2.7]{HeSprianoZbinden:sigma-compact}}.
\begin{defn}\label{lem:sigma_cpt}
    The Morse boundary $\partial_*X$ of a proper geodesic metric space $X$ is \textit{$\sigma$--compact} if there exists an increasing sequence $(M_n)_{n\in \mathbb N}$ of Morse gauges such that every Morse geodesic ray is $M_n$--Morse for some $n$ possibly depending on the ray. 
\end{defn}

\begin{lem}\label{lem:sigma_cpt2}
     Let $G = \pres$ be a finitely generated group with non-empty Morse boundary and let $q\geq 1$,  $Q\geq 0$ be constants. The Morse boundary $\partial_*G$ of $G$ is $\sigma$--compact if and only if there exists an increasing sequence $(M_n)_{n}$ of Morse gauges such that any Morse $(q, Q)$--quasi-geodesic is $M_n$--Morse for some $n$ possibly depending on the quasi-geodesic.
\end{lem}
\begin{proof}
    The only if part is immediate. We now prove the if part. Assume $\mb G$ is $\sigma$-compact. That is, there exists a sequence $(M_n)_n$ of Morse gauges such that every Morse geodesic ray is $M_n$--Morse for some $n$. Since there are only countable many pairs of vertices, we may assume (by potentially increasing the $M_n$) that any geodesic segment between a pair of vertices is $M_n$--Morse for some $n$.
    
    By \cite[Lemma~vi),~ix)~and~x)]{Z:manifold} there exists a sequence $(M_n')_n$ of Morse gauges such that if $\gamma$ is an $M_n$--Morse geodesic, and $\gamma'$ is a $(q, Q)$--quasi-geodesic whose endpoints coincide with the endpoints of $\gamma$ if they are in the boundary, or have distance at most 1 from the endpoints of $\gamma$ if they are in $\cay$, then $\gamma'$ is $M_n'$--Morse. 

    Let $\gamma$ be a Morse $(q, Q)$--quasi-geodesic and let $\gamma'$ be a geodesic such that each endpoint is either a vertex of $\cay$, in which case it has distance at most $1$ form the corresponding endpoint of $\gamma$, or the endpoint is in the boundary, in which case is coincides with the corresponding endpoint of $\gamma$. If $\gamma$ is Morse, then by \cite[Lemma~vi),~ix)~and~x)]{Z:manifold}, so is $\gamma'$. By the definition of $(M_n')_n$, we have that $\gamma'$ is $M_n$--Morse for some $n$, which in turn implies that $\gamma$ is $M_n'$--Morse, concluding the proof.
\end{proof}

The statement of Lemma~\ref{lem:sigma_cpt2} restricted to geodesics states that $\partial_* X$ is $\sigma$--compact if and only if there exists an increasing sequence $(M_n)_{n}$ of Morse gauges such that any Morse geodesic is $M_n$--Morse for some $n$.  Given two $M$--Morse geodesics, one of which is $M_n$--Morse and one of which is $M_{n'}$--Morse, by taking the maximum of the two Morse gauges, we see that they are both (without loss of generality) $M_n$--Morse.  However, given an infinite sequence of $M$--Morse geodesics, this process does not work, and, a priori, there may not be an $n$ such they are \textit{all} $M_n$--Morse. There is also the stronger notion of strong $\sigma$--compactness of the Morse boundary, which is due to He, Spriano, and the second author\cite{HeSprianoZbinden:sigma-compact}, which requires there to be such an $n$.

\begin{defn}\label{def:stronglysigmacpt}
    A metric space $X$ has \textit{strongly $\sigma$--compact} Morse boundary if there exists a sequence of Morse gauges $(M_n)_{n\in \mathbb N}$ such that for every Morse gauge $M$ there exists $n=n(M)$ such that all $M$--Morse geodesic rays are $M_n$--Morse.
\end{defn}

If $X$ is the Cayley graph of a group, then having strongly $\sigma$-compact Morse boundary also allows us to control the Morseness of geodesic segments as opposed to geodesic rays.  

\begin{lem}\label{lem:strongly-sigma-compact}
    Let $G = \pres$ be a finitely generated group with non-empty Morse boundary. The Morse boundary $\mb G$ is strongly $\sigma$-compact if and only if there exists a sequence of Morse gauges $(M_n)_n$ such that for every Morse gauge $M$, there exists $n=N(M)$ such that all $M$--Morse geodesics segments are $M_n$--Morse.
\end{lem}
\begin{proof}
    $\implies\colon$ Let $X = \cay$. Assume that $\mb G$ is strongly $\sigma$-compact. That is, there exists a sequence $(M_n)_n$ of Morse gauges such that for all $M$, there exists $n=n(M)$ such that $\partial_{x_0}^{M} X\subseteq \partial_{x_0}^{M_n} X$ for all choices of basepoints $x_0\in X$. Next we use the following, see \cite[Lemma~2.28]{Z:free_product}: every finite segment $\gamma$ starting at a basepoint $x_0$ has an associated geodesic ray $\lambda_\gamma$ starting at $x_0$ such that the following holds: if $\lambda_\gamma$ or $\gamma$ are $M$--Morse, then the other is $M'$--Morse for a Morse gauge $\Phi(M)$ only depending on $M$. We show that the second part of the equivalence holds for the sequence $(\Phi(M_n))_n$. Let $M$ be a More gauge and let $\gamma$ be an $M$--Morse geodesic. By translating $\gamma$ we may assume that $\gamma$ starts at $x_0$. Hence the associated geodesic ray $\lambda_\gamma$ is $\Phi(M)$--Morse. Since $\mb G$ is strongly $\sigma$-compact, there exists $n = n(\Phi(M))$ such that $\lambda_\gamma$ is $M_n$--Morse. The latter implies that $\gamma$ is $\Phi(M_n)$--Morse.

    \smallskip

    $\Longleftarrow\colon$ Let $(M_n)_n$ be a sequence of Morse gauges satisfying the right hand side of the equivalence and let $M$ be a Morse gauge. There exists an integer $n=n(M)$ such that $M$--Morse all finite geodesic segments are $M_n$--Morse. In particular, if $\gamma$ is an $M$--Morse geodesic ray starting at some basepoint $x_0$, all its finite subsegments are $M$--Morse and hence $M_n$--Morse. This implies that $\gamma$ is $M_n$--Morse. Since this works for any $\gamma$ $\partial_{x_0}^{M} X\subseteq \partial_{x_0}^{M_n} X$, implying that $\mb G$ is strongly $\sigma$-compact.
\end{proof}

\section{Circle of equivalences of MLTG}\label{sec:circle}

The goal of this section is to prove Theorem~\ref{thm:main1}. We prove \eqref{item:StrSigmaCpt}$\implies$\eqref{item:IPSC} in Section~\ref{sec:notIPSC}, \eqref{item:IPSC}$\implies$\eqref{item:GeodMLTG} in Section~\ref{sec:GeodMLTG}, and \eqref{item:GeodMLTG}$\implies$\eqref{item:MLTG} in Section~\ref{sec:MLTG}. The implication \eqref{item:MLTG}$\implies$\eqref{item:RegularLang} is due to \cite[Theorem~D]{CordesRussellSprianoZalloum:regularity} and \eqref{item:RegularLang}$\implies$\eqref{item:StrSigmaCpt} follows from work of the second author with He and Spriano: in the proof \cite[Theorem~A]{HeSprianoZbinden:sigma-compact}, they only use \eqref{item:RegularLang} to conclude strong $\sigma$-compactness of the Morse boundary. Finally, we prove \eqref{item:SigmaCpt}$\implies$\eqref{item:StrSigmaCpt} in Section~\ref{sec:CptToStrCpt}, and  \eqref{item:StrSigmaCpt}$\implies$\eqref{item:SigmaCpt} follows by definition. This series of implications is also depicted in Figure~\ref{fig:known-and-new-equivalences}.

Several of the equivalences of Theorem~\ref{thm:main1} were already known in the special case of small cancellation groups.  In particular, He, Spriano, and the second author showed the equivalence of \eqref{item:SigmaCpt}, \eqref{item:StrSigmaCpt}, and \eqref{item:MLTG} in the case of a $C'(1/9)$ small cancellation group \cite[Theorem~C]{HeSprianoZbinden:sigma-compact}. Further, Dru\c tu, Spriano, and the second author showed the equivalence \eqref{item:MLTG} and \eqref{item:SigmaCpt} when the group admits a bounded quasi-geodesic bicombing \cite{DrutuSprianoZbinden:weak}. In this section, we show how to generalize common techniques from small cancellation to arbitrary finitely generated groups.  In Section~\ref{sec:IPSC}, we define a \emph{projective geodesic}, which can be thought of as an analogue of a subsegment of a relator.  We use projective geodesics to introduce the general \emph{increasing partial small cancellation condition (IPSC)}, which was originally defined in the context of $C'(1/6)$ small cancellation groups using relators \cite{Z:free_product} . For arbitrary finitely generated groups, we define IPSC using projective geodesics.

\subsection{Definition of general IPSC}\label{sec:IPSC}

The \emph{increasing partial small-cancellation condition (IPSC)} was introduced by the second author in \cite{Z:free_product} for $C'(1/6)$ small cancellation groups $G=\langle S\mid R\rangle$.  Roughly speaking, given a sufficiently nice function $f$, IPSC quantifies the property of having sufficiently long subwords $w$ of longer and longer relators such that every piece of a relator $r\in R$ that is a subword of $w$ has length at most $\abs{r}/f(\abs{r})$.  

In this section, we generalize the IPSC condition to arbitrary finitely generated groups that are not necessarily small cancellation groups. To do so, we first generalize the notion of a subsegment of a relator; see Figure~\ref{fig:loopvsrelator}. 

\begin{defn}[Projective geodesic]\label{def:ProjQuadrangle}
    A geodesic $\gamma$ is \textit{$r$--projective} if there exist two points $x, y\in X$ such that 
    \begin{enumerate}
        \item $r = d(x, \gamma^+) + d(x, y) + d (y, \gamma^-) + \abs{\gamma}$;
        \item $d(\gamma^+, x) \geq d(x, y)$;
        \item $\gamma^+$ is a closest point projection of $x$ onto $\gamma$; and 
        \item $\gamma^-$ is a closest point projection of $y$ onto $\gamma$.
    \end{enumerate}
    We call the triple $(x, y, \gamma)$ a \emph{projective loop} and $r$  the \emph{loop length}.
\end{defn}

\begin{figure}
    \centering
    \def\svgwidth{3.5in}
    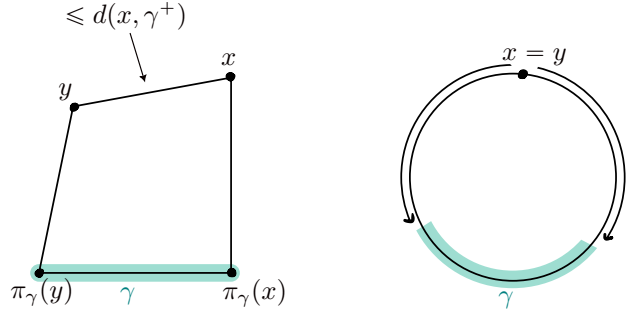
    \caption{Left: a projective loop.  Right: A short subsegment $\gamma$ of a relator in a small cancellation group satisfies the definition of a projective geodesic by choosing $x=y$ to be an antipodal point on the relator.}
    \label{fig:loopvsrelator}
\end{figure}

In small cancellation groups, an intersection function of a geodesic roughly measures the length of the longest intersection with a relator of length at most $t$.  Here, we generalize this  notion to arbitrary finitely generated groups using projective geodesics.

\begin{defn}[Intersection function]\label{def:IntFcn}
    The \textit{intersection function} $\rho_\gamma \colon \N\to \R_{\geq 0}$ of a geodesic $\gamma$ is defined as follows:
    \begin{align*}
        \rho_\gamma(t) = \max\{ \abs{\gamma'} \mid \text{$\gamma'$ is an $r$-projective subpath of $\gamma$ for some $r\leq t$} \}.
    \end{align*}
Given a function $\rho$, a geodesic $\gamma$ is \textit{$\rho$--intersecting} if its intersection function $\rho_\gamma$ is at most $\rho$.
\end{defn}

Arzhantseva, Cashen, Gruber, and Hume showed that a geodesic in a small cancellation group is Morse if and only if its intersection function (in the small cancellation sense) is sublinear \cite[Corollary~4.14]{arzhantseva2019negative}.  We will show the same holds for a general finitely generated group with this more general notion of intersection function.  We will use the following result, which gives a connection between Morse and contracting geodesics.

\begin{lem}[{\cite[Theorem~1.4]{ACHG:contraction_morse_divergence}}]\label{lem:sublinearandMorse}  
Let $Y$ be a subspace of a geodesic metric space $X$.  The following are equivalent.
\begin{enumerate}
    \item There exists a Morse gauge $M$ such that $Y$ is $M$--Morse.
    \item There exists a non-decreasing sublinear function $\rho$ such that $Y$ is $(r,\rho)$--contracting.
\end{enumerate}
\end{lem}
We are now ready to show that a geodesic is Morse if and only if its intersection function is sublinear.
\begin{lem}[Sublinearly intersecting versus Morse]\label{lem:sublinvsMorse}
    A geodesic is Morse if and only if its intersection function is sublinear. Moreover, this equivalence is quantitative in the following sense. 
    \begin{itemize}
        \item For every Morse gauge $M$, there exists a sublinear function $\rho$ such that all $M$--Morse geodesics are $\rho$--intersecting.
        \item For every (non-decreasing) sublinear function $\rho$, there exists a Morse gauge $M$ such that all $\rho$--intersecting geodesics are $M$--Morse.
    \end{itemize}
\end{lem}

\begin{proof}
    For the first bullet point, let $M$ be a Morse gauge, and let $\gamma$ be an $M$--Morse geodesic. Recall that all subsegments of $\gamma$ are $M$--Morse.  By Lemma~\ref{lem:sublinearandMorse}, there exists a non-decreasing sublinear function $\rho$ such that every $M$--Morse geodesic is $(r,\rho)$--contracting. Let $(x, y, \gamma')$ be an $\ell$--projective loop such that $\gamma'$ is a subsegment of $\gamma$. Since $\gamma'$ is $(r,\rho)$--contracting, we have that $\abs{\gamma'}\leq \rho(d(x, \gamma'^+)) \leq \rho(\ell)$. Hence, the intersection function $\rho_{\gamma}$ of $\gamma$ satisfies $\rho_\gamma\leq \rho$.

    For the second bullet point, let $\rho$ be a (non-decreasing) sublinear function, and let $\gamma$ be a geodesic whose intersection function $\rho_{\gamma}$ satisfies $\rho_{\gamma}\leq \rho$. Let $x\in X$, let $y\in B_{d(x, \gamma)}(x)$, and let $p_x, p_y$ be closest point projections of $x$ and $y$ onto $\gamma$. Then $(x, y, [p_y, p_x]_{\gamma})$ is a projective loop. Our choice of $y$ ensures that $d(x, y)\leq d(x, [p_y, p_x]_{\gamma})$, and hence $d(y, p_y) = d(y, [p_y, p_x]_{\gamma})\leq 2d(x, p_x)$. Hence, by the triangle inequality, the loop length of $(x, y, [p_y, p_x]_{\gamma})$ is at most $8d(x, p_x)$. Consequently $d(p_x, p_y)\leq \rho(8d(x, p_x))$. This inequality holds for any point $y\in B_{d(x, \gamma)}(x)$, so by the triangle inequality, $d(p_z, p_w)\leq 2\rho(8d(x, p_x))$ for all points $z, w\in B_{d(x, \gamma)}(x)$ and their closest point projections $p_z, p_w$ onto $\gamma$. This shows that $\gamma$ is $(r, \rho')$--contracting for the function $\rho'$ defined via $\rho'(x) = 2\rho(8x)$. Thus, Lemma~\ref{lem:sublinearandMorse} yields a Morse gauge $M$ depending only on $\rho$ such that $\gamma$ is $M$--Morse.
\end{proof}

We are now ready to define the general increasing partial small cancellation condition.  The main difference between this and the original definition of IPSC is that here we use $\ell$--projective geodesics instead of relators.  In this paper, we will always be considering this condition in the context of a finitely generated group, and hence we continue to simply call this condition IPSC if no confusion is possible.
\begin{defn}[General IPSC]\label{def:generalIPSC}
     A finitely generated group $G = \pres$ satisfies the \emph{general increasing partial small-cancellation condition (IPSC)} if for every sequence $(n_i)_{i\in \N}$ of positive integers, there exists a non-decreasing sublinear function $\rho$ such that the following holds. For all $K\geq 0$, there exists $i\geq K$ and an $\ell$-projective geodesic $\gamma$ such that 
    \begin{enumerate}[(\roman*)]
        \item $\ell\geq n_i$,
        \item $\abs{\gamma}\geq \ell/i$, and
        \item the geodesic $\gamma$ is $\rho$--intersecting.
    \end{enumerate}
\end{defn}

\subsection{Strongly $\sigma$--compact Morse boundary implies not IPSC}\label{sec:notIPSC}
In this subsection, we prove Theorem~\ref{thm:main1}\eqref{item:StrSigmaCpt}$\implies$\eqref{item:IPSC}: no group with strongly $\sigma$-compact Morse boundary satisfies the IPSC condition; see Theorem~\ref{thm:StrSigmaCpt=>noIPSC}. 

We first use the results from Section~\ref{sec:IPSC} to relate having strongly $\sigma$-compact Morse boundary and having countably many sublinear functions that capture the intersection function of all geodesic segments. 

\begin{lem}\label{lem:CountIntFcns}
    Let $G=\pres$ be a finitely generated group whose Morse boundary is strongly $\sigma$--compact.  Then there exists a countable collection $(\rho_n)_{n}$ of non-decreasing sublinear functions with $\rho_{n-1}\leq \rho_n$ such that the following holds. For every non-decreasing sublinear function $\rho$, there exists $n=n(\rho)$ such that all $\rho$--intersecting geodesic segments in $\cay$ are $\rho_n$--intersecting.
\end{lem}

\begin{proof}
    By Lemma~\ref{lem:strongly-sigma-compact}, there exists a sequence of Morse gauges $(M_n)_n$ such that for every Morse gauge $M$, there exists $n=n(M)$ such that any $M$--Morse geodesic segment in $\cay$ is $M_n$--Morse. By Lemma~\ref{lem:sublinvsMorse} there exists a sequence of non-decreasing sublinear functions $(\rho_n)_n$ such that any $M_n$--Morse geodesic is $\rho_n$--intersecting. By potentially replacing $\rho_n$ with the point-wise maximum of $\rho_0, \ldots, \rho_n$ we can assume that $\rho_{n-1}\leq \rho_{n}$ for all $n$.

    Let $\rho$ be a non-decreasing sublinear function. By Lemma~\ref{lem:sublinvsMorse}, there exists a Morse gauge $M$ such that all $\rho$--intersecting geodesics are $M$--Morse. Hence, choosing $n(\rho) = n(M)$, we obtain that all $\rho$--intersecting geodesics segments in $\cay$ are $M_n$--Morse and hence $\rho_n$--intersecting.
\end{proof}

We are now ready to show the main result of this subsection. 

\begin{thm}\label{thm:StrSigmaCpt=>noIPSC}
    Let $G$ be a finitely generated group. If $\mb G$ is strongly $\sigma$--compact, then  $G$ does not satisfy the general IPSC condition.
\end{thm}

\begin{proof}
    Suppose towards a contradiction that $\mb G$ is strongly $\sigma$-compact and satisfies the general IPSC condition. By Lemma~\ref{lem:CountIntFcns}, there are countably many intersection functions $(\rho_i)_{i}$ with $\rho_{i-1}\leq \rho_{i}$ such that for every intersection function $\rho$ there exists $i$ such that every $\rho$--intersecting geodesic is $\rho_i$--intersecting.
    
    We will construct an intersection function $\rho$, a sequence of $\rho$--intersecting geodesics $(\gamma_i)_i$ and an increasing sequence of integers $(k_i)_i$ such that for each $i$, the geodesic $\gamma_{i}$ is not $\rho_{k_i}$ intersecting. Since  $\rho_i\geq \rho_{i-1}$ for all $i$, this contradicts the existence of an integer $j$ such that all $\rho$--intersecting geodesic segments are $\rho_j$--intersecting. 
    
    For each $i\geq 1$, choose $n_i\geq i$ such that $\rho_i(t)< t/i$ for all $t\geq n_i$; such an integer always exists since $\rho_i$ is sublinear. Since $\pres$ satisfies the general IPSC condition, there exist a non-decreasing and sublinear function $\rho$, a sequence of $\ell_i$--projective geodesics $\gamma_i$, and a sequence of indices $k_i\geq i$ such that: 
    
    \begin{enumerate}
        \item $\ell_i\geq n_{k_i}$,
        \item $\abs{\gamma_i}\geq \ell_i/k_i$, and
        \item the geodesic $\gamma_i$ is $\rho$--intersecting.
    \end{enumerate}

    By construction of the sequence $(n_i)_{i}$, we have that $\rho_i(t)<t/i$ for all $t\geq n_i$. Moreover, for all $i$, the geodesic segment $\gamma_i$ is $\ell_i$--projective for some $\ell_i\geq n_{k_i}$. Hence $\rho_{\gamma_i}(\ell_i) = \abs{\gamma_i}\geq \ell_i/k_i > \rho_{k_i}(\ell_i)$, concluding the proof.
\end{proof}

\subsection{Not IPSC implies geodesic MLTG}\label{sec:GeodMLTG}
The main result of this section is Theorem~\ref{thm:main1}\eqref{item:IPSC}$\implies$\eqref{item:GeodMLTG}: groups that do not satisfy the IPSC condition satisfy the geodesic MLTG property. We follow the proof of \cite[Proposition~3.2]{HeSprianoZbinden:sigma-compact}, where He, Spriano and the second author prove the implication in the case of small cancellation groups.

While most of the proof works for projective geodesics instead of subsegments of relators, there is one crucial difference: while a subsegment of a subsegment of a relator of length $r$ is clearly a subsegment of a relator of length $r$, a subsegment of an $\ell$--projective geodesic is not necessarily $\ell$--projective.  To overcome this difficulty, we use following auxiliary lemmas, which allow us to find ``good'' subsegments of projective geodesics.

\begin{lem}\label{lem:making-shorter-loops}
    Let $X$ be a geodesic metric space, and let $(x, y, \gamma)$ be an $\ell$--projective loop in $X$. There exists an $\ell'$--projective loop $(x', y', \gamma')$ for some $\ell'\leq \ell$ such that $\gamma'$ is a subgeodesic of $\gamma$ with $\abs{\gamma}/4\leq \abs{\gamma'}\leq 3\abs{\gamma}/4$.
\end{lem}

\begin{proof}
    Let $z\neq y$ be a point on $[x, y]$, and let $p_z$ be a closest point projection of $z$ onto $\gamma$. Observe that $d(z,\gamma)\geq d(z,y)$, for if not, then $d(x,y)=d(x,z)+d(z,y)>d(x,z)+d(z,\gamma) \geq d(x,\gamma)$, which is a contradiction.  Thus both $(x, z, [p_z, \gamma^+]_{\gamma})$ and $(z, y, [\gamma^-, p_z]_{\gamma})$ are projective loops whose loop lengths $L_x$ and $L_y$ satisfy $L_x \leq  \ell-d(\gamma^-,p_z)$ and $L_y \leq \ell - d(p_z,\gamma^+)$. 

    At least one of $d(\gamma^-, p_z)$ and $d(p_z, \gamma^+)$ is larger or equal to $\abs{\gamma}/2$. If it is also smaller than $3\abs{\gamma}/4$, we have found our desired loop $(x', y', \gamma')$ and can conclude the proof. If it is larger or equal to $3\abs{\gamma}/4$, then $p_z$ is in the (closed) $\abs{\gamma}/4$ neighborhood of $\gamma^+$ or $\gamma^-$. 

    It remains to consider the case that for each point $z\in [x, y]$, any closest point projection $p_z$ onto $\gamma$ is in the (closed) $ \abs{\gamma}/4$ neighborhood of one of $\gamma^-$ or $\gamma^+$. By continuity, there has to exist a point $z\in [x, y]$ and closest point projections $p_1, p_2$ of $z$ onto $\gamma$ such that $p_1, p_2$ are points on $\gamma$ in the $\abs{\gamma}/4$--neighborhood of $\gamma^-$ and $\gamma^+$, respectively. Let $m$ be the midpoint of $\gamma$ and consider the loop $\mathcal{L} = (z, p_1, [m ,p_2]_{\gamma})$. By the triangle inequality, the loop length of $\mathcal{L}$ is at most $\ell$. Hence choosing $(x', y', \gamma') = \mc{L}$ concludes the proof.
\end{proof}

\begin{lem}\label{lem:making-custom-loops}
    Let $X$ be a geodesic metric space, let $(x, y, \gamma)$ be an $\ell$--projective loop, and let $C\colon \R_{\geq_0} \to \R_{\geq 0}$ be a non-decreasing function with $C(0) > 0$ such that $0 < C(\ell)\leq \abs{\gamma}$. There exists an $\ell'$--projective loop $(x', y', \gamma')$ for some $\ell'\leq \ell$ such that $\gamma'$ is a subgeodesic of $\gamma$ of length $C(\ell')\leq \abs{\gamma'}\leq 4C(\ell')$.
\end{lem}
\begin{proof}
    If $\abs{\gamma}\leq 4C(\ell)$, then we can take $(x', y', \gamma'):=(x, y, \gamma)$. Otherwise,  use Lemma~\ref{lem:making-shorter-loops} to replace the $\ell$--projective loop $(x, y, \gamma)$ by an $\ell'$--projective loop $(x', y', \gamma')$ with $\ell'\leq \ell$. It follows that $C(\ell')\leq C(\ell) \leq \abs{\gamma'}$. Repeat the above two steps until the process stops, which must occur after a finite number of steps since $C(0)> 0$ and $\abs{\gamma'}\leq 3\abs{\gamma}/4$.
\end{proof}

We are now ready to prove the main result of this section.

\begin{prop}\label{lem:sigma_cpct_implies_geosdesicMLTG}
	Let $G = \pres$ be a finitely generated group. If $\pres$ does not satisfy the IPSC condition, then $X = \cay$ satisfies the geodesic MLTG property. 
\end{prop}
\begin{proof}

We will prove the contrapositive.  Assume that $G = \pres$ does not satisfy the geodesic MLTG property, so that there exists a Morse gauge $M_0$ such that for all Morse gauges $M\geq M_0$ and all scales $L>0$, there exists a geodesic that is $L$--locally $M_0$--Morse but not globally $M$--Morse. Using the relationship between Morse and sublinearly intersecting from Lemma~\ref{lem:sublinvsMorse}, this is equivalent to the existence of a sublinear and non-decreasing function $\rho_0$ such that, for all intersection functions $\rho\geq \rho_0$ and scales $L>0$, there exists a geodesic $\gamma$ which is $L$--locally $\rho_0$--intersecting but not globally $\rho$--intersecting. In particular, for some $\ell$, there has to exist an $\ell$--projective loop $(x', y', \gamma')$ such that $\gamma'$ is $L$--locally $\rho_0$--intersecting but $\abs{\gamma'} > \rho(\ell)$. Let $\mc{L}_{L, \rho}$ be such an $\ell$--projective loop that minimizes $\lceil \ell\rceil$.

We are now ready to prove that $G$ satisfies the general IPSC condition. Let $(n_i)_{i}$ be a sequence of integers. We will construct a sublinear function $g$, sequences of integers $k_i\geq k_{i-1}+1$ and $\ell_i\geq n_{k_i}$ and a sequence of $\ell_i$--projective geodesics $\gamma_i$ such that for all $i$,
\begin{enumerate}[(I)]
    \item $\abs{\gamma_i}\geq \ell_i/k_i$,\label{1}
    \item $\gamma_i$ is $g$--intersecting. \label{3}
\end{enumerate}

Since $k_i$ tends to $\infty$, the existence of such a sequence of geodesics will imply that $G$ satisfies the IPSC condition and hence conclude the proof. The fact that we do not require $g$ to be non-decreasing is not a problem: if \eqref{3} holds for $g$ it also holds for the sublinear and non-decreasing function $g'$ defined via $g'(t) = \min\{t, \max_{t'\leq t} \{g(t')\}\}$.

Define a function $\rho' \colon \N\to \R_{\geq 0}$ by
\begin{align*}
    \rho'(t) &= 1+\max\left\{\frac{t}{j} ,\rho_0 (t), \max_{0\leq t'<t}\{\rho'(t')\}\right\} \quad \text{for $n_j\leq t <n_{j+1}$ and $j\geq 1$, and}\\
    \rho'(t) &= \max\{\rho_0(t), t\} \quad \text{for $t < n_1$.}
\end{align*}
Note that $\rho'$ is sublinear and non-decreasing. Increasing terms in the sequence $(n_i)_i$ makes it harder to satisfy $\ell_i\geq n_{k_i}$ (and all other conditions stay the same),  so we may assume $n_{i+1}\geq n_i+1$ and, by possibly increasing $n_{i+1}$ further, 
\begin{align}\label{eq_niassumption}
    \frac{t}{4\rho'(t)} \geq  \frac{t'}{\rho'(t')}+1
\end{align}
for all $t'\leq n_i < n_{i+1}\leq t$. 

We now inductively construct the sequence of integers $k_i, \ell_i$ and $\ell_i$--projective geodesics $\lambda_i$. To start the induction, set  $k_{0} = 6$, and let $\lambda_0$ be a trivial geodesic so that $\abs{\lambda_0}= 0$. For $i \geq 1$, define
\[
L_i = n_{\max\{\abs{\lambda_{i-1}}, k_{i-1}+2\}}.
\]
Let $\ell_i'$ be such that $\mc{L}_{L_i, \rho'} = (\lambda_i', x', y')$ is an $\ell_i'$--projective loop. The geodesic $\lambda_i'$ is $L_i$--locally $\rho_0$--intersecting but not $\rho'$--intersecting. By minimality of $\ceil{\ell_i'}$, we must have  $\abs{\lambda_i'} > \rho'(\ceil{\ell_i'})$. By Lemma~\ref{lem:making-custom-loops}, there exists an $\ell_i$--projective loop $\lambda_i$ with $\ell_i\leq \ell_i'$ and
\begin{align}\label{eq:rho-sandwich}
    \rho'(\lceil\ell_i\rceil)\leq \abs{\lambda_i} \leq 4\rho'(\lceil \ell_i\rceil).
\end{align}

 Define $k_i$ such that $n_{k_i}\leq \ell_i < n_{k_{i+1}}$. Note that $\abs{\lambda_i}\geq \rho'(\lceil\ell_i\rceil)\geq \ell_i/k_i$. Moreover, $\lambda_i$ is not $\rho_0$--intersecting, since $\rho_0<\rho'$, but because $\lambda_i$ is a subsegment of $\lambda_i'$, it is $L_i$--locally $\rho_0$--intersecting.  Thus $\abs{\lambda_i'} > L_i$, which implies that
\begin{align}\label{eqn:bound_x_i} 
 \ell_i \geq \abs{\lambda_i} > L_i = n_{\max\{\abs{\lambda_{i-1}}, k_{i-1}+2\}} \geq \abs{\lambda_{i-1}},
\end{align}
and hence
\begin{equation}\label{eqn:lowerbound_k_i}
    k_i\geq\max\{\abs{\lambda_{i-1}}, k_{i-1}+2\}\geq k_{i-1}+2.
\end{equation}
With this, for all $i$, \eqref{1} is satisfied, and we have $\ell_i\geq n_{k_{i}}$. 

It remains to construct a sublinear function $g$ such that \eqref{3} is satisfied. To do so, we first define $m_i := \ceil{\ell_i}/\abs{\lambda_i}$. Using \eqref{eq_niassumption} and \eqref{eq:rho-sandwich}, we obtain 
\begin{align}\label{eq:m-i}
    m_i = \frac{\ceil{\ell_i}}{\abs{\lambda_i}} \geq \frac{\ceil{\ell_i}}{4\rho'(\ceil{\ell_i})} \geq 1+\frac{\ceil{\ell_{i+1}}}{\rho'(\ceil{\ell_{i+1}})}\geq \frac{\ceil{\ell_{i+1}}}{\abs{\lambda_{i-1}}} =m_{i-1} +1.
\end{align}
Next, for all $i$, we define the functions $g_i\colon\N\to \R_{\geq 0}$ by
\begin{align*}
        g_i(t) = \begin{cases}
            t &\text{if $t\leq \ceil{\ell_i}$}\\
            \frac{t}{m_i} &\text{otherwise.}
        \end{cases}
    \end{align*}
    Further, define 
    \[
    g_i'(t) = \max\{4\rho'(t), g_i(t)\} \qquad \text{and} \qquad  g(t) = \min_{i \in \N}\{g_i'(t)\} = \max\{4\rho'(t), \min_{i\in \N}\{g_i(t)\}\}.
    \]
    The function $\min_{i\in \N}\{g_i(t)\}$ is sublinear (see, for example, \cite[Lemma~3.3]{HeSprianoZbinden:sigma-compact}), and hence $g$, is sublinear.

    Fix $i$. We want to show that $\lambda_i$ is $g$--intersecting. Let $\gamma$ be an $\ell$--projective geodesic which is a subsegment of $\lambda_i$. It suffices to show that $\abs{\gamma}< g_j'(\lceil \ell\rceil )$ for all $j$. 

    \noindent\textbf{Case 1: $i<j$.}  If $\ceil{\ell_j}\leq \ceil{\ell}$, then 
        \begin{align*}
            \abs{\gamma}\leq \abs{\lambda_i}\leq \abs{\lambda_j} \leq 4\rho'(\ceil{\ell_j})\leq 4\rho'(\ceil{\ell}) \leq g_j'(\ceil{\ell}),
        \end{align*}
        where we used \eqref{eqn:bound_x_i} for the second step and \eqref{eq:rho-sandwich} for the third.
    On the other hand, if $\ceil{\ell}< \ceil{\ell_j}$, then $g_j(\ceil{\ell}) = \ceil{\ell}$, and so $\abs{\gamma} \leq \ceil{\ell} = g_j(\ceil{\ell}) \leq g_j'(\ceil{\ell})$.\\

    \noindent\textbf{Case 2: $i\geq j$.}  If $\ceil{\ell_i} \leq \ceil {\ell}$, then 
    \begin{align*}
        \abs{\gamma}\leq \abs{\lambda_i} = \frac{\ceil{\ell_i}}{m_i}\leq \frac{\ceil{\ell}}{m_i}\leq \frac{\ceil{\ell}}{m_j} < g_j(\ceil{\ell}) \leq g_j'(\ceil{\ell}),
    \end{align*}
    where the fourth inequality is due to \eqref{eq:m-i}, which implies that $m_j\leq m_i$.
    On the other hand, if $\ceil{\ell} < \ceil{\ell_i}$, then since $\gamma$ is a subsegment of $\lambda_i$, it is $L_i$--locally $\rho_0$--intersecting. In particular, by the choice of $\mc L_{L_i, \rho'}$, either $\gamma$ is $\rho'$--intersecting, implying that $\abs{\gamma}\leq \rho'(\ceil{\ell}) < g_j'(\ceil{\ell})$, or $\ceil{\ell} \geq \ceil{\ell_i'}\geq \ceil{\ell_i}$. Since the latter contradicts the case assumptions, this completes the proof.

\end{proof}

\subsection{Geodesic MLTG implies MLTG} \label{sec:MLTG}
In this section, we prove the implication Theorem~\ref{thm:main1}\eqref{item:GeodMLTG}$\implies$\eqref{item:MLTG} in Theorem~\ref{thm:geo_mltg-implies-mltg}.  We begin with a few preliminary lemmas. The first gives conditions under which a local Morse quasi-geodesic is contained in a uniform neighborhood of the geodesic between its endpoints.

\begin{lem}\label{lem:Morse_for_combing_lines}
    Let $X$ be a geodesic metric space, let $C\geq 0$, and let $M$ be a Morse gauge. There exists a constant $\mu_0$ such that for every $\mu$, there exists $L_{\mu}$ such that the following holds. If $\gamma\colon [0, T]\to X$ is a path that is $L_\mu$--locally an $M$--Morse $C$--quasi-geodesic and contained in the $\mu$--neighborhood of $[\gamma^-, \gamma^+]$, then $\gamma$  is contained in the $\mu_0$--neighbourhood of $[\gamma^-, \gamma^+]$.
\end{lem}

\begin{proof}    
    We adapt the proof of \cite[Lemma 3.1]{DrutuSprianoZbinden:weak}. Let $\gamma$ be as in the statement, and let $\gamma(s)$ be a point at maximal distance $D\leq \mu$ from $[\gamma^-, \gamma^+]$. Let $\rho  = (8 + C/\mu)C$ be a constant and define $L_{\mu} = 2\rho \mu$. Let $t_1 = \max (0,  s -\rho \mu)$ and $t_2 = \min (T, s + \rho \mu)$. We can assume that $T\geq \rho \mu$, as otherwise $\gamma$ is an $M$--Morse $C$--quasi-geodesic, hence $[\gamma^-, \gamma^+]$ is contained in the $M(1, 0)$--neighbourhood of $\gamma$ and the statement follows directly from Lemma~\ref{lem:reverse_inclusion_QG_nbhd}. Hence $2\rho \mu \geq t_2 - t_1\geq \rho \mu$.

    For $i = 1, 2$, let $u_i\in [\gamma^-, \gamma^+]$ be a closest point to $\gamma(t_i)$. Since $d(u_i, \gamma(t_i)) \leq D$, we have that 
    \begin{align*}
        d(u_1, u_2)&\geq d (\gamma(t_1), \gamma(t_2)) - 2D \geq \frac{t_2 - t_1}{C} - C -   2D \geq 6D\geq 3\left(d(u_1, \gamma(t_1)) + d(u_2, \gamma(t_2))\right).
    \end{align*}
    
    Here we used that $\gamma[t_1, t_2]$ is a $C$--quasi-geodesic because $\rho\mu \leq t_2 - t_1\leq 2\rho\mu = L_{\mu}$. By Lemma \ref{lem:Concat_with_npp}, it follows that if $\alpha_i$ are geodesics connecting $\gamma(t_i)$ to $u_i$, then the concatenation $\eta = \alpha_1 \ast [u_1, u_2]_{[\gamma^-, \gamma^+]}\ast \alpha_2^{-1}$ is a $(3, 0)$-quasi-geodesic joining two points on the $C$--quasi-geodesic $\gamma[t_1,t_2]$, and so $\eta$ is contained in the $M(3, 0)$--neighbourhood of $\gamma[t_1, t_2]$. It follows from Lemma~\ref{lem:reverse_inclusion_QG_nbhd} that $\gamma[t_1, t_2]$, and, in particular, $\gamma(s)$, is in the $\mu_0$--neighbourhood of $\eta$ for some $\mu_0$ only depending on $M(3, 0)$ and $C$. 
\end{proof}

We are now ready to prove that the geodesic MLTG property implies the MLTG property.
\begin{thm}\label{thm:geo_mltg-implies-mltg}
    If a geodesic metric space $X$ satisfies the geodesic MLTG property, then it satisfies the MLTG property.
\end{thm}

\begin{proof}
    Let $M$ be a Morse gauge and $C\geq 1$ a constant. Let $\mu_0$ be the constant from Lemma~\ref{lem:Morse_for_combing_lines} applied to $M$ and $C$. Let $C' = Q'$ be the constant from Lemma~\ref{lemma:close-to-geodesic-implies-quasi-geodesic} applied to $C = \mu_0$ and $Q = C$. Let $\mu_1 = \mu'$ be the constant from Lemma~\ref{lem:reverse_inclusion_QG_nbhd} applied to $C = C'$ and $d = 0$ and $\mu = \mu_0$. Finally, let $M'$ be the Morse gauge from Lemma~\ref{lem:close-to-local-implies-local} applied to $M$ and $Q =\max\{\mu_1, C'\}$. Since $X$ satisfies the geodesic MLTG property, there exists a scale $L'$ and a Morse gauge $M''$ such that any geodesic which is $L'$--locally $M'$--Morse is globally $M''$--Morse. Further, let $L$ be the scale obtained from Lemma~\ref{lem:close-to-local-implies-local} applied to $M = M$, $Q = \max\{\mu_1, C'\}$ and $L' = L'$. Let $\mu_2$ be an upper bound of the Hausdorff distance between an $M''$--Morse geodesic and a $(1, 2)$--quasi-geodesic between its endpoints. Such a constant $\mu_2$ exists by Lemma~\ref{lem:reverse_inclusion_QG_nbhd}. Let $L_{\mu}$ be the constant from Lemma~\ref{lem:Morse_for_combing_lines} applied to $C, M$ and $\mu = \mu_0 + \mu_2+1$. Finally, let $L_0 = \max\{L, L_{\mu}, C'(3\mu_0 + C' + 2)+1\}$.

    \begin{claim}\label{claim:tight-containement}
        Every $L_0$--locally $M$--Morse $C$--quasi-geodesic segment $\gamma$ is contained in the $\mu_0$--neighbourhood of the geodesic connecting its endpoints.
    \end{claim}
    \begin{claimproof}
        Towards a contradiction, let $\gamma$ be a minimal length path that is $L_0$--locally an $M$--Morse $C$--quasi-geodesic such that $\gamma\not\subseteq \mathcal N_{\mu_0}([\gamma^-,\gamma^+])$. Let $a=\gamma^-$ and $b=\gamma^+$. 

        Let $b'\in \gamma$ the last point on $\gamma$ with $d(b,b')=1$, and let $\gamma'$ be the subpath of $\gamma$ from $a$ to $b'$.  Then $\gamma'$ is $L_0$--locally an $M$--Morse $C$--quasi-geodesic, and by the minimality of $\gamma$, we have $\gamma'\subseteq\mathcal N_{\mu_0}([a,b'])$. 
        Hence, by Lemma~\ref{lemma:close-to-geodesic-implies-quasi-geodesic}, $\gamma'$ is a $C'$--quasi-geoodesic. Thus, by Lemma~\ref{lem:reverse_inclusion_QG_nbhd}, $[a, b']$ is contained in the $\mu_1$--neighbourhood of $\gamma'$. Consequently, by Lemma~\ref{lem:close-to-local-implies-local}, $[a, b']$ is $L'$--locally $M'$--Morse. Since $X$ is geodesic MLTG and by the choice of $M'$ and $L'$, the geodesic $[a, b']$ is $M''$--Morse. In particular, by Lemma~\ref{lem:reverse_inclusion_QG_nbhd}, the Hausdorff distance between the $M''$--Morse geodesic $[a, b']$ and  the $(1, 2)$--quasi-geodesic $[a, b]\ast [b, b']$ between its endpoints is at most $\mu_2$. Consequently, $\gamma'$ is contained in the $\mu_0+\mu_2$ neighbourhood of $[a, b]$, implying that $\gamma$ is contained in the $\mu$--neighbourhood of $[a, b]$. Thus by Lemma~\ref{lem:Morse_for_combing_lines}, $\gamma$ is contained in the $\mu_0$--neighbourhood of $[a, b]$. This contradicts our assumption that $\gamma$ is not contained in the $\mu_0$--neighbourhood between its endpoints.
    \end{claimproof}

    Let $\gamma$ be an $L_0$--locally $M$--Morse $C$--quasi-geodesic segment. By Claim~\ref{claim:tight-containement}, $\gamma$ is contained in the $\mu_0$--neighbourhood of $[\gamma^-,\gamma^+]$, and, as we argued in the proof of Claim~\ref{claim:tight-containement} for the geodesic $[a,b']$, this implies that $[\gamma^-,\gamma^+]$ is $M''$--Morse. Since $\gamma$ is a $C'$--quasi-geodesic whose endpoints lie on a $M''$--Morse geodesic, it is  $M'''$--Morse for a Morse gauge $M'''$ only depending on $M''$ and $C'$. Since $M''$ and $C'$ depend only on $M$ and $C$, this concludes the proof.
\end{proof}

\begin{remark}
    Instead of the ``geodesic MLTG'' property, we could define the \emph{special path MLTG} property as follows: between each pair of points, fix a uniform quasi-geodesic. Then replace the geodesics in the geodesic MLTG property with the special paths to get the definition of special path MLTG. An analogue of the proofs above will show that any geodesic metric space satisfying the special path MLTG property satisfies the MLTG property.
\end{remark}

\subsection{$\sigma$--compact Morse boundary implies strongly $\sigma$--compact}\label{sec:CptToStrCpt}
In this section, we show Theorem~\ref{thm:main1}\eqref{item:SigmaCpt}$\implies$\eqref{item:StrSigmaCpt}.

\begin{thm}\label{thm:sigma-compact-implies-strongly-sigma-compact}
   If the Morse boundary of a finitely generated group is $\sigma$-compact, then it is strongly $\sigma$-compact.
\end{thm}

The proof of Theorem~\ref{thm:sigma-compact-implies-strongly-sigma-compact} relies on the following lemma.

\begin{lem}\label{lem:concatenations}
    Let $G$ be a non-hyperbolic finitely generated group with non-empty Morse boundary, and let $X = \cay$ for some finite generating set $S$ of $G$. There exists a constant $C$ such that for any sequence $(\gamma_n)_n$ of geodesic segments in $X$, there exists a $C$--quasi-geodesic $\eta$ such that
    \begin{enumerate}
        \item for all $n\in \N$, there exists a translate of $\gamma_n$ which is a subsegment of $\eta$; and
        \item if $\gamma_n$ is $M$--Morse for all $n\in \N$, then $\eta$ is $M'$--Morse for a Morse gauge $M'$ only depending on $M$.
    \end{enumerate}
\end{lem}

We postpone the proof of Lemma~\ref{lem:concatenations}, which is technical, until after the proof of Theorem~\ref{thm:sigma-compact-implies-strongly-sigma-compact}.
\begin{proof}[Proof of Theorem~\ref{thm:sigma-compact-implies-strongly-sigma-compact}]
    Let $G$ be a finitely generated group with $\sigma$-compact Morse boundary. If $G$ is hyperbolic or has empty Morse boundary, then its Morse boundary is strongly $\sigma$-compact and hence the statement trivially holds. Thus, we assume that $G$ is non-hyperbolic with non-empty Morse boundary.

    Let $C$ be the constant from Lemma~\ref{lem:concatenations}. By Lemma~\ref{lem:sigma_cpt2} there exists a sequence $(M_n)_n$ of Morse gauges such that any $C$--quasi-geodesic which is Morse is $M_n$--Morse for some $n$ by Lemma~\ref{lem:sigma_cpt2}. Suppose toward a contradiction that $\mb G$ is not strongly $\sigma$-compact. Then by Lemma~\ref{lem:strongly-sigma-compact} there exists a Morse gauge $M$ and a sequence of $M$--Morse geodesic segments $(\gamma_n)_n$ such that for all $n$, the geodesic $\gamma_n$ is not $M_{n}$--Morse.  

    Applying Lemma~\ref{lem:concatenations} to $(\gamma_n)_n$, we obtain a $C$--quasi-geodesic $\eta$ containing a translate of each the $\gamma_n$. Since each $\gamma_n$ is $M$--Morse, $\eta$ is Morse by Lemma~\ref{lem:concatenations}. In particular, by $\sigma$-compactness, there exists an integer $i\geq 0$ such that $\eta$ is $M_i$--Morse. Consequently, $\gamma_i$, which is a subsegment of a translate of $\eta$, must be $M_i$--Morse. This contradicts  the choice of $\gamma_i$ and concludes the proof.
\end{proof}

Intuitively, to prove Lemma~\ref{lem:concatenations}, we would like to build $\eta$ as a concatenation of translates of the segments $\gamma_n$.  That is, we would like to choose elements $h_n\in G$ such that $h_n\gamma_n^+=h_{n+1}\gamma_{n+1}^-$. However, when concatenating  arbitrary quasi-geodesics, the quasi-geodesic constants of the concatenation cannot be controlled.  Moreover, when each $\gamma_n$ is $M$--Morse, the Morse constant of the concatenation may grow as we concatenate more segments.

To solve this problem, we choose translates of $\gamma_n$ that have larger and larger gaps between them, that is, so that the distance from $g_n\gamma_n^+$ to $g_{n+1}\gamma_{n+1}^-$ goes to infinity.  Then we build $\eta$ by concatenating the translates $g_n\gamma_n$ and the geodesics $[g_n\gamma_n^+, g_{n+1}\gamma_{n+1}^-]$ between them. The key feature here is that the lengths of the extra geodesics we add go to infinity, and this allows us to gain uniform control on the quasi-geodesic constants of the resulting infinite path. 

In the next two technical lemmas, we analyze the Morse and quasi-geodesic properties of increasingly complicated concatenation structures.  First, in Lemma~\ref{claim:Morse-projection2}, we study the possible concatenations of a bi-infinite Morse quasi-geodesic with a finite quasi-geodesic.  Then, given a quasi-geodesic $\beta$ and a translate $g\gamma$ of a fixed geodesic $\gamma$, we analyze the concatenation $\beta\ast [\beta^+,g\gamma^-]\ast g\gamma$ in Lemma~\ref{lem:single-concatenation}.  With these in hand, we are then able to prove Lemma~\ref{lem:concatenations}.

\begin{lem}
\label{claim:Morse-projection2}
        Let $X$ be a metric space, let $\zeta \colon [0, T]\to X$ be a $K$--quasi-geodesic, and let $\eta$ be a bi-infinite $M$--Morse geodesic with $\zeta(0) = \eta(0)$. Then at least one of $\eta(-\infty, 0]\ast \zeta$ and $\eta(\infty, 0]\ast \zeta$ is a $(K\cdot C')$--quasi-geodesic, where $C' = C'(K) = 16 + 3M(3K)$.
\end{lem}

Note that for the lemma to hold it is essential that $\eta$ is Morse. Indeed, take $\zeta$ to be the log-spiral quasi-geodesic around the origin in $\Z^2$, and let $\eta$ be any bi-infinite geodesic containing the origin. The longer the subpath of $\zeta$ we choose (which are all are uniform quality quasi-geodesics), the more often both sides of $\eta$ intersect $\zeta$, contradicting the lemma for any choice of $C'(K)$. 

\begin{figure}
    \centering
    \includegraphics[width=0.6\linewidth]{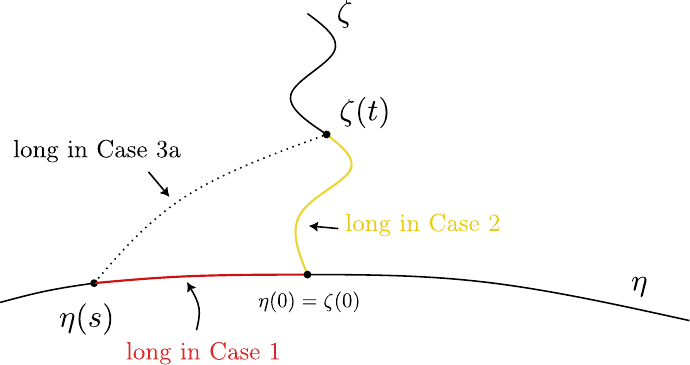}
    \caption{The different cases represent different subsegments being long.}
    \label{fig:anti-spiral1}
\end{figure}

\begin{proof}
    It suffices to prove that
    \begin{align}\label{eq:upper-bound1}
        \frac{\abs{t}+\abs{s}}{KC'} - KC'\leq d(\eta(s), \zeta(t))
    \end{align}
    either for all $t\in [0, T]$ and $s\in [0,\infty)$ or for all $t\in [0, T]$ and $s\in [0, -\infty)$. To do so, we divide the proof into cases. In particular, we show that if $s\in (-\infty, \infty)$ and $t\in [0, T]$ satisfy the assumptions of Case 1, 2 or 3a, then \eqref{eq:upper-bound1} holds. Finally, we show that if Case 3b occurs for positive $s$, it cannot occur for negative $s$ and vice versa. It is in this case where we use that $\eta$ is $M$--Morse. Cases 1, 2 and 3a are depicted in Figure~\ref{fig:anti-spiral1}.\\

    \noindent\textbf{Case 1: $\abs{s} = d(\eta(0), \eta(s)) > 2 d(\zeta(t), \eta(0))$.} In this case, using the triangle inequality (in the first step), the case assumption (in the second step) and the fact that $\eta$ is a geodesic and $\zeta$ is a $K$--quasi-geodesic (in the third step), we obtain \eqref{eq:upper-bound1}:
    \begin{align*}
        d(\eta(s), \zeta(t)) &\geq d(\eta(s), \eta(0)) - d(\eta(0), \zeta(t)) \geq \frac{d(\eta(0), \eta(s))}{4} + \frac{d(\eta(0), \zeta(t))}{2}\\ &\geq \frac{\abs{s}}{4} + \frac{\abs{t}}{2K} - \frac{K}{2} \geq \frac{\abs{s}+\abs{t}}{4K} - 4K.
    \end{align*}
    
    \noindent\textbf{Case 2: $\abs{s} = d(\eta(0), \eta(s)) < d(\zeta(t), \eta(0))/2$.}  In this case, using the triangle inequality (in the first step), the case assumption (in the second step) and the fact that $\eta$ is a geodesic and $\zeta$ is a $K$--quasi-geodesic (in the third step), we obtain \eqref{eq:upper-bound1}:
    \begin{align*}
        d(\zeta(t), \eta(s)) &\geq d(\zeta(t), \eta(0)) - d(\eta(0), \eta(s)) \geq \frac{d(\zeta(t), \eta(0))}{4} + \frac{d(\eta(0), \eta(s))}{2}\\
        &\geq \frac{\abs{t}}{4K} - \frac{K}{4} +\frac{\abs{s}}{2} \geq \frac{\abs{t}+\abs{s}}{4K} - 4K.
    \end{align*}

    \noindent\textbf{Case 3:  Neither Case 1 nor Case 2 holds.}  That is,
    \begin{align*}
        \frac{d(\zeta(t), \eta(0))}{2}\leq d(\eta(0), \eta(s))\leq 2d(\zeta(t), \eta(0)). 
    \end{align*}
    
    We divide this case into two subcases.\\
    
    \noindent\textbf{Case 3a: $d(\eta(s), \zeta(t))\geq d(\eta(0), \zeta(t))/4 - C'/2$. }
    In this case, we can again use the case assumptions (in the first and second step) and the fact that $\eta$ and $\zeta$ are a geodesic and a $K$--quasi-geodesic (in the third step) to obtain \eqref{eq:upper-bound1}:
    \begin{align*}
        d(\zeta(t), \eta(s))&\geq \frac{d(\eta(0), \zeta(t))}{4} - \frac{C'}{2}\geq \frac{d(\eta(0), \zeta(t))}{8} + \frac{d(\eta(0), \eta(s))}{16} - \frac{C'}{2}\\
        &\geq \frac{\abs{t}}{8K} - \frac{K}{8} +\frac{\abs{s}}{16} - \frac{C'}{2} \geq \frac{\abs{s}+\abs{t}}{C'K} - C'K.
    \end{align*}

    \begin{figure}
        \centering
        \includegraphics[width=0.6\linewidth]{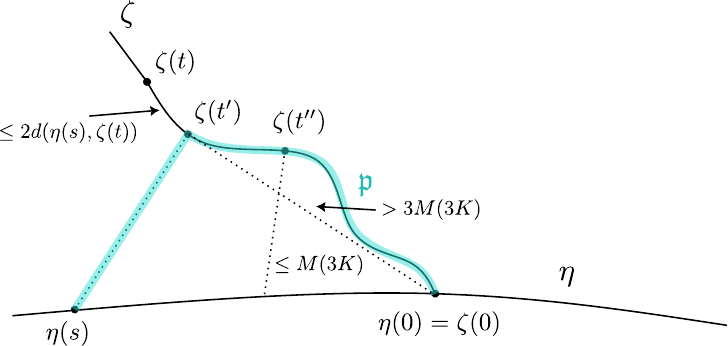}
        \caption{The point $\zeta(t')$ is the point on $\zeta$ closest to $\eta(s)$ in Case 3b. Since $\mathfrak{p}$ is a $3K$--quasi-geodesic, $\mathfrak{p}$, and specifically $\zeta(t'')$, is in the $M(3K)$--neighborhood of $\eta[0,s]$.}
        \label{fig:snti-spiral2}
    \end{figure}
    
    \noindent\textbf{Case 3b: $d(\eta(s), \zeta(t)) < d(\eta(0), \zeta(t))/4 - C'/2$.} Let $\zeta(t')$ be the point on $\zeta$ closest to $\eta(s)$. The path $\mathfrak{p} = \zeta[0, t'] \ast [\zeta(t'), \eta(s)]$ is a $3K$--quasi-geodesic by Lemma~\ref{lem:Concat_with_npp}. By the triangle inequality and the fact that $d(\zeta(t'), \eta(s))\leq d(\zeta(t), \eta(s))$ by our choice of $t'$, $d(\zeta(t), \zeta(t'))\leq 2 d(\eta(s), \zeta(t))$.  Using the case assumption and the triangle inequality, this yields $d(\zeta(t'),\eta(0))\geq d(\zeta(t),\eta(0))-d(\zeta(t),\zeta(t')) >4(d(\eta(s),\zeta(t))+C'/2)-2d(\eta(s),\zeta(t))=2d(\eta(s),\zeta(t))+2C'\geq C'$.
    
    Moreover, since $\eta$ is $M$--Morse, $d(\zeta(t''), \eta[0, s])\leq M(3K)$ for all $0\leq t''\leq t'$. We will now show that this case cannot happen for both positive and negative $s$. Assume it does. Then there exist $s_1\leq 0\leq s_2$ and $t_1', t_2'\in [0, T]$ such that for $i = 1, 2$
    \begin{align}\label{eq:sandwich}
        \zeta[0, t_i'] \subset \mc N_{M(3K)} \left( \eta[0, s_i]\right)\qquad \text{and}\qquad d(\zeta(t_i'), \eta(0))\geq C' > 3M(3K).
    \end{align}

    \begin{figure}
        \centering
        \includegraphics[width=0.6\linewidth]{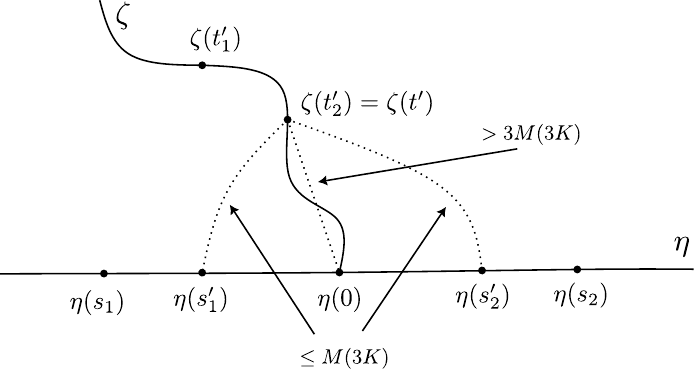}
        \caption{In this picture, $t_1' > t_2'$ and hence $t' = t_2'$. If $t_2' > t_1'$ we instead have $t' = t_1'$.}
        \label{fig:anti-spiral3}
    \end{figure}

    Let $t'$ be the smaller of $t_1'$ and $t_2'$. By \eqref{eq:sandwich} there exists $s_1'  \leq 0 \leq s_2'$ with $d(\zeta(t_i'), \eta(s_i'))\leq M(3K)$ for $i = 1,2$. This is depicted in Figure~\ref{fig:anti-spiral3}. Since $\eta(s_1')$ and $\eta(s_2')$ are on a geodesic containing $\eta(0)$, we obtain $d(\eta(s_1'), \eta(0))\leq d(\eta(s_1'), \eta(s_2'))\leq 2M(3K)$. Thus, by the triangle inequality, $d(\zeta(t'), \eta(0))\leq 3M(3K)$, contradicting \eqref{eq:sandwich}. 

    Hence, Case 3b cannot happen for both $s<0$ and $s >0$. In all other cases, \eqref{eq:upper-bound1} is satisfied, concluding the proof of the claim. 
\end{proof}
  
The next lemma deals with a more complicated concatenation.
\begin{lem}\label{lem:single-concatenation}
    Let $G$ be a non-hyperbolic finitely generated group with non-empty Morse boundary, and let $X = \cay$ for some finite generating set $S$ of $G$. There exist constants $C_1, C_2\geq 1$ and a Morse gauge $M$ such that the following holds for all constants $L$, all $C_1$--quasi-geodesics $\beta$, and all geodesics $\gamma$ in $X$. There exists an element $g\in G$ such that 
    \begin{enumerate}[(\roman*)]
        \item the path $\beta\ast [\beta^+, g\cdot \gamma^-] \ast g\cdot \gamma$ is a $C_2$--quasi-geodesic; 
        \item the path $[\beta^+, g\cdot \gamma^-]\ast g\cdot \gamma$ is a $C_1$--quasi-geodesic; and
        \item the geodesic $[\beta^+, g\cdot \gamma^-]$ is $M$--Morse and has length at least $L$.
    \end{enumerate}
\end{lem}

\begin{proof}

    Let $\beta$ and $\gamma$ be as in the statement. By potentially increasing $C_1, C_2$ and $M$ slightly at the end, it suffices to prove the lemma in the case where the endpoints of $\beta \colon [0, T_1]\to X$ and $\gamma \colon  [0, T_2]\to X$ are  vertices.
    
    Since $G$ has non-empty Morse boundary, there is a biinfinite geodesic $\eta \colon (-\infty, \infty)\to X$ which is $M_1$--Morse for some Morse gauge $M_1$. We may assume that $\eta(0)$ is a vertex. 

    \begin{figure}
        \centering
        \includegraphics[width=0.6\linewidth]{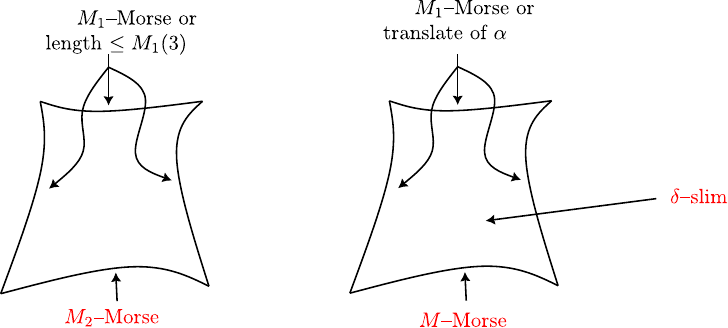}
        \caption{If any quadrangles satisfies the properties written in black, then the properties written in red follow.}
        \label{fig:dependency}
    \end{figure}

    Let $M_2$ be a Morse gauge such that in a geodesic quadrangle where three of the four sides are $M_1$--Morse or have length at most $M_1(3)$, the fourth is $M_2$--Morse; such a constant exists by noting that a finite geodesic is Morse for some Morse gauge depending only on its length and then applying Lemma~\ref{ref:triangles}. Let $\alpha$ be a geodesic segment that is not $M_2$--Morse whose endpoints are vertices. Since $G$ is not hyperbolic, such a geodesic $\alpha$ exists. Further, let $M$ be a Morse gauge and $\delta$ an integer such that a geodesic quadrangle where two of the four sides are $M_1$--Morse and the third side is (a translate of) $\alpha$ is $\delta$--slim and its fourth side is $M$--Morse.  The Morse gauge $M$ exists again by Lemma~\ref{ref:triangles}, and $\delta$ exists by Lemma~\ref{ref:delta}. Note that $M$ and $\delta$ only depend on $\alpha$ and $M_1$ and not on $\beta$ or $\gamma$. This dependency is depicted in Figure~\ref{fig:dependency}.

    Let $C_1' =16+3M_1(1)$, $C_1 = C_1' + 2M_1(3)+4, C_2' = 16 +3M_1(3C_1)$ and $C_2 = C_2' + 2M_1(3)+4$. Let $\eta_1, \eta_2$ be translates of $\eta$ such that $\eta_1(0) = \beta^+$ and $\eta_2(0) = \gamma^-$. Lemma~\ref{claim:Morse-projection2} implies that, up reparameterizing $\eta_1$ and/or $\eta_2$, the concatenation $\beta\ast \eta_1[0, \infty)$ is a $C_2'$--quasi-geodesic and the concatenation $\eta_2(\infty, 0]\ast \gamma$ is a $C_1'$--quasi-geodesic.

    \smallskip

    \begin{figure}
        \centering
        \includegraphics[width=0.7\linewidth]{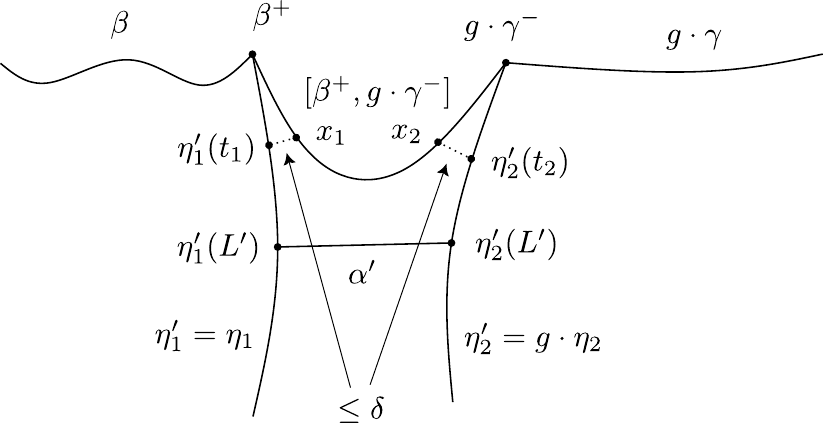}
        \caption{Setup of Lemma~\ref{lem:single-concatenation}.}
        \label{fig:complex-concatenation}
    \end{figure}
    Finally, let $L' \geq \abs{\alpha}+101\delta + L + K$, where $K = 2(\diam(\beta)+\abs{\beta}+\diam(\gamma)+\abs{\gamma})$ and $\abs{\beta}$ denotes the domain length of $\beta$. Let $g\in G$ be such that $\alpha' = [\eta_1(L'), g\cdot \eta_2(L')]$ is a translate of $\alpha$.  Such a $g$ exists because $G$ acts transitively on the vertices of $X=\Cay(G,S)$. Let $\eta_1' = \eta_1$ and $\eta_2' = g\cdot \eta_2$. This is depicted in Figure~\ref{fig:complex-concatenation}.

    Observe that $[\beta^+, g\cdot \gamma^-]$ is $M$--Morse by construction of $M$ and $g$: it is one side of a geodesic quadrangle whose other three sides are $\eta_1'[0,L']$, $\eta_2'[L', 0]$, which are $M_1$--Morse, and $\alpha'$, a translate of $\alpha$. 

    \begin{claim}\label{claim:close-enough}
        For $i = 1, 2$, there exists a point $x_i\in [\beta^+, g\cdot \gamma^-]$ and $t_i \in [L+\delta + K, L+100\delta+K]$ such that $d(x_i, \eta_i'(t_i))\leq \delta$. Moreover, one can choose $x_1, x_2$ such that every point on $[\beta^+, x_1]$ and $[x_2, g\cdot \gamma^-]$ is in the $M_1(3)$--neighbourhood of $\eta_1'[0, \infty)$ and $\eta_2'[0, \infty)$, respectively.
    \end{claim}

    The claim above implies $d(\beta^+, x_1)\geq L$ and $d(x_2, g\gamma^-)\geq L$. In particular, it shows that the geodesic $[\beta^+, g\cdot \gamma^-]$ has length at least $L$.  This shows (iii) holds.

    \begin{claimproof}[Proof of Claim]
        We show this for $i = 1$, as the proof for $i = 2$ is analogous. By construction and the definition of $\delta$, there exist points $y_1, y_2$ on $[\beta^+, g\cdot \gamma^-], \eta_2'$ or $\alpha'$ in the $\delta$--neighborhood of $\eta_1'(L+\delta+K)$ and $\eta_1'(L+100\delta+K)$, respectively. Since $L' - L - 100\delta - K> \abs{\alpha'}+\delta$, the point $y_i$ cannot lie on $\alpha'$ for $i=1,2$. If one of $y_1, y_2$ lies on $[\beta^+, g\cdot \gamma^-]$, then we can define $x_1$ as the closest point on $[\beta^+, g\cdot \gamma^-]$ to $\eta_1'(L+\delta+K)$, if $y_1$ lies on $[\beta^+, g\cdot \gamma^-]$, or as the closest point to $\eta_1'(L+100\delta+K)$, if $y_2$ lies on $[\beta^+, g\cdot \gamma^-]$. In this case $x_1$ clearly satisfies the first part of the statement. The moreover part follows from the fact that, by Lemma~\ref{lem:Concat_with_npp}, $[\eta_1'(L+\delta+K, x_1] \ast [x_1, \beta^+]$ is a $(3, 0)$--quasi-geodesic with endpoints on the $M_1$--Morse geodesic $\eta_1'$.
        
        Next, toward a contradiction, we assume that both $y_1$ and $y_2$ lie on $\eta_2'$. Let $y_1', y_2'$ points on $\eta_2'$ closest to $\eta_1'(L+\delta+K)$ and to $\eta_2'(L+100\delta+K)$, respectively. Applying  Lemma~\ref{lem:Concat_with_npp}, we see that $[\eta_1'(L+\delta+K), y_1']\ast [y_1', y_2']_{\eta_2'}\ast [y_2', \eta_1'(L+100\delta+K)]$ is a $(3, 0)$--quasi-geodesic with endpoints on $\eta_1'$. In particular, $y_1'$ is in the $M_1(3)$--neighborhood of $\eta_1'$. This is a contradiction to $\alpha$, and hence $\alpha'$, not being $M_2$--Morse, as it places $\alpha'$ in a quadrangle where all other sides are either subsegments of $\eta_1'$ or $\eta_2'$, and hence $M_1$--Morse, or have length at most $M_1(3)$. 
    \end{claimproof}

    \begin{claim}\label{claim:q-geo}
    The concatenations $\beta\ast[\beta^+, x_1]$ and $[x_2, g\cdot \gamma^-]\ast g \gamma$ are a $C_2$--quasi-geodesic and a $C_1$--quasi-geodesic, respectively. 
    \end{claim}
    \begin{claimproof}
        We prove this for $\beta\ast [\beta^+, x_1]$; the proof for $[x_2, g\cdot \gamma^-]\ast g \gamma$ is analogous. By Claim~\ref{claim:close-enough}, any point $z$ on $[\beta^+, x_1]$ has distance at most $M_1(3)$ to $\eta_1'[0, \infty)$ and hence, by the triangle inequality, distance at most $2M_1(3)$ to $\eta_1'(d(\beta^+, z))$. Since $\beta\ast \eta_1'[0, \infty)$ is a $C_2'$--quasi-geodesic, the above implies that $\beta\ast [\beta^+, x_1]$ is a $C_2 \geq C_2'+2M_1(3)$--quasi-geodesic.
    \end{claimproof}

    \begin{claim}\label{claim:full-q-geo}
        The concatenations $\beta\ast[\beta^+, g\cdot \gamma^-]$ and $[\beta^+, g\cdot \gamma^-]\ast g\gamma$ are a $C_2$--quasi-geodesic and a $C_1$--quasi-geodesic, respectively. 
    \end{claim}
    \begin{claimproof}
        We prove this for $\xi = \beta\ast [\beta^+, g\cdot \gamma^-]$; the proof for $[\beta^+, g\cdot \gamma^-]\ast g\cdot \gamma$ is analogous. By Claim~\ref{claim:q-geo}, it is enough to prove that the lower quasi-geodesic inequality holds for points $p = \xi(s)$ on $\beta$ and $p' = \xi(s')$ on $[\beta^+, g\cdot \gamma^-]$ with $\ell = d(\beta^+, p')\geq L+K$. By the triangle inequality we have that $d(p, p')\geq \ell - \diam(\beta) \geq \ell/2$. We also know that $\abs{s'- s}\leq \ell + \abs{\beta}\leq 2\ell$, hence $d(p, p')\geq \abs{s'-s}/4\geq \abs{s'-s}/C_2 - C_2$. 
    \end{claimproof}

    \begin{claim}\label{claim:total-full-q-geo}
        The concatenation $\xi = \beta\ast[\beta^+, g\cdot \gamma^-]\ast g\gamma$ is a $C_2$--quasi-geodesic. 
    \end{claim}
    \begin{claimproof}
        By Claim~\ref{claim:full-q-geo}, it is enough to prove that the lower quasi-geodesic inequality holds for points $p = \xi(s)$ on $\beta$ and $p' = \xi(s')$ on $g\cdot \gamma$ with $\ell = d(\beta^+, g\cdot \gamma^-)\geq L+K = L+2\diam(\beta) + 2\abs{\beta}+2\diam(\gamma)+2\abs{\gamma}$. By the triangle inequality we have that $d(p, p')\geq \ell - \diam(\beta) -\diam(\gamma) \geq \ell/2$. We also know that $\abs{s'- s}\leq \ell + \abs{\beta}+\abs{\gamma}\leq 2\ell$, hence $d(p, p')\geq \abs{s'-s}/4\geq \abs{s'-s}/C_2 - C_2$. 
    \end{claimproof}

\end{proof}

We are now ready to prove Lemma~\ref{lem:concatenations}.
\begin{proof}[Proof of Lemma~\ref{lem:concatenations}]
    Let $(\gamma_n)_n$ be a sequence of geodesic segments in $X$, as in the statement.  Let $C_1, C_2, M_0$ be the constants and Morse gauge from Lemma~\ref{lem:single-concatenation}. We prove Lemma~\ref{lem:concatenations} for $C = 4C_2$. Let $g_0 = 1$, and define $\eta_0 = \beta_0 =\gamma_0$. For $i\geq 0$, inductively apply Lemma~\ref{lem:single-concatenation} to $\beta = \beta_i$, $\gamma = \gamma_{i+1}$ and $L = L_i =  2(i-1)(\abs{\eta_{i-2}}+1)$ to get an element $g_{i+1}\in G$ such that $\beta_{i+1} = [\beta_i^+, g_{i+1}\cdot \gamma_{i+1}^-]\ast g_{i+1}\cdot\gamma_{i+1}$ is a $C_1$--quasi-geodesic, $\beta_i \ast \beta_{i+1}$ is a $C_2$--quasi-geodesic, and $[\beta_i^+, g_{i+1}\cdot \gamma_{i+1}^-]$ is $M_0$--Morse and has length at least $L_i$. Lastly, define $\eta_{i+1} = \eta_i \ast \beta_{i+1}$ and
    \begin{align*}
        \eta = \prod_{i = 0}^\infty g_i\cdot \gamma_i\ast [g_i\cdot\gamma_i^+, g_{i+1}\cdot \gamma_{i+1}^-] = \prod_{i=0}^\infty \beta_i.
    \end{align*}

    \begin{claim}\label{claim:con_qg}
        The path $\eta$ is a $4C_2$--quasi-geodesic.
    \end{claim}

    \begin{claimproof}
    Let $\eta(s), \eta(t)$ be points on $\eta$. Since $\eta$ is a concatenation of geodesics, it suffices to prove that 
    \begin{align}\label{eq:lowerines}
        d(\eta(s), \eta(t)) \geq \frac{\abs{t -s}}{4C_2}-4C_2.
    \end{align}
    By definition, $\beta_i\ast \beta_{i+1}$ is a $C_2$--quasi-geodesic for all $i\geq 0$. It thus suffices to show \eqref{eq:lowerines} holds in the case where $\eta(t)$ lies on $\beta_i$ and $\eta(s)$ lies on $\eta_{i-2}$ for some $i\geq 2$. Let $s_0$ be such that $\eta(s_0) = \beta_{i-2}^+ = \beta_{i-1}^-$.

    \begin{figure}
        \centering
        \includegraphics[width=0.5\linewidth]{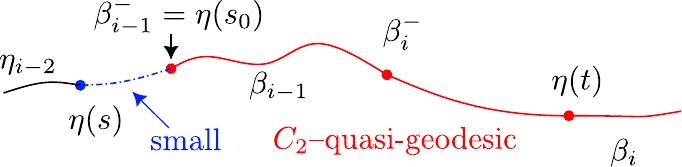}
        \caption{The path $\eta$ is a $4C_2$--quasi-geodesic.}
        \label{fig:multi-concatenation1}
    \end{figure}
    
    Since $\eta(s)$ lies on $\eta_{i-2}$, we have that $\abs{s-s_0} \leq \abs{\eta_{i-2}} \leq L_{i-2}/2\leq L_{i-1}/2$. This is depicted in Figure~\ref{fig:multi-concatenation1}. In particular, $d(\eta(s), \eta(s_0))\leq \abs{s-s_0}\leq L_{i-1}/2$. On the other hand, since $\eta(t)$ lies on $\beta_i$, $\abs{s_0 - t}\geq \abs{\beta_{i-1}} \geq L_{i-1}$, implying $\abs{s_0 - t}\geq \abs{s - t}/2$ and $\abs{s-s_0}\leq \abs{s-t}/2$. Using this and the fact that $\beta_{i-1}\ast \beta_i$ is a $C_2$--quasi-geodesic, we obtain
    \begin{align*}
        d(\eta(s), \eta(t))&\geq d(\eta(s_0), \eta(t)) - d(\eta(s), \eta(s_0)) \geq \frac{\abs{t - s_0}}{C_2} - C_2 - \abs{s-t}/2\\ 
        &\geq \frac{\abs{t - s}}{2C_2} - C_2 - \frac{\abs{t-s}}{2}\geq \frac{\abs{t-s}}{4C_2} - 4C_2,
    \end{align*}
    which concludes the proof of the claim.
    \end{claimproof}

    \smallskip

    Assume that all $\gamma_n$ are $M$--Morse for some Morse gauge $M$. Then $\beta_{i}\ast \beta_{i+1}$ is the concatenation of 4 geodesics that are all $\max\{M,M_0\}$--Morse. Since $\beta_i\ast \beta_{i+1}$ is a $C_2$--quasi-geodesic, it is $M'$--Morse for some Morse gauge $M'$ depending only on $M$, $M_0$, and $C_2$ \cite[Lemma~2.8 v)]{Z:manifold}.

    \begin{claim}\label{claim:con_morse2}
        There exists a Morse gauge $N$ only depending on $M'$ such that the following holds for all constants $K\geq 1$. Let $\lambda \colon [0, T]\to X$ be a $K$--quasi-geodesic with $\lambda^-=\eta(t_1)$ on $\beta_j$ and $\lambda^+$ at distance at most $M'(3K)$ from $\eta(t_2)\in \beta_i$ for some $j\leq i$. Then $\lambda$ is contained in the $N(K)$--neighbourhood of $\eta[t_1, t_2]$.
    \end{claim}
    \begin{claimproof}
        Fix a constant $K$. We will prove this by induction on $i$. That is, assume that Claim~\ref{claim:con_morse2} holds for $i < n$. We will show that it holds for $i = n$. Let $\lambda$ be as in the statement with $i = n$. Let $i_0 = i_0(K)$ be a large enough index.

        We can choose $N(K)$ large enough (compared to $K$ and $i_0(K)$) such that the statement for $i\leq i_0(K)$ follows from $\lambda$ being a $K$--quasi-geodesic whose endpoints are in the $M'(3K)$ neighbourhood of $\eta_{i_0(K)}$. The latter is the concatenation of at most $2i_0(K)$ geodesics which are $M'$--Morse and hence $M''$--Morse for some $M''$ only depending on $M'$ and $i_0(K)$. Thus, we may assume that $i> i_0(K)$. 
        
        If $j\geq i-1$, then define $\lambda' = \lambda\ast [\lambda(T), \eta(t_2)]$. This is a $K' = (3K + 2M'(3K))$--quasi-geodesic with endpoints on the $M'$--quasi-geodesic $\beta_{i-1}\ast \beta_i$. Hence the statement follows for $N(K)\geq M'(K')$.

        \begin{figure}
            \centering
            \includegraphics[width=0.5\linewidth]{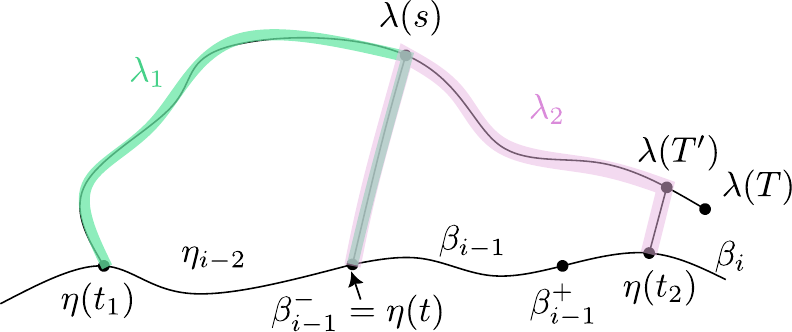}
            \caption{Proof of Claim~\ref{claim:con_morse2}. The paths $\lambda_1$ and $\lambda_2$ are $(3K)$-- and $K'$--quasi-geodeiscs respectively.}
            \label{fig:multi-concatenation2}
        \end{figure}

        From now on, assume $j < i-1$. Let $\lambda(s)$ be a point on $\lambda$ closest to $\beta_{i-1}^- = \eta(t)$ and let $\lambda(T')$ be a point on $\lambda$ closest to $\eta(t_2)$. This is depicted in Figure~\ref{fig:multi-concatenation2}. Define 
        \begin{align*}
            \lambda_1 = \lambda[0, s]\ast [\lambda(s), \beta_{i-1}^-] \qquad \text{and}\qquad \lambda_2 =[\beta_{i-1}^-, \lambda(s)]\ast \lambda[s, T']\ast [\lambda(T'), \eta(t_2)]. 
        \end{align*}
        Because we used closest point projections, Lemma~\ref{lem:Concat_with_npp} ensures that $\lambda_1$ is a $(3K)$--quasi-geodesic and $\lambda_2$ is a $K'$--quasi-geodesic. Since $\lambda_2$ has endpoints on $\beta_{i-1}\ast \beta_i$, it is contained in the $M'(K')$ neighbourhood of $\eta[t, t_2]$. Moreover, $\lambda[T', T]$ is a $K$--quasi-geodesic with endpoints in the $M'(3K)$--neighbourhood of $\eta(t_2)$. Hence, for $N(K)$ large enough compared to $M'(3K')$, we have that $\lambda_2$ and $\lambda[T', T
        ]$ are contained in the $N(K)$--neighbourhood of $\eta[t, t_2]$. 

        It remains to show that $\lambda[0, s]$ is contained in the $N(K)$--neighbourhood of $\eta[t_1, t_2]$. In fact, it suffices to show that $d(\beta_{i-1}, \lambda(s))\leq M'(3K)$, as then $\lambda[0, s]$ is contained in the $N(K)$ neighbourhood of $\eta[t_1, t_2]$ by induction.

        By Lemma~\ref{lem:Concat_with_npp}, if
        \begin{align}\label{eq:3K-quasi}
            \abs{T' - s}\geq 3K\left(d(\beta_{i-1}^-, \lambda(s)) + d(\lambda(T'), \eta(t_2))\right),
        \end{align} 
        then $\lambda_2$ is not only a $K'$--quasi-geodesic but actually a $3K$--quasi-geodesic. To show \eqref{eq:3K-quasi} holds, choose $i_0(K)$ large enough such that $L_{i-1}$, and hence $\abs{T' - s}$, is much larger than  $d(\lambda(T'), \eta(t_2))$ and $d(\beta_{i-1}, \lambda(s))$. The former is bounded by $M'(3K)$ and the latter is bounded by the maximal diameter of a $(3K)$--quasi-geodesic with endpoints on $\eta_{i-2}$, whereas $L_{i-1} = 2(i-1)(\abs{\eta_{i-2}}+1)$. 

        Thus, $\lambda_2$ is indeed a $(3K)$--quasi-geodesic, implying that $d(\lambda(s), \beta_{i-1}\ast \beta_i)\leq M'(3K)$. Again using that $i_0(K)$ is large enough and $i\geq i_0(K)$, we obtain that $d(\lambda(s), \beta_{i-1}^-)$ is much smaller than $L_{i-1}$, implying that $d(\lambda(s), \beta_i) > M'(3K)$. Hence $d(\lambda(s), \beta_{i-1})\leq M'(3K)$, concluding the induction.
    \end{claimproof}
    
    \smallskip
    
    Claim~\ref{claim:con_morse2} yields Item \ref{prop:morse-carries}, and hence concludes the proof.
\end{proof}

\section{A MLTG group that is not acylindrically hyperbolic}\label{sec:nonahgroup}
The goal of this section is to construct a finitely generated MLTG group with an infinite order Morse element that is not acylindrically hyperbolic, proving Theorem~\ref{thm:main2}.  Such a group is necessarily not a $C'(1/N)$ small cancellation group for any $N$, since such groups are acylindrically hyperbolic.  However, the group we construct satisfies a property related to \textit{graded small cancellation}, which itself is a generalization of a small cancellation group.  Our group is what we call an \textit{expanding graded small cancellation} group; see Definition~\ref{def:ExpandingGSC}. 

In Section~\ref{sec:GSC}, we review the definition of a graded small cancellation group.  In Section~\ref{sec:ExpandingGSC}, we define an expanding graded small cancellation group and explore their basic properties, in particular proving a version of a Greendlinger lemma for such groups in Lemma~\ref{lem:upgraded-contiguity-diagrams}.  In Section~\ref{sec:IntFcnsAndEGSC}, we investigate the relationship between the intersection function as defined in Definition~\ref{def:IntFcn} and a generalization of the usual intersection function in small cancellation groups, the \textit{relator intersection function} defined in Definition~\ref{defn:relator-intersection-function}, which, roughly speaking, measures the length of subwords of the label of a geodesic that are contained in a relator of length at most $n$.  

Finally, in Section~\ref{sec:nonAHgp}, we construct the desired non-acylindrically hyperbolic MLTG group.  The construction is inspired by the authors' construction of a finitely generated MLTG group that contains a Morse element that is not loxodromic in any isometric action of the group on a hyperbolic space \cite{AZ}.

\subsection{The $C(\eps, \mu, \rho)$--condition and contiguity diagrams}\label{sec:GSC}
In this section, we review some basic definitions related to graded small cancellation and recall the definition of a contiguity diagram.  For a more detailed discussion of this material, we refer the reader to \cite{OlshankiiOsinSapir:lacunary}.

Every group that appears in this section has a generating set attached to it. We identify the group with the Cayley graph associated to that generating set. In particular, we say that a word $w$ over this generating set is a geodesic in the group if it some (equivalently any) path labeled by $w$ in the Cayley graph with respect to this generating set is a geodesic.

We first recall the notion of an $\eps$--piece, which, roughly speaking, is a subpath of a relator that $\eps$--fellow travels a subpath of another relator. Let $H$ be a group generated by a set $S$.
\begin{defn}
     Let $\mathcal R$ be a symmetrized set of reduced words in $S^{\pm 1}$.  For $\eps>0$, a subword $U$ of a word $R\in \mathcal R$ is called an \textit{$\eps$--piece} if there exists a word $R'\in \mathcal R$ such that:
        \begin{enumerate}
            \item $R\equiv UV$ and $R'\equiv U'V'$ for some $V,U', V'$;
            \item $U'=YUZ$ in $H$ for some words $Y,Z$ such that $\max\{|Y|,|Z|\}\leq \eps$; and
            \item $YRY^{-1}\neq R'$ in the group $H$.
        \end{enumerate}
    Note that if $U$ is an $\eps$--piece, then $U'$ is an $\eps$--piece, as well.
\end{defn}

\begin{defn}
    Let $\eps\geq 0$, $\mu\in (0,1)$, and $\rho>0$.  A symmetrized set $\mathcal R$ of words over the alphabet $S^{\pm 1}$ satisfies the \emph{$C(\eps,\mu,\rho)$--condition} over the group $H$ if
        \begin{enumerate}[(C1)]
            \item all words from $\mathcal R$ are geodesic in $H$;\label{cond:geodesics-in-minus-1}
            \item $|R|\geq \rho$ for any $R\in \mathcal R$; and \label{C2}
            \item the length of any $\eps$--piece contained in any word $R\in \mathcal R$ is smaller than $\mu|R|$.\label{C3}
        \end{enumerate}
\end{defn}

Suppose $H$ is a group defined by $H=\langle S\mid \mathcal O\rangle$. Given a symmetrized set of words $\mathcal R$, we consider the quotient group 
\begin{equation}\label{eqn:gpH_1}
    H_1=\langle H\mid \mathcal R\rangle = \langle S \mid \mathcal O\cup \mathcal R\rangle.
\end{equation}

A cell in a van Kampen diagram over \eqref{eqn:gpH_1} is called an \textit{$\mc R$--cell} (respectively, an \textit{$\mc O$--cell}) if its boundary label is a word in $\mc R$ (respectively, $\mc O$).  

Let $\Delta$ be a van Kampen diagram over \eqref{eqn:gpH_1}, let $q$ be a subpath of its boundary $\partial \Delta$, and let $\Pi$ be an $\mc R$--cell of $\Delta$. Suppose that there exists a simple closed path $p=s_1q_1s_2q_2$ in $\Delta$, where $q_1$ and $q_2$ are subpaths of $\partial \Pi$ and $q$, respectively, and $\max\{\abs{s_1},\abs{s_2}\}\leq \eps$ for some constant $\eps$. Let $\Gamma$ denote the subdiagram of $\Delta$ bounded by $p$. If $\Gamma$ contains no $\mc R$--cells, then $\Gamma$ is an \emph{$\eps$--contiguity subdiagram} of $\Pi$ to $q$. We note that this is a special case of the definition of an $\eps$--contiguity diagram in \cite{OlshankiiOsinSapir:lacunary}, where $q_2$ is allowed to be a subpath of the boundary of another $\mc R$--cell of $\Delta$. This situation will not occur in this paper, so we focus on the special case that $q_2\subseteq \partial \Delta$. The subpaths $q_1,q_2$ are called \emph{contiguity arcs} of $\Gamma$, and to distinguish them, we call $q_1\subset \partial \Pi$ the \emph{interior contiguity arc} and $q_2\subset \partial \Delta$  the \emph{boundary contiguity arc}.  The ratio $\abs{q_1}/\abs{\partial \Pi}$ is called the \emph{contiguity degree} of $\Pi$ to $\partial \Delta$ and is denoted by $(\Pi,\Gamma, q_2)$.

The following lemmas from \cite{OlshankiiOsinSapir:lacunary} give control over the contiguity degree of an $\mathcal R$--cell $\Pi$ when $H$ is a hyperbolic group\footnote{We note that in \cite{OlshankiiOsinSapir:lacunary}, the set $\mathcal O$ in the definition of the group $H$ consists of \textit{all} relators of $H$, not just the defining ones. The Lemmas~\ref{lem:OOSLem4.4} and \ref{lem:OOSLem4.6} still hold when $\mathcal O$ is the set of defining relators.}. 
\begin{lem}[{\cite[Lemma~4.4]{OlshankiiOsinSapir:lacunary}}]\label{lem:OOSLem4.4}
    Suppose that the group $H$ is hyperbolic.  Let $\mathcal R$ be a set of words that are geodesic in $H$, let $\Delta$ be a diagram over \eqref{eqn:gpH_1}, and let $q$ be a subpath of $\partial \Delta$ whose label is geodesic in $H_1$.  Then for any $\eps \geq 0$, no $\mathcal R$--cell in $\Delta$ has an $\eps$--contiguity subdiagram $\Gamma$ to $q$ such that $(\Pi,\Gamma,q)>1/2 + 2\eps/|\partial \Pi|$.
\end{lem}

\begin{lem}[{\cite[Lemma~4.6]{OlshankiiOsinSapir:lacunary}}]\label{lem:OOSLem4.6}
    Suppose $H = \langle S\mid \mathcal O\rangle$ is a $\delta$--hyperbolic group, $0<\mu\leq 0.01$, and $\rho$ is large enough (it suffices to choose $\rho>10^6\eps/\mu$).  Let $H$ be given by a presentation
    \[
    H_1=\langle H\mid \mathcal R\rangle = \langle S\mid \mathcal O\cup \mathcal R\rangle
    \]
    as in \eqref{eqn:gpH_1}, where $\mathcal R$ is a finite symmetrized set of words in $S^{\pm 1}$ satisfying the $C(\eps,\mu,\rho)$--condition.  Then the following statements hold.
        \begin{enumerate}
            \item Let $\Delta$ be a minimal disk diagram over \eqref{eqn:gpH_1}. Suppose that $\partial \Delta=q^1\cdots q^t$, where the labels of $q^1,\dots, q^t$ are geodesic in $H$ and $t\leq 12$.  Then, provided that $\Delta$ has an $\mathcal R$--cell, there exists an $\mathcal R$--cell $\Pi$ in $\Delta$ and disjoint $\eps$--contiguity subdiagrams $\Gamma_1,\dots, \Gamma_t$ (some of them may be absent) of $\Pi$ to $q^1,\dots, q^t$ respectively such that 
            \[
            (\Pi,\Gamma_1,q^1)+\cdots +(\Pi, \Gamma_t,q^t)>1-23\mu.
            \]
            \item $H_1$ is a $\delta_1$--hyperbolic group with $\delta_1\leq 4\max\{|R|\mid R\in\mathcal R\}$.
        \end{enumerate}
\end{lem}

\if0
\begin{defn}
    Let $\alpha=0.01$ and $K=10^6$.  The presentation 
    \[
    \pres = \gpres 
    \]
    of a group $G$ satisfies the \emph{graded small cancellation condition} if the following conditions hold for some sequences $(\eps_n)_n$, $(\mu_n)_n$, and $(\rho_n)_n$ of positive real numbers ($n=1,2,\dots$).
        \begin{enumerate}[(Q0)]
            \item The group $G_0=\langle S\mid R_0\rangle$ is $\delta_0$--hyperbolic for some $\delta_0$.
            \item For every $n\geq 1$, $R_n$ satisfies $C(\eps_n,\mu_n,\rho_n)$ over $G_{n-1}=\langle S\mid \cup_{i=0}^{n-1} R_i\rangle$.
            \item $\mu_n=\mathcal{o}(1)$, $\mu_n\leq \alpha$, and $\mu_n\rho_n>K\eps_n$ for any $n\geq 1$.
            \item $\eps_{n+1}>\max\{|r|\mid r\in R_n\}=\mathcal{O}(\rho_n)$.
        \end{enumerate}
\end{defn}
\fi

\subsection{Expanding graded small cancellation}\label{sec:ExpandingGSC}
In this section, we define an expanding graded small cancellation group and investigate the basic properties of such groups.  The notion of expanding graded small cancellation depends on a choice of a sublinear function that is unbounded and non-decreasing.

\begin{defn}[Expanding graded small cancellation]\label{def:ExpandingGSC}

    A presentation 
    \[
    \pres = \gpres 
    \]
    of a group $G$ satisfies the $(f,\mu, (\eps_n)_n)$--\emph{expanding graded small cancellation condition} if the following conditions hold for an unbounded and non-decreasing sublinear function $f$, a constant $\mu \leq 0.001$, and a non-decreasing sequence $(\eps_n)_{n}$ of positive real numbers with $\eps_1 = 0$.
        \begin{enumerate}[(T1)]
            \item For every $n\geq 1$, the set $R_n$ satisfies the $C(\eps_n,\mu,10^6\eps_n/\mu+1)$ condition over $G_{n-1}:=\langle S\mid \bigcup_{i=0}^{n-1} R_i\rangle$, where $R_0:=\emptyset$. \label{T:C-cond}
            \item For every $n\geq 1$, $r_{n+1}\in R_{n+1}$ and $r_n\in R_n$, we have $f(\abs{r_{n+1}}) > \eps_{n+1} \geq 4\abs{r_n}$. \label{T2}
            \item \label{cond:T4} Every relator $r\in R$ has a cyclic subsegment $r_f$ with $\abs{r_f}\leq \abs{r} -  f(\abs{r})$ such that $\abs{r_f\cap r'}\leq \abs{r'}/100$ for all $r'\in \bigcup_{i=1}^\infty R_i$.
            \item \label{cond:T5} There is the following compatibility between $f$ and $\mu$: $f(x)\leq \mu x/1000$ for all $x$.
        \end{enumerate} 
\end{defn}

\begin{defn}[rank]\label{def:rank}
    Let $\gpres = \pres$ be a presentation satisfying the $(f,\mu, (\eps_n)_n)$--expanding graded small cancellation condition. The \emph{rank} of a relator $r\in R$ is the index $i$ for which $r\in R_i$.
\end{defn}

We will be interested in reduced words whose longest common subword with a relator of rank at most $i$ is small compared to the length of that relator. 
\begin{defn}\label{defn:small-intersection}
    Let $\gpres$ be a presentation satisfying the $(f,\mu, (\eps_n)_n)$--expanding graded small cancellation condition, and let $w$ be a word over $S$. If $w$ is reduced and the length of the longest common subword of $w$ and $r$ is at most $|r|/100$ for all $r\in \bigcup_{j=1}^iR_j$ for some $i$, then we say that $w$ is \textit{$i$--good}.  A path is $i$--good if its label is an $i$--good word. %If a reduced word $w$ (resp. path) is $i$--good for all $i$, we say it is \emph{$\infty$-good}.
\end{defn}

We now list some direct consequences of the definition of an expanding graded small-cancellation presentation.

\begin{lem}
    If $\gpres = \pres$ is a presentation satisfying the $(f,\mu, (\eps_n)_n)$--expanding graded small cancellation condition, then the following hold.
    \begin{enumerate}[(P1)]
    \item \label{P:hyperbolicity} For every $n\geq 0$, the group $G_n$ is $\eps_{n+1}$--hyperbolic.
    \item \label{P:smaller-rank-shorter-relator} If $r, r'\in R$ are relators such that the rank of $r$ is less than the rank of $r'$, then $\abs{r} < \mu \abs{r'}$. 
    \item \label{P:small-cancellation-on-level} If embedded relators $r\neq  r'\in R$ have the same rank, then their intersection has length less than $\mu \abs{r}$. 
    \item \label{P:good-segments-of-relators} For every $n\geq 1$, each relator $r\in  R_n$ has a cyclic subsegment $r_f$ of length at least $\abs{r} - f(\abs{r})$ which is $(n-1)$--good.
    \item \label{P:smalleps} For every $n\geq 1$ and $r\in R_n$, we have that $\abs{r} >  1000 \eps_n /\mu \geq 1000 \eps_n$.
\end{enumerate}
\end{lem}

\begin{proof}

Property (P\ref{P:hyperbolicity}) holds by induction. Namely, $G_0$ is a free group and hence $0$--hyperbolic. If $G_{n-1}$ is $\eps_n$--hyperbolic, then $G_{n}$ is $\eps_{n+1} \geq 4\max\{\abs{r_n}\mid r_n\in R_n\}$ hyperbolic by Lemma~\ref{lem:OOSLem4.6}. 

   Property (P\ref{P:small-cancellation-on-level}) holds because $G_n$ satisfies the $C(\eps_n, \mu, 10^6\eps_n/\mu +1)$ condition over $G_{n-1}$ for all $n\geq 1$. 
   
   Property (P\ref{P:good-segments-of-relators}) follows directly from (T\ref{cond:T4}) and the definition of $i$--good. 

   For property (P\ref{P:smalleps}), the case $n = 1$  follows as $\eps_1 = 0$; see Definition~\ref{def:ExpandingGSC}.  For $n\geq 2$, property (P\ref{P:smalleps}) follows by combining the left inequality of (T\ref{T2}) and (T\ref{cond:T5}). 

   For property (P\ref{P:smaller-rank-shorter-relator}), let $n$ be the rank of $r'$. By (P\ref{P:smalleps}) we have that $\abs{r'}\geq 1000\eps_n/\mu$. Since the rank of $r$ is less than $n$, we have that $\abs{r} < \eps_n$. The statement follows by combining the two inequalities. 
\end{proof}

The next proposition further investigates the properties of $i$--good paths.  
\begin{prop}\label{prop:good-geodesics}
    Let $G = \gpres$ satisfy the  $(f,\mu, (\eps_n)_n)$--expanding graded small-cancellation condition. The following hold for all $i\geq 0$.
    \begin{enumerate}[(1)]
        \item \label{prop1:geo}
        An $i$--good path $\gamma$ is the unique geodesic in $G_i$ between its endpoints.
        \item \label{prop2_contiguity-diagrams}
        Let $\gamma\ast\alpha_1\ast\beta\ast\alpha_2$ be a loop in $G_{i}$ such that $\gamma$ is $i$--good and $\alpha_1, \alpha_2$ and $\beta$ are geodesics in $G_{i-1}$, where $G_{-1}$ is defined to be $G_0$. Then either $\abs{\gamma}\leq 5(\abs{\alpha_1}+\abs{\alpha_2})$ or there exists a subgeodesic $\gamma'\subset \gamma \cap \beta$ of length $\abs{\gamma'}\geq \abs{\gamma} - 5\left(\abs{\alpha_1} +  \abs{\alpha_2}\right)$.
    \end{enumerate}
\end{prop}

\begin{proof}
    We will prove both statements simultaneously by induction on $i$. For $i= 0$, the statements follow because $G_0$ is a free group. Fix $i\geq 1$, and assume that \eqref{prop1:geo} and \eqref{prop2_contiguity-diagrams} hold for $i-1$.

    We first prove \eqref{prop2_contiguity-diagrams} for $i$. By possibly passing to a subsegment of $\beta$, we may assume that $\beta^-$ and $\beta^+$ are points on $\beta$ closest to $\gamma^+$ and $\gamma^-$ respectively. 
    
    Let $\Delta$ be a minimal diagram over $G_{i}$ with $\partial \Delta = \gamma\ast\alpha_1\ast\beta\ast\alpha_2$. If $\Delta$ does not contain an $R_i$--cell, then $\gamma\ast \alpha_1\ast \beta \ast \alpha_2$ is a loop in $G_{i-1}$ and hence satisfies all the assumptions of \eqref{prop2_contiguity-diagrams} for $i-1$. The statement follows from the induction hypothesis that \eqref{prop2_contiguity-diagrams} holds for $i-1$. From now on, we assume that $\Delta$ contains an $R_i$--cell and aim to show that $\abs{\gamma}\leq 5(\abs{\alpha_1}+ \abs{\alpha_2})$ which would conclude the proof of \eqref{prop2_contiguity-diagrams} for $i$. 
    
    Since $\gamma$ is a geodesic in $G_{i-1}$ by \eqref{prop1:geo} for $i-1$, Lemma~\ref{lem:OOSLem4.6} yields an $R_i$--cell $\Pi$ of $\Delta$ and disjoint $\eps_i$--contiguity diagrams $\Gamma_\gamma, \Gamma_{\alpha_1}, \Gamma_{\alpha_2}$ and $\Gamma_\beta$ of $\Pi$, the sum of whose contiguity degrees is at least $1 - 23\mu$. As in Lemma~\ref{lem:OOSLem4.6}, not all such diagrams have to exist; we consider the contiguity degree of those that do not exist to be 0.

    \begin{claim}\label{claim:cont-degree}
        The contiguity degree of $\Gamma_\gamma$ is at most $1/10$.
    \end{claim}

    \begin{figure}
        \centering
        \includegraphics[width=0.5\linewidth]{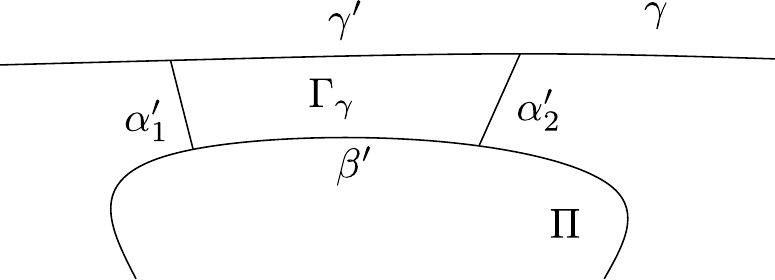}
        \caption{The contiguity diagram $\Gamma_\gamma$.}
        \label{fig:cont0}
    \end{figure}
    \begin{claimproof}
        Let $\gamma'\subset \gamma$ and $\beta'\subset \partial \Pi$ be the boundary and interior contiguity arcs of $\Gamma_\gamma$, respectively. Let $\alpha_1'$ and $\alpha_2'$ be geodesics in $G_{i-1}$ between the respective endpoints of $\gamma'$ and $\beta'$, as depicted in Figure~\ref{fig:cont0}, so that $\abs{\alpha_j'}\leq \eps_i$ for $j = 1, 2$. Applying \eqref{prop2_contiguity-diagrams} to $\gamma'\ast \alpha_1'\ast \beta'\ast \alpha_2'$ for $i-1$ yields two possibilities. If  $\abs{\gamma'}\leq 5(\abs{\alpha_1'}+\abs{\alpha_2'})\leq 10\eps_i$, then the triangle inequality and (P\ref{P:smalleps}) imply that  $(\Pi, \Gamma_{\gamma}, \gamma)\abs{\partial \Pi}=\abs{\beta'}\leq 12\eps_i< 12\abs{\partial\Pi}/1000$. Otherwise, there exists a subgeodesic $\gamma''$ of $\gamma'\cap \beta'$ with 
        \begin{align}\label{eq:0}
            \abs{\gamma''}\geq \abs{\gamma'} - 10\eps_i\geq \abs{\beta'} - 12\eps_i  =  (\Pi, \Gamma_{\gamma}, \gamma)\abs{\partial \Pi} - 12 \eps_i.
        \end{align}
        Here we used the triangle inequality for the second step and the definition of the contiguity degree for the third step. Since $\gamma$ is $i$-good and $\gamma''$ is a subsegment of both $\gamma$ and $\beta'\subset \partial \Pi$, 
        \begin{align}\label{eq0.1}
            \frac{\abs{\partial \Pi}}{100}\geq \abs{\gamma''}.
        \end{align}
        Inequalities \eqref{eq:0} and \eqref{eq0.1}, combined with the fact that $\eps_i\leq \abs{\partial \Pi}/1000$ by (P\ref{P:smalleps}), imply $\abs{\partial \Pi}/100\geq  (\Pi, \Gamma_{\gamma}, \gamma)\abs{\partial \Pi} - 12 \eps_i \geq (\Pi, \Gamma_{\gamma}, \gamma)\abs{\partial \Pi} - 12\abs{\partial\Pi}/1000$.  Hence, for either possibility, we  have  $(\Pi, \Gamma_{\gamma}, \gamma)\leq 1/10$. 
     \end{claimproof}
    
    \smallskip
     By Lemma~\ref{lem:OOSLem4.4}, the contiguity degree of any contiguity diagram is at most $1/2  + 2\eps_i/\abs{\partial \Pi}$, and by Lemma~\ref{lem:OOSLem4.6}, the sum of all contiguity degrees is at least $1-23\mu$. Together with Claim~\ref{claim:cont-degree}, this implies 
    \begin{align}\label{eq:degrees}
          (\Pi, \Gamma_{\alpha_1}, \alpha_1) + (\Pi, \Gamma_{\alpha_2}, \alpha_2) \geq  \frac{1}{2}-\frac{1}{10}- 23\mu - 2\frac{\eps_i}{\abs{\partial\Pi}}\geq \frac{1}{4}+6\frac{\eps_i}{\abs{\partial \Pi}}.
    \end{align}

    Here we used that $\mu\leq 1/230$ by Definition~\ref{def:ExpandingGSC} and $\eps_i\leq \abs{\partial \Pi}/1000$ by (P\ref{P:smalleps}).

    \begin{claim}\label{cliam:sides-exist}
        The contiguity diagrams $\Gamma_{\alpha_1}$ and $\Gamma_{\alpha_2}$ both exist.
    \end{claim}
    \begin{claimproof}
    Without loss of generality it suffices to prove that $\Gamma_{\alpha_1}$ exists. Assume by contradiction that it does not. Then,
    \begin{align}\label{eq:deg2}
        (\Pi, \Gamma_{\alpha_2}, \alpha_2) + (\Pi, \Gamma_\beta, \beta) \geq 1 - \frac{1}{10} - 23\mu > \frac{3}{4} + 2\frac{\eps_i}{\abs{\partial \Pi}}.
    \end{align}

    \begin{figure}
        \centering
        \includegraphics[width=0.3\linewidth]{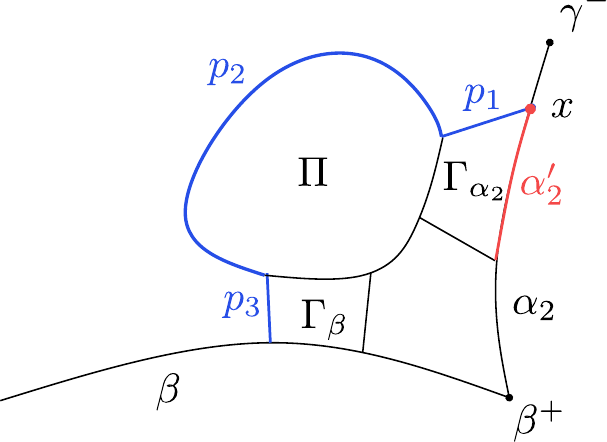}
        \caption{A short path $p = p_1 \ast p_2\ast p_3$ from $x$ to $\beta$.}
        \label{fig:cont1}
    \end{figure}
    In particular, both $\Gamma_\beta$ and $\Gamma_{\alpha_2}$ have to exist by Lemma~\ref{lem:OOSLem4.4}. Let $x$ be the point on $\alpha_2\cap \partial \Gamma_{\alpha_2}$ farthest away from $\beta$. Let $\alpha_2'$ be the boundary contiguity arc of $\Gamma_{\alpha_2}$. By the triangle inequality, \eqref{eq:degrees} and the non-existence of $\Gamma_{\alpha_1}$, we have 
    \begin{align*}
        d(x, \beta^+) \geq \abs{\alpha_2'} \geq (\Pi, \Gamma_{\alpha_2}, \alpha_2)\abs{\partial \Pi} - 2\eps_i\geq \frac{\abs{\partial \Pi}}{4}.
    \end{align*}
    Since the $\eps_i$--contiguity diagrams $\Gamma_\beta$ and $\Gamma_{\alpha_2}$ exist, there exists a path $p = p_1\ast p_2\ast p_3$ from $x$ to $\beta$ where $\abs{p_j}\leq \eps_i$ for $j=1, 3$ and $p_2$ is a subsegment of $\partial \Pi$ whose interior is disjoint from $\partial \Gamma_{\beta}$ and $\partial \Gamma_{\alpha_2}$. This is depicted in Figure~\ref{fig:cont1}. In particular by \eqref{eq:deg2},
    \begin{align*}
        \abs{p} <2\eps_i + \left(\frac{1}{4}-2\frac{\eps_i}{\abs{\partial \Pi}}\right)\abs{\partial \Pi} = \frac{\abs{\partial \Pi}}{4}.
    \end{align*}
    This contradicts our assumption that $\beta^+$ is a closest point on $\beta$ to $\gamma^-$ and hence is a closest point to $x$ on $\beta$.
    \end{claimproof}

    \smallskip
    \begin{figure}
        \centering
        \includegraphics[width=0.5\linewidth]{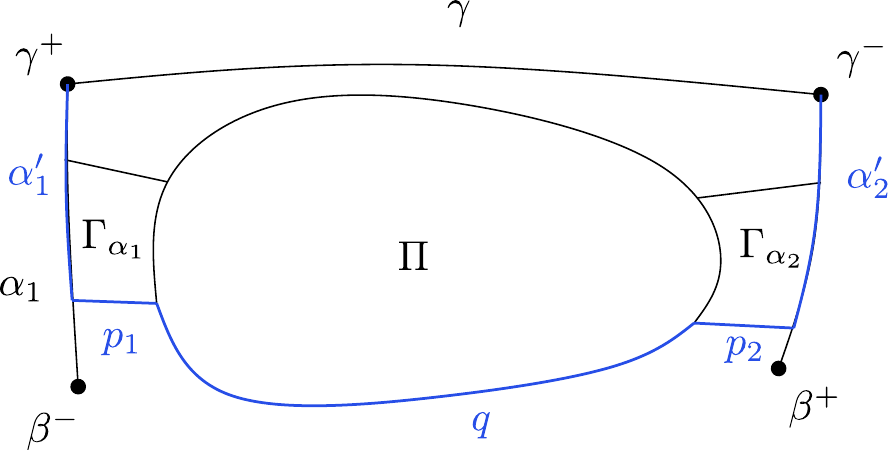}
        \caption{A short path $p = \alpha_1'\ast p_1 \ast q \ast p_2 \ast \alpha_2'$ from $\gamma^+$ to $\gamma^-$.}
        \label{fig:cont2}
    \end{figure}

    By Claim~\ref{cliam:sides-exist}, $\Gamma_{\alpha_1}$ and $\Gamma_{\alpha_2}$ both exist. In particular, as depicted in Figure~\ref{fig:cont2}, there exists a path $p = \alpha_1'\ast p_1 \ast q \ast p_2 \ast \alpha_2'$ from $\gamma^+$ to $\gamma^-$ such that $\alpha_1'$ and $\alpha_2'$ are subsegments of $\alpha_1$ and $\alpha_2$, $p_1$ and $p_2$ have length at most $\eps_i$ and $q$ is a subsegment of $\partial \Pi$. Consequently
    \begin{align}\label{eq:deg3}
        \abs{\gamma} \leq \abs{p}\leq \abs{\alpha_1} + \eps_i + \abs{\partial \Pi}+\eps_i +\abs{\alpha_2} \leq 5 (\abs{\alpha_1} + \abs{\alpha_2}).
    \end{align}

    Here we used \eqref{eq:degrees} which implies that $4(\abs{\alpha_1} +\abs{\alpha_2}) \geq \abs{\partial \Pi} + 8 \eps_i$. Inequality \eqref{eq:deg3} concludes the proof of \eqref{prop2_contiguity-diagrams} for $i$.

    Finally, we prove \eqref{prop1:geo} for $i$. Let $\gamma$ be $i$-good, and let $\beta$ be a geodesic in $G_i$ from $\gamma^+$ to $\gamma^-$. Applying \eqref{prop2_contiguity-diagrams} for $i$ to the loop $\gamma\ast \beta$ yields that $\gamma$ is a subsegment of $\beta$. Since $\beta$ is a geodesic, this implies $\gamma = \beta$ and concludes the proof.
\end{proof}

We next turn our attention to contiguity diagrams in a group with an expanding graded small cancellation presentation.  The following can be thought of a version of Greendlinger's lemma.
\begin{lem}[Upgraded contiguity diagrams]\label{lem:upgraded-contiguity-diagrams}
        
        Let $G = \pres = \gpres$ be a group satisfying the $(f, \mu, (\eps_n)_n)$--expanding graded small cancellation condition. 
        
        Let $\Delta$ be a minimal disk diagram over $G_n$. Suppose that $\partial \Delta=q^1\cdots q^t$, where the labels of $q^1,\dots, q^t$ are geodesics in $G_{n-1}$ and $t\leq 12$. Then either $\Delta$ does not contain any cells, in which case $\partial \Delta$ freely reduces to the trivial word, or there exists an $R_m$--cell $\Pi$ of $\Delta$ for some $m\leq n$ such that the following hold.
        \begin{enumerate}
            \item Every cell $\Pi'$ of $\Delta$ is an $R_i$--cell for some $i\leq m$.\label{cond:minimality}
            \item There exist disjoint subsegments $p_1, \ldots, p_{2t}$ of $\partial \Pi$ such that $\abs{p_1}+\ldots +\abs{p_{2t}}\geq \abs{\partial \Pi}(1 - 24 \mu )$ and such that $p_\ell$ is a subsegment of $q^{\ceil{\ell/2}}$ for all $1\leq \ell \leq 2t$.\label{cond:subsegments}
        \end{enumerate}
\end{lem}

\begin{proof}
    Let $\Delta$ be as in the statement. If $\Delta$ does not contain any cells, then the statement holds, so assume that $\Delta$ contains at least one cell. Let $m$ be the maximal index for which $\Delta$ contains an $R_m$--cell. Note that $\Delta$ is a diagram in $G_n$, and hence $m\leq n$. This shows that the first conclusion holds.  Moreover, since $q^1, \ldots , q^t$ are geodesics in $G_{n-1}$, they are geodesic in $G_{m-1}$.
    
    Since $G = \gpres$ satisfies the $(f, \mu , (\eps_n)_n)$--expanding graded small cancellation condition, the group $G_m$ satisfies the $C(\eps_m, \mu, 1006\eps_m/\mu +1)$ condition over $G_{m-1}$, and thus we can apply Lemma~\ref{lem:OOSLem4.6} to obtain an $R_m$--cell $\Pi'$ and disjoint $\eps_m$--contiguity diagrams $\Gamma_1, \ldots, \Gamma_t$ (some of which may be absent) with 
    \[
        (\Pi',\Gamma_1,q^1)+\cdots +(\Pi', \Gamma_t,q^t)>1-23\mu.
    \]

    Let $1\leq k \leq t$ be an index such that the contiguity diagram $\Gamma_k$ exists, and let $\alpha\subset \partial \Pi'$ and $\beta\subset q^k$ be its interior and boundary contiguity arcs, respectively. By (P\ref{P:good-segments-of-relators}), $\alpha$ can be written as $\alpha_1 \ast \alpha'\ast \alpha_2$, where $\alpha_1$ and $\alpha_2$ are $(m-1)$--good and $\abs{\alpha'}\leq f(\abs{\partial\Pi'})$; some of these paths may be empty. Since $\Gamma_k$ is an $\eps_m$--contiguity diagram, its boundary $\partial \Gamma_k$ can be viewed as a quadrangle in the $\eps_m$--hyperbolic group $G_{m-1}$. In particular, there are points $x_1, x_2$ on $\partial \Gamma_k \setminus \alpha$ which satisfy $d(\alpha_1^+, x_1)\leq 2 \eps_m$ and $d(\alpha_2^-, x_2)\leq 2\eps_m$. For $\ell = 1, 2$, if $\abs{\alpha_\ell}> 3 \eps_m$, then by the triangle inequality, these points have to lie on $\beta$. Thus, in any case, we can write $\beta = \beta_1\ast \beta'\ast \beta_2$, where for $\ell = 1, 2$, the endpoints of $\beta_\ell$ and $\alpha_\ell$ have distance at most $4\eps_m$. See Figure~\ref{fig:upgraded-contiguity-diagrams}. Proposition~\ref{prop:good-geodesics}\eqref{prop2_contiguity-diagrams} applied with $i = m-1$ implies that  $\alpha_\ell$ and $\beta_\ell$ have a common subsegment $\gamma^k_\ell$ of length at least $\abs{\alpha_\ell} - 25\eps_m$ for $\ell = 1, 2$; if $\abs{\alpha_\ell}<25\eps_m$, then we take $\gamma^k_\ell$ to be empty. Hence $\abs{\gamma^k_1} + \abs{\gamma^k_2}\geq \abs{\alpha} - \abs{\alpha'} - 50\eps_m \geq \abs{\partial \Pi'}(\Pi', \Gamma_k, q^k) - 50\eps_m - f(\abs{\partial \Pi'})$.
    \begin{figure}
        \centering
        \includegraphics[width=0.7\linewidth]{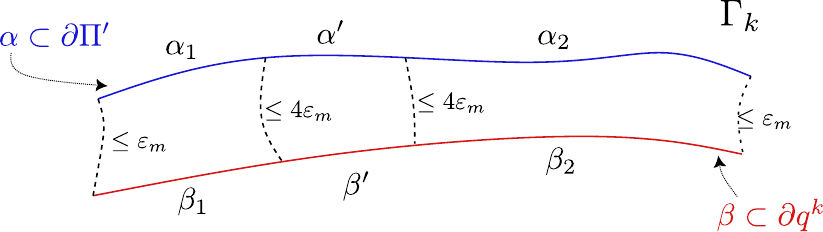}
        \caption{The $\eps_j$--continguity diagram $\Gamma_k$.}
        \label{fig:upgraded-contiguity-diagrams}
    \end{figure}

    Since this holds for all $k$ for which the $\eps_m$--contiguity diagram $\Gamma_k$ exists, and the contiguity diagrams are disjoint, we have that 
    \begin{align*}
       \sum_{k = 1 }^t\left(\abs{\gamma^k_1}+\abs{\gamma^k_2}\right) &\geq \abs{\partial \Pi'}\left((\Pi',\Gamma_1,q^1)+\cdots +(\Pi', \Gamma_t,q^t)\right) - 12\cdot 50\cdot\eps_m - 12 f(\abs{\partial \Pi'})\\
        &\geq \abs{\partial \Pi'}(1 - 23 \mu) - 12\cdot50 \cdot\eps_m - 12 f(\abs{\partial \Pi'}).
    \end{align*}

    By (P\ref{P:smalleps}), $ 12\cdot 50 \cdot \eps_m \leq 3\abs{\partial \Pi'}\mu/4$, by (T\ref{cond:T5}), $12f(\abs{\partial \Pi'})\leq \abs{\partial \Pi'}\mu/4$ and therefore, item \ref{cond:subsegments} holds for $\Pi = \Pi'$, concluding the proof.  
\end{proof}

 We call  a cell $\Pi$ satisfying conditions 1 and 2 of Lemma~\ref{lem:upgraded-contiguity-diagrams} a \emph{Greendlinger cell} of $\Delta$. 
\begin{remark}\label{rem:greendlinger-in-G}
    Suppose $q^1, \ldots , q^t$ are geodesics in $G$ such that the path $ p = q^1\cdots q^t$ is a loop in $G$. Note that the paths $q^k$ are geodesics in $G_i$ for all $i$. Choose $i$ such that the path $p$ is a loop in $G_i$, and let $\Delta$ be a minimal diagram over $G_i$ with $\partial \Delta= p$. By Lemma~\ref{lem:upgraded-contiguity-diagrams}, if $t\leq 12$, then $\Delta$ contains a Greendlinger cell. 
\end{remark}

The next lemma compares the length of the boundary of a minimal disk diagram over $G_i$ to the length of any cell of the diagram.
\begin{lem}\label{lem:boundary-proportion}
        Let $C\geq 1$ be a constant, and let $G = \pres = \gpres$ be an $(f, \mu, (\eps_n)_n)$--expanding graded small cancellation group satisfying $\abs{r}\leq C\abs{r'}$ for all $r, r'\in R$ of the same rank. 
        Let $\Delta$ be a minimal disk diagram over $G_i$ (or $G$). Suppose that $\partial \Delta=q^1\cdots q^t$, where the labels of $q^1,\dots, q^t$ are geodesics in $G_{i-1}$ (or $G$) and $t\leq 12$. Then for all cells $\Pi$ of $\Delta$, 
        \begin{align*}
          C \cdot K\cdot  \abs{\partial \Delta}\geq \abs{\partial \Pi},  
        \end{align*}
        where $K = 1/(1 - 24\mu)$.
\end{lem}

\begin{proof}
    If $\Delta$ has no cells, the statement vacuously holds, so assume that $\Delta$ has at least one cell. By Lemma~\ref{lem:upgraded-contiguity-diagrams}, there is a Greendlinger cell $\Pi'$ of $\Delta$, which implies that
    \begin{align}\label{eq1}
        \abs{\partial \Delta} \geq (1 - 24 \mu)\abs{ \partial \Pi'} = 1/K \abs{\partial \Pi'}.
    \end{align}
    Let $\Pi$ be any cell of $\Delta$. Since Greendlinger cells have maximal rank, the rank of $\Pi$ is at most the rank of $\Pi'$. If $\Pi$ has the same rank as $\Pi'$, then by assumption $\abs{\Pi}\leq C\abs{\Pi'}$. If the rank of $\Pi$ is less than the rank of $\Pi'$, then $\abs{\Pi}\leq \abs{\Pi'}$ by (P\ref{P:smaller-rank-shorter-relator}).  In either case,  the statement follows from \eqref{eq1}. 
\end{proof}

We next show that images of relators are isometrically embedded in $G$. 
\begin{lem}\label{lem:relators-are-isometrically-embedded}
    Let $G = \pres = \gpres$ be a $(f, \mu, (\eps_n)_n)$--expanding graded small cancellation group. The images of all relators $r\in R$ are isometrically embedded in $G$.
\end{lem}

\begin{figure}
    \centering
    \includegraphics[width=0.2\linewidth]{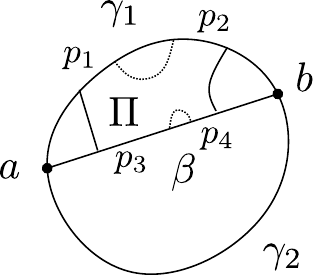}
    \caption{The embedded relator $\gamma$ subdivided into $\gamma_1$ and $\gamma_2$ is isometrically embedded.}
    \label{fig:isom-relators}
\end{figure}
\begin{proof}
    Fix $r\in R_i$ for some $i$. By definition, $r$ labels a geodesic in $G_{i-1}$. Toward a contradiction, assume that $r$ is not isometrically embedded in $G$, and let $j\geq i$ be the smallest index such that the embedded relator $\gamma$ labelled by $r$ is not isometrically embedded in $G_j$. Let $a, b$ be two points on $\gamma$, and let $\gamma_1$ and $\gamma_2$ be the two paths with $\abs{\gamma_1}\leq \abs{\gamma_2}$ from $a$ to $b$ that arise by following along $\gamma$ and such that the geodesic $\beta$ in $G_j$ from $b$ to $a$ satisfies $\abs{\beta} < \abs{\gamma_1}$. This is depicted in Figure~\ref{fig:isom-relators}.
    
    Let $\Delta_1$ be a minimal diagram over $G_j$ with $\partial \Delta_1 = \gamma_1\ast \beta$; by the choice of $a$ and $b$, the diagram $\Delta_1$ cannot be empty. Moreover, the sides $\gamma_1$ and $\beta$ of $\Delta_1$ are geodesics in $G_{j-1}$, and hence $\Delta_1$ has a Greendlinger cell $\Pi$ by Lemma~\ref{lem:upgraded-contiguity-diagrams}. By the minimality of $j$, the rank of the Greeendlinger cell $\Pi$ has to be equal to $j$. Since $\Pi$ is a Greendlinger cell of $\Delta_1$, there exist disjoint subsegments $p_1,p_2,p_3,p_4$ of $\partial \Pi$ such that $\abs{p_1}+\abs{p_2}+\abs{p_3}+\abs{p_4}\geq \abs{\partial\Pi}(1-24\mu)$ and $p_1,p_2$ are subsegments of $\gamma_1$ while $p_3,p_4$ are subsegments of $\beta$. Since $\beta$ is a geodesic, we must have $\abs{p_3}+\abs{p_4}\leq \abs{\partial\Pi}/2$, and hence 
    \begin{align}\label{eq:geodesic}
        \abs{p_1} + \abs{p_2} \geq \abs{\partial \Pi} \left(\frac{1}{2} - 24\mu\right).
    \end{align}
    
    If $j > i$, then by (P\ref{P:smaller-rank-shorter-relator}), $\abs{\gamma_1}\leq \mu \abs{\Pi}$. In particular $\abs{p_1}+\abs{p_2}\leq \mu \abs{\Pi}$, contradicting \eqref{eq:geodesic}.  On the other hand, if $j = i$ and $\Pi\neq \gamma$, then by (P\ref{P:small-cancellation-on-level}), $\abs{p_1}$ and $\abs{p_2}$ can be at most $\mu\abs{\Pi}$, again contradicting \eqref{eq:geodesic}. Thus assume from now on that $i=j$ and $\Pi = \gamma$. Let $\Delta_2$ be the subdiagram of $\Delta_1$ obtained by removing $\Pi = \gamma$ from $\Delta_1$, so that $\partial \Delta_2 = \gamma_2\ast \beta$.  Note that by (C\ref{cond:geodesics-in-minus-1}), $\gamma_{2}$ is a geodesic in $G_{i-1}$. If $\Delta_2$ were empty, then $\abs{\beta} = \abs{\gamma_2} >\abs{\gamma_1}$, contradicting $\abs{\gamma_1}\leq \abs{\gamma_2}$. In particular, Lemma~\ref{lem:upgraded-contiguity-diagrams} implies that $\Delta_2$ has a Greendlinger cell $\Pi'$.  Since $\gamma_2$ is a geodesic in $G_{i-1}$, such a cell $\Pi'$ cannot have rank less than $i$ and thus has to have rank $i$. Again by Lemma~\ref{lem:upgraded-contiguity-diagrams}, there exist disjoint subsegments $p_1',p_2',p_3',p_4'$ of $\partial \Pi'$ such that $\abs{p_1'}+\abs{p_2'}+\abs{p_3'}+\abs{p_4'}\geq \abs{\partial\Pi'}(1-24\mu)$ and $p_1',p_2'$ are subsegments of $\gamma_2$ while $p_3',p_4'$ are subsegments of $\beta$. As before, $\abs{p_1'}+\abs{p_2'}\geq \abs{\partial \Pi'}\left(1/2 - 24 \mu\right)$. If $\Pi'\neq \gamma$, we have a contradiction with (P\ref{P:small-cancellation-on-level}), while if $\Pi' = \gamma$, we contradict the minimality of $\Delta_1$. 
\end{proof}

We end this section with two technical lemmas that will be useful for proving that groups satisfy the expanding graded small-cancellation condition, which we do in Section~\ref{sec:nonAHgp}. 

\begin{lem}\label{lem:remove-eps-pieces}
    Let $n\geq 1$ and let $G_{n-1} = \langle S\mid \bigcup_{i=1}^{n-1} R_i \rangle$ be an $(f, \mu, (\eps_n)_n)$--expanding graded small cancellation group. Let $Q$ be a symmetrized set of reduced words over $S^{\pm}$. Suppose $Q$ satisfies the following conditions.
    \begin{enumerate}[($\tilde{C}$1)]
        \item all words from $Q$ are geodesics in $G_{n-1}$,\label{tildeC1}
        \item $\abs{q}\geq 10^6\eps_n/\mu +1$ for all words $q\in Q$, \label{tildeC2}
        \item $Q$ satisfies the $C'(\mu/3)$--condition, and \label{tildeC3}
        \item each relator $q\in Q$ has a subsegment $q_g\subset q$ of length at least $\abs{q_g}\geq \abs{q} - f(\abs{q})$ which is $(n-1)$--good.\label{tildeC4}
    \end{enumerate}
    Then $Q$ satisfies the $C(\eps_n, \mu, 10^6\eps_n/\mu +1)$--condition over $G_{n-1}$.
\end{lem}

\begin{proof}
    We need to show that conditions (C\ref{cond:geodesics-in-minus-1})--(C\ref{C3}) hold. Conditions (C\ref{cond:geodesics-in-minus-1}) and (C\ref{C2}) are the same as conditions ($\tilde{C}$\ref{tildeC1}) and ($\tilde{C}$\ref{tildeC2}), so it remains to show (C\ref{C3}) holds. 

    Let $q\in Q$, and let $u\subset q$ be an $\eps_{n}$--piece. That is, there exists $q' = u'v'\in Q$ and words words $y, z$, which we can assume to be geodesic in $G_{n-1}$, of length at most $\eps_n$ such that $yuz(u')^{-1}$ represents the identity in $G_{n-1}$ and such that $q'$ and $yqy^{-1}$ represent different elements in $G_{n-1}$. Let $\ell = \alpha_1\ast \gamma\ast \alpha_2\ast (\gamma')^{-1}$ be a loop in $G_{n-1}$ such that $\alpha_1, \gamma, \alpha_2$ and $\gamma'$ are labeled by $y, u, z$ and $u'$, respectively. The paths $\gamma$ and $\gamma'$ are geodesics by ($\tilde{C}$\ref{tildeC1}) and the paths $\alpha_1, \alpha_2$ are geodesics by the choice of $y$ and $z$. 

    By ($\tilde{C}$\ref{tildeC4}), $\gamma$ has disjoint $(n-1)$--good subpaths $\beta_1, \beta_2$ such that $\abs{\gamma} - \abs{\beta_1} - \abs{\beta_2}\leq f(\abs{q})$. The group $G_{n-1}$ is $\eps_n$--hyperbolic by (P\ref{P:hyperbolicity}). Hence, any point on $\gamma$ is in the $2\eps_n$--neighbourhood of some point in $\alpha_2\ast (\gamma')^{-1}\ast \alpha_1$. Since $\alpha_1, \alpha_2$ have length at most $\eps_n$, any point on $\gamma$, and in particular the endpoints of $\beta_1$ and $\beta_2$ have distance at most $3\eps_n$ from $\gamma'$. For $i = 1, 2$, Proposition~\ref{prop:good-geodesics} yields that either $\abs{\beta_i}\leq 30\eps_n$ or that there is a subpath $\beta_i'$ of $\beta_i\cap \gamma'$ of length at least $\abs{\beta_i} - 30\eps_n$. 

    Assume we are in the latter case. We will obtain an upper bound on $\abs{\beta_i}$ by arguing that the word $b_i'$ labeling $\beta_i'$ is a piece in $Q$.  By construction, we can write $q = x_1b_i'x_2$, $q' = x_1'b_i'x_2'$ such that the word $yx_1(x_1')^{-1}$ represents the identity in $G_{n-1}$. Moreover, if $x_1^{-1}qx_1$ and $x_1'^{-1}q'x_1'$ do not reduce to the same word, then $b_i'$ is a piece, as desired. So assume they do. Then since $y$ and $x_1'x_1^{-1}$ represent the same element in $H$, the words $yqy^{-1}$ and $x_1'x_1^{-1}qx_1(x_1')^{-1}$ represent the same element in $H$. However, $x_1'x_1^{-1}qx_1(x_1')^{-1}$ reduces to the same word as $q'$. Hence $yqy^{-1}$ and $q'$ represent the same word in $H$, which contradicts the definition of $y$. Thus, $b_i'$ is indeed a piece. Using ($\tilde{C}\ref{tildeC3})$ we obtain $\abs{b_i'} < \mu\abs{q}/3$, and so $\abs{\beta_i}\leq 30\eps_n + \mu\abs{q}/3$.

     Hence, in either case, $\abs{u}< 2\mu\abs{q}/3 + 60\eps_n+f(\abs{q}) \leq \mu\abs{q}$. For the last step we used (T\ref{cond:T5}) and (P\ref{P:smalleps}). This concludes the proof of (C\ref{C3}).
\end{proof}

\begin{lem}\label{lem:IgoodShort}
    Let $M> 0$ be a constant. Let $\Delta$ be a minimal diagram over $G_n$ with $\partial\Delta=q_1\cdots q_k$, with $k\leq 6$, such that the sides $q_i$ are geodesics in $G_n$ and the following hold:
        \begin{itemize}
            \item when $1\leq i\leq k-1$, the sides alternate between $n$--good sides whose length is at least $10M$ and \emph{short} sides, that is, sides $q_i$ with $\abs{q_i}\leq M$; and 
            \item the label of $q_1\cdots q_{k-1}$ is a reduced word.
        \end{itemize}   
        Then any point on an $n$--good side that is at least $50M$ from its endpoints lies on $q_k$.
\end{lem}

    \begin{proof}

        We first prove the following useful claim:
        
        \begin{claim}\label{claim:3igood+geod}
            Any diagram whose boundary consists of 3 geodesics, at least two of which are $n$--good, must be empty.
        \end{claim}

        \begin{claimproof}
            Let $\Delta'$ be a diagram as in the assumptions of the claim.  If $\Delta'$ is not empty, then it contains a Greendlinger cell $\Pi'$. Each subsegment of the intersection of an $n$--good side with $\partial \Pi'$ has length at most $\frac{1}{100}\abs{\partial \Pi'}$. If there exists a side that is not $n$--good, it is unique, and its intersection with $\partial \Pi'$ has length at most $\frac12\abs{\partial\Pi'}$. Thus the total intersection of $\partial \Pi'$ with $\partial\Delta'$ has length at most $(4\cdot \frac{1}{100}+\frac{1}{2})\abs{\partial\Pi'}<(1-24\mu)\abs{\partial\Pi'}$. This contradicts \Cref{lem:upgraded-contiguity-diagrams}. 
        \end{claimproof}

    For the following claim, we consider a minimal diagram $\Delta'$ over $G_n$ and a subdiagram $\Delta$, which is also a minimal diagram over $G_n$.  We say that a geodesic $q\subseteq \partial \Delta$ is \emph{short} if $\abs{q}\leq M$; \emph{long} if it is $n$--good and satisfies $\abs{q}>10M$; and \emph{very short} if $q$ is a subsegment of $\partial \Pi'$ for some cell $\Pi'$ of $\Delta' -\Delta$ with $\abs{\partial \Pi'}\leq 8M$.  The lemma follows directly from the following claim applied to $\Delta$ as a subdiagram of itself, in which there are no very short sides.
        
        \begin{claim}\label{claim:Igoodshort-detail}
            Let $k\geq 1$ be an integer, and let $\Delta\subset \Delta'$ be two minimal diagrams over $G_n$.  Suppose $\partial \Delta = q_1\cdots q_k$ such that $q_i$ are (potentially trivial) geodesics, each geodesic $q_i$ for $1\leq i \leq k-1$ is either long, short, or very short, and there is a subsegment $q_R$ of $q_1\cdots q_k$ containing all long sides whose label is reduced. Further suppose the following hold.
            \begin{enumerate}[(D1)]
                \item If for $1\leq i, j \leq k-1$ the geodesics $q_i, q_j$ are not short, then $\abs{i-j} > 1$.
                \item If for $1\leq i < j \leq k-1$ the geodesics $q_i$ and $q_j$ are short, then $i = j+2$ and $q_{i+1}$ is long.\label{alternate}
                \item At most two sides are short.\label{few-short}
            \end{enumerate}
            Then each point $x$ that lies on a long side and has distance at least $50M$ from its endpoints also lies on $q_k$. 
        \end{claim}

        \begin{claimproof}
            The three conditions imply that $k\leq 6$.
            Assume $\Delta = \emptyset$. Since the label of $q_R$ is reduced, $q_R$ has to be a subset of $\partial \Delta- q_R$. Since $q_R$ contains all the long sides, the claim follows. From now on we assume that $\Delta\neq \emptyset$.

            Assume by induction that Claim~\ref{claim:Igoodshort-detail} holds for all diagrams with fewer cells than $\Delta$.
            
            Let $\Pi$ be a Greendlinger cell of $\Delta$, and let $r$ be its rank. By Lemma~\ref{lem:upgraded-contiguity-diagrams}, there exist disjoint subsegments $p_1, \ldots, p_{2k}$ of $\partial \Pi$ with
            \begin{align}\label{eq:greendlinger-detail}
                \abs{p_1}+\cdots +\abs{p_{2k}} \geq \abs{\partial \Pi} (1 - 24\mu)
            \end{align}
             such that $p_{2i}$ and $p_{2i-1}$ are either empty or a subsegment of $q_i$ for $1\leq i \leq k$. 
            
            For $1\leq i \leq k$ we have $\abs{p_{2i}} + \abs{p_{2i-1}}\leq \abs{\partial \Pi}/2$, because $q_i$ is a geodesic. Moreover, if $q_i$ is long, it is $n$--good by definition, and hence $\abs{p_{2i}}+\abs{p_{2i-1}}\leq \abs{\partial \Pi}/50$. 
            
            We divide the proof into two cases, depending on the very short sides.\\

            \textbf{Case 1:} There exist $1\leq i < k$ such that $q_i$ is very short and $\abs{p_{2i}} + \abs{p_{2i-1}}\geq \abs{\partial\Pi}/20$. Let $\Pi_i$ be the cell of $\Delta'-\Delta$ such that $q_i\subseteq \partial \Pi_i$ and $\abs{\partial \Pi_i}\leq 8M$.  By (P\ref{P:smaller-rank-shorter-relator}) and (P\ref{P:small-cancellation-on-level}), this can only happen if the rank $r$ of $\Pi$ is less than the rank of $\Pi_i$. In particular, $q_1\cdots q_k$ can be viewed as a $6$--gon in a $\eps_{r+1}$--hyperbolic space with $\eps_{r+1}\leq f(\abs{\partial\Pi_i}) \leq \abs{\partial \Pi_i}/100\leq M/4$; here the first inequality holds by (T\ref{T2}), the second by (T\ref{cond:T5}), and the last as $q_i$ is very short. In particular, any long side $q_j$ is contained in the $M$--neighborhood of the other sides. Moreover, since we assume in this case that there is a very short side, there are at most two long sides and they are adjacent to a common short side. 

            Let $x_0$ be one of the endpoints of $q_j$, and for $1\leq \ell \leq 4$ let $x_\ell$ be a point on $q_j$ at distance $10M\ell$ from $x_0$. Consider $q_h$, where $h\neq j,k$.  If $q_h$ is short or very short, then, by the triangle inequality, there exists at most one $\ell$ such that $d(x_\ell, q_h)\leq M$. On the other hand, if $q_h$ is the other long side and if $x_\ell$ and $x_{\ell'}$ for $\ell\neq \ell'$ are both at distance at most $M$ from $q_h$, then $q_h$ and $q_j$ intersect by Lemma~\ref{prop:good-geodesics}. In particular, there exists a subdiagram $\Delta ''$ of $\Delta$ whose boundary is labelled by a concatenation of two $n$--good geodesics, one of which is non-trivial, and a short side. By Claim~\ref{claim:3igood+geod}, such a diagram has to be empty. However, the label of its boundary is a subword of the label of $q_1\cdots q_{k-1}$ and is reduced, which implies that $\abs{\partial\Delta''} =0$, a contradiction. 

            We have shown that each side $q_h$ for $h\neq j,k$ has distance more than $M$ to all but at most one $x_0, \ldots, x_4$. Consequently, $d(q_k, x_\ell)\leq M$ for some $0\leq \ell\leq 4$. In paticular, $x_\ell$ is in the $40M$--neighbourhood of one of the endpoints of $q_j$ and in the $M$--neighbourhood of $q_k$.  Repeating the same argument choosing $x_0$ to be the other endpoints of $q_j$, we find a second point that is in the $40M$--neighbourhood of the other endpoint of $q_j$ and in the $M$--neighbourhood $q_k$.  Lemma~\ref{prop:good-geodesics} concludes the proof in this case. \\

            \textbf{Case 2:} For all $1\leq i < k$ for which $q_i$ is very short, we have $\abs{p_{2i}}+\abs{p_{2i-1}}<\abs{\partial\Pi}/20$. Since there are at most 2 short sides, there are at most 3 long sides and at most 3 very short sides. In order to satisfy \eqref{eq:greendlinger-detail}, we must have that the subsegments that lie on $\partial \Pi\cap q_k$ or on $\partial \Pi$ and short sides have combined length at least $\abs{\partial \Pi}(1 - 24\mu - 3/50 - 3/20)\geq 3\abs{\partial \Pi}/4$.  Since $q_k$ is geodesic, at least $\abs{\partial \Pi}/4$ of this length has to lie on short sides, which implies that $\abs{\partial\Pi}\leq 8M$.

            If there are two short sides, at most one can intersect $\partial \Pi$. To see this, note that if there were two short sides, say $q_i$ and $q_j$ for $i< j$, it must be the case that $j = i+2$ and  $q_{i+1}$ is long by (D\ref{alternate}). If the two short sides both intersect $\partial \Pi$, then the endpoints of $q_{i+1}$ lie at distance at most $\abs{q_i} + \abs{q_j} +\abs{\partial\Pi}\leq 10M$ from each other, which contradicts that $q_{i+1}$ is long.

            \begin{figure}
                \centering
                \includegraphics[width=0.5\linewidth]{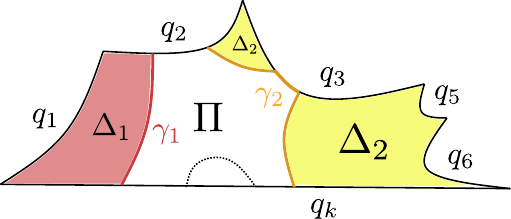}
                \caption{One possible configuration of subdiagrams $\Delta_1$ and $\Delta_2$. Here, $q_2$ is a short side and $\Delta_2$ is a non-simple diagram.}
                \label{fig:short-good}
            \end{figure}

            Since a short side can only contribute $\abs{\partial \Pi}/2$ to \eqref{eq:greendlinger-detail}, $\partial\Pi$ has to intersect $q_k$. Let $i_0$ be the index of a short side intersecting $\partial \Pi$. Let $\gamma_1, \gamma_2$ be the two subsegments of $\partial \Pi$ from $q_{i_0}$ to $q_k$ that intersect $q_k$ and $q_{i_0}$ only in their endpoints. Both $\gamma_1$ and $\gamma_2$ have length less than $\abs{\partial \Pi}/2$ and are thus geodesics by Lemma~\ref{lem:relators-are-isometrically-embedded}.
            
            Consider the two (not necessarily simple) subdiagrams $\Delta_1$ and $\Delta_2$ of $\Delta$ (and hence of $\Delta'$) with $\partial \Delta_1 = q_1 \cdots q_{i_0 -1} \ast [q_{i_0}^-, \gamma_1^-]_{q_{i_0}} \ast \gamma_1 \ast [\gamma_1^+, q_k^+]$ and $\partial \Delta_2 = \gamma_2^{-1}\ast [\gamma_2^-, q_{i_0}^+]_{q_{i_0}}\ast q_{i_0+1}\cdots q_{k-1} \ast [q_k^-, \gamma_2^+]_{q_k}$. This is depicted in Figure~\ref{fig:short-good}.

            Those two diagrams both satisfy the assumptions of Claim~\ref{claim:Igoodshort-detail}; note that $\gamma_1$ and $\gamma_2$ are very short sides and the other sides inherit their label from their label in $\Delta$. Moreover, both diagrams $\Delta_1$ and $\Delta_2$ have fewer cells than $\Delta$. Thus we can apply the induction assumption that Claim~\ref{claim:Igoodshort-detail} holds for $\Delta_1$ and $\Delta_2$, which concludes the proof.
         \end{claimproof}
\end{proof}

\subsection{Intersection functions in expanding graded small-cancellation groups}\label{sec:IntFcnsAndEGSC}

In this section, we prove Proposition~\ref{prop:connection-intersecting-relator-intersecting}, which states that in expanding graded small cancellation groups in which the relators of the same rank have approximately the same length, the relator intersection function (see Definition~\ref{defn:relator-intersection-function}) and the intersection function from Defintion~\ref{def:IntFcn} are linked.  We later use Proposition~\ref{prop:connection-intersecting-relator-intersecting} to show that the group we construct in Section~\ref{sec:nonAHgp} has $\sigma$--compact Morse boundary. 

We first show that projective geodesics in expanding graded small cancellation groups have long subpaths labeled by a subword of a relator.

\begin{prop}\label{prop:projective-from-relators}
    Let $C\geq 1$ be a constant and let $G = \pres = \gpres$ be an $(f, \mu, (\eps_n)_n)$--expanding graded small cancellation group satisfying $\abs{r}\leq C\abs{r'}$ for all $r, r'\in R$ of the same rank. Any $\ell$--projective geodesic $\gamma$ has a subsegment $\gamma'$ such that 
    \begin{enumerate}[(i)]
        \item  $\gamma'$ is labelled by a subword of a relator $r\in R$ with $\abs{r}\leq 2 C\ell$, and \label{prop:proj-from-rel1}
        \item  $\abs{\gamma'}\geq \abs{\gamma} /4- 100f(2C\ell)$. \label{prop:proj-form-rel2}
    \end{enumerate}
\end{prop}

\begin{proof}
    Let $\gamma$ be an $\ell$--projective geodesic, so that there exist $x_1, x_2\in X$ such that $\gamma^+ = p_1$ and $\gamma^- = p_2$ are closest points on $\gamma$ to $x_1$ and $x_2$, respectively, and such that $d(x_1, x_2)\leq d(x_1, p_1)$. Let $\beta_1=[p_1, x_1]$, $\eta=[x_1, x_2]$, and $\beta_2=[x_2, p_2]$. Any point $z\in \eta$ satisfies $d(z, x_2)\leq d(z, \gamma)$, as otherwise $d(x_1,x_2)=d(x_1,z)+d(z,x_2) > d(x_1,z) + d(z,\gamma) \geq d(x_1,\gamma)$, which is a contradiction. Moreover, for $i = 1, 2$, any point $z\in \beta_i$ satisfies $d(z, \gamma) = d(z, p_i)$. 
    
    Let $D$ be a minimal diagram over $G$ with $\partial D = \gamma \ast \beta_1\ast \eta \ast \beta_2$. Note that $\abs{\partial D} = \ell$. In particular, by Lemma~\ref{lem:boundary-proportion}, 
    \begin{align}\label{eq:diameter-cell}
        \abs{\partial \Pi}\leq 2C\abs{\partial D} \leq 2 C\ell
    \end{align}
    for all cells $\Pi$ of $D$. 

    We perform a case distinction to construct a subdiagram $D'$ of $D$ whose Greendlinger cells we understand.

    \textbf{Case 1: There exists at least one cell of $D$ whose boundary intersects $\partial D$ in both $\beta_1$ and $\beta_2$.} Let $\Pi_1$ be the cell with this property that intersects $\beta_1$ closest to $p_1$, as in Figure~\ref{fig:projective-1}. For $i = 1,2$, let $x_i'$ be the point closest to $p_i$ on $\partial \Pi_1\cap \beta_i$.

    \begin{figure}
        \centering
        \includegraphics[width=0.9\linewidth]{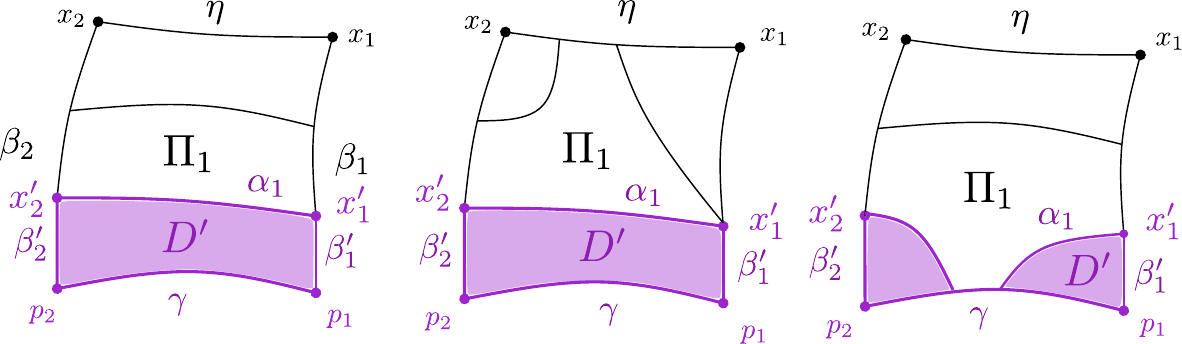}
        \caption{The different possibilities for the diagram $D'$ in Case 1.}
        \label{fig:projective-1}
    \end{figure}
    
    Let $\beta_1'$ and $\beta_2'$ be the subsegments of $\beta_1$ and $\beta_2$ from $p_1$ to $x_1'$ and from $x_2'$ to $p_2$, respectively. Furthermore, let $D'$ be the subdiagram of $D\setminus \Pi_1$ that contains $\gamma$, so that
    \begin{align*}
        \partial D' = [x_1', x_2']_{\partial \Pi_1} \ast \beta_2' \ast \gamma \ast \beta_1'. 
    \end{align*}

    Note that, there are two choices for $[x_1', x_2']_{\partial \Pi_1}=:\alpha_1$; we choose the one for which $\Pi_1$ is not in $D'$.

    \begin{figure}
        \centering
        \includegraphics[width=0.9\linewidth]{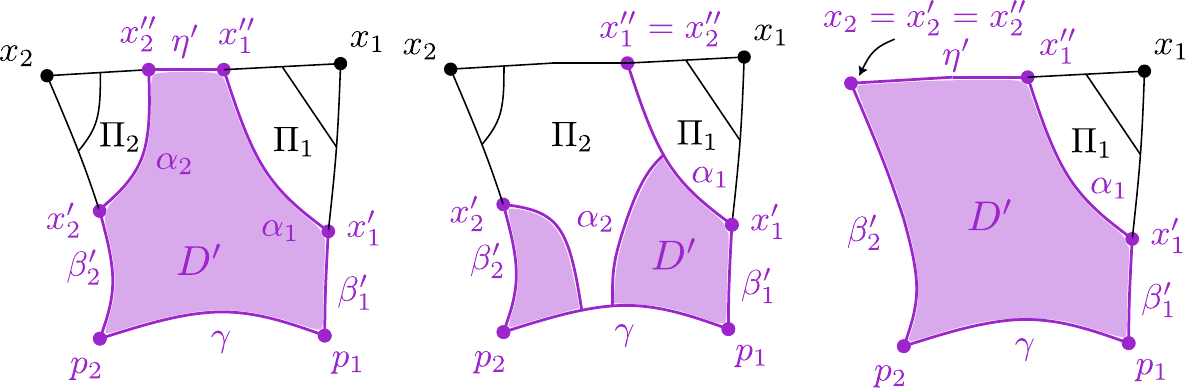}
        \caption{The different possibilities for the diagram $D'$ in Case 2.}
        \label{fig:projective-2}
    \end{figure}
    
    \textbf{Case 2: No cell of $D$ has boundary that intersects both $\beta_1$ and $\beta_2$.} In this case,  let $\Pi_i$ be a cell whose boundary intersects $\partial D$ in both $\eta$ and $\beta_i$ for $i=1,2$, if such cells exist. If there are multiple such cells, we choose the one which intersects $\beta_i$ closest to $p_i$. 
    If $\Pi_i$ exists, let $x_i', x_i''$ be the points on $\beta_i\cap \partial \Pi$ and, respectively, $\eta \cap \partial \Pi_i$ furthest from $x_i$. Further,  let $\alpha_1 = [x_1', x_1'']_{\partial \Pi_1}$ and $\alpha_2 = [x_2'', x_2']_{\partial \Pi_2}$. There are two choices for $\alpha_i$; we choose the one which is in the same component of $D - \Pi_i$ as $\gamma$.
    On the other hand, if $\Pi_i$ does not exist, let $x_i' = x_i'' = x_i$, and let $\alpha_i$ be the trivial path. This is depicted in Figure~\ref{fig:projective-2}.
    Further, define $\beta_1'$ and $\beta_2'$ as the subsegments of $\beta_1$ and $\beta_2$ from $p_1$ to $x_1'$ and from $x_2'$ to $p_2$, respectively, and $\eta'$ as the subsegment of $\eta$ from $x_1''$ to $x_2''$.
   
    Finally, let $D'$ be the subdiagram of $D$ with 
    \begin{align*}
        \partial D' = \beta_2'\ast \gamma \ast \beta_1'\ast \alpha_1 \ast \eta' \ast \alpha_2. 
    \end{align*}

    Hence, both in Case 1 and Case 2, $\partial D' = \beta \ast \alpha$, where $\beta = \beta_2' \ast \gamma \ast \beta_1'$, $\alpha = \alpha_1\ast \eta' \ast \alpha_2$, $\eta'\subset \eta$ and for $i=1, 2$, $\alpha_i$ is either trivial or is contained in the boundary of a cell $\Pi_i$ of $D$; in Case 1, $\eta'$ and $\alpha_2$ are both trivial. Let $\mathfrak r$ be the maximum of the ranks of $\Pi_1$ and $\Pi_2$, where we use the convention that the rank of $\Pi_i$ is 0 if $\alpha_i$ is trivial.

    By construction, no cell in $D'$ has boundary that intersects more than one of $\eta', \beta_1'$ and $\beta_2'$ non-trivially. Moreover, the paths $\alpha_i$ might not necessarily be geodesics, as they could have length more than half of $\abs{\partial \Pi_i}$, but they are the concatenation of at most two geodesics by Lemma~\ref{lem:relators-are-isometrically-embedded}.

    \begin{claim}\label{claim:two-projections-at-most-half}
        If $\Pi'$ is any cell of $D'$, then $\abs{\partial \Pi' \cap \beta}\leq 3\abs{\partial \Pi'} /4$.
    \end{claim}

    \begin{claimproof}
    Since $\partial \Pi'$ cannot intersect both $\beta_1'$ and $\beta_2'$ non-trivially, it suffices to prove that $\partial \Pi' \cap(\gamma \ast \beta_1')$ and $\partial \Pi' \cap (\beta_2' \ast \gamma)$ have length at most $3\abs{\partial \Pi'}/4$. We will only prove the former, as the proof of the latter is analogous.

    If $\partial \Pi'$ does not intersect $\beta_1'$, the statement follows from $\gamma$ being a geodesic, which implies $\abs{\partial \Pi'\cap \gamma}\leq \abs{\partial \Pi'}/2$. If $\partial \Pi'$ does intersect $\beta_1'$, let $z\in \beta_1'$ be the point on $\partial \Pi'$ furthest away from $p_1$, let $L_{\beta}= \abs{\beta_1'\cap \partial \Pi'}$ and let $L_{\gamma} = \abs{\gamma \cap \partial \Pi'}$. There is a path from $z$ to $\gamma$ that follows along $\partial \Pi'$ that only trivially intersects $\beta_1'$ and $\gamma$. We denote the length of this path by $L_{\emptyset}$. 

    Since $\beta_1'$ is a subsegment of $\beta_1$, $d(z, \gamma) \geq d(z, p_1) \geq L_{\beta}$. In particular, $L_{\emptyset}\geq d(z, \gamma) \geq L_{\beta}$. As $\gamma$ is a geodesic, we have  $L_{\gamma}\leq \abs{\partial \Pi'}/2$.
    If $L_{\beta}\leq \abs{\partial \Pi'}/4$, the statement follows. If, instead, $L_{\beta}\geq \abs{\partial \Pi'}/4$, then so is $L_{\emptyset}$. Since $L_{\gamma}+L_{\emptyset} + L_{\beta}\leq \abs{\partial \Pi'}$, the statement again follows.
    \end{claimproof}

    \smallskip
    
    Claim~\ref{claim:two-projections-at-most-half} implies that any Greendlinger cell $\Pi'$ of $D'$ has to satisfy 
    \begin{align}\label{eq:alpha-int}
        \abs{\alpha\cap \partial \Pi'} \geq (1/4 - 24 \mu) \abs{\partial \Pi'}.
    \end{align}

    \begin{claim}\label{claim:small-rank-ofgreendlinger}
        Either $D'$ contains no cell or the rank of its Greendlinger cells is less than $\mathfrak r$.
    \end{claim}    

    Recall that $\mathfrak r$ is the maximum of the ranks of $\Pi_1$ and $\Pi_2$, where if $\Pi_i$ does not exist, we say its rank is $0$, and $\Pi_1, \Pi_2$ are the cells of $D$ containing $\alpha_1$ and $\alpha_2$, unless $\alpha_1$ or $\alpha_2$ is trivial, in which case $\Pi_1$, respectively $\Pi_2$, does not exist).

    \begin{figure}
        \centering
        \includegraphics[width=0.3\linewidth]{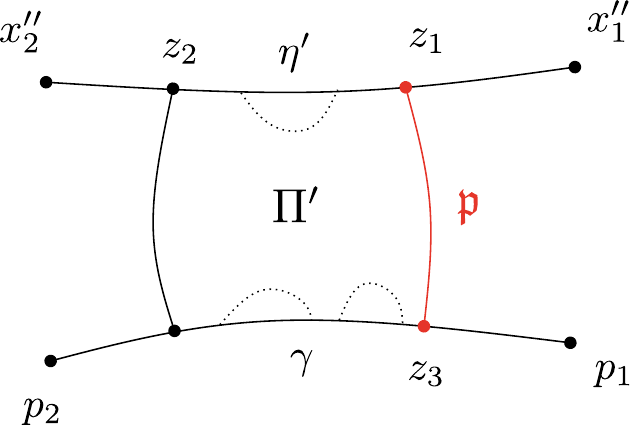}
        \caption{Definition of the path $\mathfrak{p}$ in Claim~\ref{claim:small-rank-ofgreendlinger}. The dotted lines are there to suggest that $\partial \Pi'\cap \gamma$ and $\partial \Pi'\cap \gamma$ do not necessarily need to be contiguous.}
        \label{fig:projective:claim2}
    \end{figure}
    \begin{claimproof}
        Assume toward a contradiction that $D'$ contains at least one cell and that the rank $k$ of a Greendlinger cell $\Pi'$ of $D'$ is at least $\mathfrak r$. Observe that for $i = 1, 2$
        \begin{align}\label{eq:small-prop1}
            \abs{\alpha_i\cap\partial \Pi'}\leq \mu \abs{\partial \Pi'}.
        \end{align}
          To see this, note that if $\alpha_i$ is trivial, this is immediate. If $\alpha_i$ is not trivial, then $\Pi_i$ exists and by assumption its rank is either equal to $k$ or smaller than $k$. If the rank of $\Pi_i$ is equal to $k$, then \eqref{eq:small-prop1} follows from (P\ref{P:small-cancellation-on-level}). If the rank of $\Pi_i$ is smaller than $k$, then \eqref{eq:small-prop1} follows from (P\ref{P:smaller-rank-shorter-relator}).

        Since $\Pi'$ is a Greendlinger cell, \eqref{eq:alpha-int} combined with \eqref{eq:small-prop1} implies that $\abs{\eta' \cap \partial \Pi'}\geq \abs{\partial \Pi'} (1/4 - 26 \mu) > 0$. In particular, the intersection of $\eta'$ and $\partial \Pi'$ is non-trivial, implying that the intersection of $\partial \Pi'$ with $\beta_1'$ and $\beta_2'$ has to be trivial. Consequently, $\abs{\partial \Pi' \cap \gamma}\geq\left ( (1 - 24 \mu ) - 2\mu - 1/2\right) \abs{\partial \Pi'} > 0$, and so $\partial \Pi'$ must intersect $\gamma$.

        Let $z_1, z_2 \in \eta' \subset \eta $ and $z_3 \in \gamma$ be the points contained in $\partial \Pi'$ closest to $x_1'', x_2''$ and $p_1$ respectively. This is depicted in Figure~\ref{fig:projective:claim2}. There is a path $\mathfrak p$ from $z_1$ to $z_3$ following along $\partial \Pi'$ whose intersection with $\eta'$ and $\gamma$ is trivial.  In particular, $\abs{\mathfrak p}\leq 26 \mu \abs{\partial \Pi'}$ and hence
        \begin{align*}
            d(z_1, \gamma) \leq d(z_1, z_3)\leq \abs{\mathfrak p} \leq 26 \mu \abs{\partial \Pi'}.
        \end{align*}
        On the other hand,
        \begin{align*}
            d (z_1, x_2) &\geq d(z_1, z_2) \geq \abs{\partial \Pi' \cap \eta'} \geq (1 -24\mu)\abs{\partial \Pi'} - \abs{\partial \Pi' \cap \beta} - \abs{\partial \Pi' \cap \alpha_1} - \abs{\partial \Pi' \cap \alpha_2}\\
            &\geq (1 - 24 \mu - 1/2 - \mu - \mu )\abs{\partial \Pi'} \geq (1/2 - 26\mu) \abs{\partial \Pi'}.
        \end{align*}
        Since $\mu < 1/1000$ by (P\ref{cond:T5}), this is a contradiction to $d(z_1, x_2)\leq d(z_1, \gamma)$, which has to hold since $z_1$ lies on $\eta'$ and hence lies on $\eta$.
    \end{claimproof}

    \smallskip
    
    By Claim~\ref{claim:small-rank-ofgreendlinger}, the diagram $D'$ can be viewed as a diagram over $G_{\mathfrak r -1}$ and $\partial D'$ can be viewed as a geodesic octagon in the $\eps_{\mathfrak r}$ hyperbolic space $G_{\mathfrak{r}-1}$. Here we have octagon instead of hexagon because a priori the $\alpha_i$ might only be the concatenation of two geodesics instead of actual geodesics.

    Consequently, we can partition $\gamma$ into five possibly empty subsegments $\gamma_1, \ldots ,\gamma_5$ such that $\gamma_5$ is contained in the $6\eps_{\mathfrak r}$--neighbourhood of $\eta'$ and for $j=1,2$ $\gamma_j$ and $\gamma_{j+2}$ are contained in the $6\eps_{\mathfrak r}$--neighbourhoods of $\beta_j'$ and $\alpha_j$, respectively. 

    \begin{claim}\label{claim:small-gamma-j's}
        For $j = 1, 2$ and $5$, $\gamma_j$ has length at most $24\eps_{\mathfrak r}$.
    \end{claim}

    \begin{claimproof}
        For $j= 1, 2$: Let $z$ be a point on $\beta_j'$ with $d(z, \gamma_j)\leq 6\eps_{\mathfrak r}$. Since $\beta_j'$ is a subsegment of $\beta_j$, we have that $d(z, p_j) \leq d(z, \gamma)\leq 6 \eps_{\mathfrak r}$. Hence $\gamma_j$ has to be contained in the $12\eps_{\mathfrak r}$--neighbourhood of $p_j$ and so has length at most $12\eps_{\mathfrak r}$. 
    
        For $j = 5$: Let $z$ be a point on $\eta'$ with $d(z, \gamma_5)\leq 6\eps_{\mathfrak r}$. Since $\eta'$ is a subsegment of $\eta$, $6\eps_{\mathfrak r}\geq d(z, \gamma)\geq d(z, x_2)$, any point on $\gamma_5$ has distance at most $12\eps_{\mathfrak r}$ from $x_2$, implying that the length of $\gamma_5$ is at most $24\eps_{\mathfrak r}$.
    \end{claimproof}
    
    \smallskip
    
    \begin{claim}\label{claim:good-gamma-j's}
        For $j = 3, 4$, $\gamma_{j}$ has two disjoint subsegments ${\gamma'}_j^1$ and ${\gamma'}_j^3$, each of which is either trivial or a subsegment of the boundary of a cell in $D$ and such that $\abs{\gamma_j} - \abs{{\gamma'}_j^1} - \abs{{\gamma'}_j^3} \leq 132\eps_{\mathfrak r} - f(2C\ell)$.
    \end{claim}

    \begin{claimproof}
    Let $j = 3$ or $j = 4$. If $\partial \Pi_{j-2}$ does not exist, then $\alpha_{j-2}$ is trivial and the statement follows immediately. If the rank of $\partial \Pi_{j-2}$ is less than $\mathfrak r$, then $\abs{\alpha_{j-2}}\leq \abs{\partial \Pi_{j - 2}} \eps_{\mathfrak r}$ by (T\ref{T2}), and the statement follows immediately. Assume that the rank of $\Pi_{j-2}$ is $\mathfrak r > 0$. By (P\ref{P:good-segments-of-relators}), we can write $\alpha_{j-2}$ as $\alpha_{j-2}^1\ast \alpha_{j-2}^2\ast \alpha_{j-2}^3$, where $\alpha_{j-2}^1$ and $\alpha_{j-2}^3$ are $(\mathfrak r - 1)$--good and $\alpha_{j-2}^2$ has length at most $f(\abs{\partial \Pi_{j-2}})$. Moreover, we can write $\gamma_j$ as $\gamma_j^1\ast \gamma_j^2 \ast \gamma_j^3$ where for $o = 1, 2, 3$, the endpoints of $\gamma_j^o$ lie in the $6\eps_{\mathfrak r}$--neighbourhood of $\alpha_{j-2}^o$. For $o = 1, 3$, we can apply Proposition~\ref{prop:good-geodesics} to obtain $\abs{\gamma_j^o}\leq 60 \eps_{\mathfrak r}$ or to get a subsegment ${\gamma'}_j^o$ of $\gamma_j^o$ of length at least $\abs{\gamma_j^o} - 60\eps_{\mathfrak r}$ which is a subsegment of $\partial \Pi_{j-2}$. For $o = 2$, we have that $\abs{{\gamma'}_j^o}\leq 12\eps_{\mathfrak r} + f(\abs{\partial\Pi_{j-2}})\leq f(2C\ell)$ by \eqref{eq:diameter-cell}, concluding the statement.
    \end{claimproof}

    \smallskip
    Claims~\ref{claim:small-gamma-j's} and \ref{claim:good-gamma-j's} imply that $\gamma$ has four subsegments $\zeta_1, \ldots, \zeta_4$ which are subsegments of the boundary of a cell of $D$ (or trivial) and such that $\abs{\gamma} - \sum_{i=1}^4\abs{\zeta_i}\leq 3\cdot 24\eps_{\mathfrak r} + 2(132\eps_{\mathfrak r} + f(2C\ell))$. Let $\gamma'$ be the longest of $\zeta_1, \ldots, \zeta_4$, and let $\Pi$ be a cell of $D$ such that $\gamma'\subset \partial \Pi$. With this definition of $\gamma'$, \eqref{prop:proj-from-rel1} follows from \eqref{eq:diameter-cell}. By maximality of the length of $\gamma'$ amongst $\zeta_1, \ldots, \zeta_4$, we have
    \begin{align}\label{eq:proj-final}
        \abs{\gamma'}\geq \left(\abs{\gamma} - 3\cdot 24\eps_{\mathfrak r} - 2(132\eps_{\mathfrak r} + f(2C\ell))\right)/4.
    \end{align}
    
    To show \eqref{prop:proj-form-rel2} holds, it suffices to show that $\eps_{\mathfrak r}\leq f(2C\ell)$. If $\mathfrak r = 0$, then $\eps_{\mathfrak r} = 0$ by (T\ref{cond:T5}) and the statement follows. If $\mathfrak r > 0$, it is equal to the rank of $\Pi_j$ for $j = 1$ or $j = 2$. We have by (T\ref{T2}) and \eqref{eq:diameter-cell} that $\eps_{\mathfrak r}\leq f(\abs{\partial \Pi_j})\leq f(2C\ell)$, concluding the proof.
\end{proof}

The above proposition shows that in expanding graded small-cancellation groups, understanding the intersection function of a geodesic, and hence whether the geodesic is Morse, is related to understanding its intersection with relators. We thus introduce the \emph{relator intersection function} that measures the intersection of geodesics and relators.

\begin{defn}\label{defn:relator-intersection-function}
    Let $G = \pres$ be a group and  $\gamma$ a geodesic in $\cay$. The \emph{relator intersection function $\rint{\gamma} \colon \N \to \N$ of $\gamma$}, is defined as follows: $\rint{\gamma}(n)$ is the maximal length of any word over $S^\pm$ that is both the label of a subpath of $\gamma$ and a subword of a relator $r\in R$ with $\abs{r}\leq n$. The geodesic $\gamma$ is \emph{$\rho$--relator intersecting} for a function $\rho$ if $\rint{\gamma}\leq \rho$.
\end{defn}

The following proposition formalizes the relationship between the intersection function and the relator intersection function of a geodesic in an expanding graded small cancellation group.
\begin{prop}\label{prop:connection-intersecting-relator-intersecting}
    Let $C\geq 1$ be a constant, and let $G = \pres = \gpres$ be an $(f, \mu, (\eps_n)_n)$--expanding graded small cancellation group satisfying $\abs{r}\leq C\abs{r'}$ for all $r, r'\in R$ of the same rank. For every non-decreasing sublinear function $\rho$, there exists a non-decreasing sublinear function $\rho'$ such that the following hold for all geodesics $\gamma$ in $G$.
    \begin{enumerate}
        \item If $\gamma$ is $\rho$--intersecting, then $\gamma$ is $\rho$--relator intersecting.\label{C:int-to-relint}
        \item If $\gamma$ is $\rho$--relator intersecting, then $\gamma$ is $\rho'$--intersecting.\label{C:relint-to-int}
    \end{enumerate}
\end{prop}

\begin{proof}
    We first show that item \ref{C:int-to-relint} holds. Let $\gamma$ be a geodesic which is $\rho$--intersecting, and let $\Gamma$ be an embedded relator labelled by some $r\in R$. Let $\gamma'\subset \gamma \cap \Gamma$ be a subgeodesic of $\gamma$. Define $x = y$ as the point on $\Gamma$ furthest away from $\gamma'$. Since relators are isometrically embedded in $G$ by Lemma~\ref{lem:relators-are-isometrically-embedded}, we have that $\gamma'^+$ and $\gamma'^-$ are closest points on $\gamma'$ to $x$ and $y$. Consequently, $\gamma'$ is $\abs{r}$--projective. Since $\gamma$ is $\rho$--intersecting, we have $\abs{\gamma'}\leq \rho(\abs{r})$, and so $\gamma$ is $\rho$--relator intersecting.
   
    \smallskip
    We now show that item \ref{C:relint-to-int} holds for the function $\rho'(x) =4 \rho(2Cx) + 400f(2Cx)$. Let $\gamma'\subset \gamma$ be an $\ell$--projective geodesic. By Proposition~\ref{prop:projective-from-relators}, there exists a subsegment $\gamma''\subset \gamma'$ with $\abs{\gamma''}\geq \abs{\gamma'}/ 4 - 100f(2C\ell)$ and such that $\gamma''$ is labelled by a subword of a relator $r\in R$ with $\abs{r}\leq 2C\ell$. Since $\gamma$ is $\rho$--relator intersecting, we have $\abs{\gamma''}\leq \rho(2C\ell)$. Consequently, $\abs{\gamma'}\leq 4\rho(2C\ell) + 400f(2C\ell) = \rho'(\ell)$.
\end{proof}

If an $(f, \mu, (\eps_n)_n)$--expanding graded small cancellation group satisfies $\abs{r}\leq C\abs{r'}$ for all $r, r'\in R$ of the same rank, then we say the group is \textit{balanced}.  Using this terminology, Proposition~\ref{prop:connection-intersecting-relator-intersecting} and Lemma~\ref{lem:sublinvsMorse} prove Proposition~\ref{prop:morse-geodesics} from the introduction.

\subsection{Construction of a MLTG group that is not acylindrically hyperbolic}\label{sec:nonAHgp}
In this section, we construct a MLTG group with an infinite-order Morse element that is not acylindrically hyperbolic, proving Theorem~\ref{thm:main2}.  This construction is inspired by the construction of the authors in \cite{AZ} of a finitely generated MLTG group with an infinite-order Morse element that is not loxodromic in any isometric action of the group on a hyperbolic space.  We begin by recalling a key result from that paper, 
which shows that if a cycle $C$ is embedded in a hyperbolic space, then there is a subpath of $C$ with definite length whose endpoints are much closer than its length.

\begin{lem}[{\cite[Lemma~3.4]{AZ}}]\label{lem:shortsubpaths}
    Let $g$ be a sublinear function that is superlogarithmic. For all integers $U\geq 1$ there exists an integer $L\geq U$ such that for all $\delta\geq 0$ there exists $K = K(\delta, U, L)$ such that the following holds. If $C$ is an embedded cycle in a $\delta$-hyperbolic space and $\abs{C}\geq K$, then there exists a subsegment $\lambda$ of $C$ with endpoints $\lambda^-$ and $\lambda^+$ such that   
    \begin{align}\label{eqn:UL}
        \frac{\abs{C}}{L}\leq \abs{\lambda}\leq \frac{\abs{C}}{U}
    \end{align}
    and 
    \begin{align}\label{eqn:CloseEndpoints}
          d(\lambda^-, \lambda^+)\leq g(\abs{C}).
    \end{align}
\end{lem}

Let $g(x) = \log^2(x)$, let $\mu \leq 0.001$ and let $U\geq 1000/\mu$ be a constant. Let $L\geq U$ be as in Lemma~\ref{lem:shortsubpaths} applied to $g$ and $U$.

\begin{defn}
    Let $w$ be a cyclically reduced word. We say that a cyclic subword $u$ of $w$ is a \emph{proper fraction} of $w$ if 
    \begin{align*}
        \frac{\abs{w}}{L}\leq \abs{u}\leq\frac{\abs{w}}{U}.
    \end{align*}
\end{defn}

Let $N\geq2L$, and let $S = \{s_1, \ldots s_{4N}\}$ be a set of formal variables. For $1\leq i \leq 3$ define $r_i^1 = s_{Ni + 1}\cdots s_{Ni +N}$. Let $H=\langle S \mid R\rangle$ with $R=\bigcup_{i=1}^\infty R_i$ be the group constructed in \cite[Section~4]{Zbinden:small-cancellation} with starting words $r_1^1, r_2^1$ and $r_3^1$. We use this group as a basegroup as it has the following properties; any other group satisfying these would also work.
\begin{enumerate}[(H1)]
    \item $H$ is a $C'(4/N)$--small cancellation group.\footnote{It says $C'(1/(4N))$ in \cite{Zbinden:small-cancellation} but this is a typo; it should be $C'(4/N)$.}\label{S:sc}
    \item The set $R_i$ consists of 3 relators of equal length and whose length increases as $i$ increases. This holds because each relator $r^i_k\in R_i$ is a concatenation of $N^2$ many subwords $y$ of relators from $R_{i-1}$, where each $y \subset r^{i_1}_{k'}\in R_{i-1}$ has length $\abs{y} = \abs{r^{i-1}_{k'}}/N$, and the starting relators $r_1^1, r_2^1, r_3^1$ have the same length. \label{S:reslator-set}
    \item Let $r\in R_i$ be a relator, $w\subset r$ a cyclic subword and $j< i$. If $\abs{w} > 2\abs{r'}/N$ for some (equivalently, any) $r'\in R_j$, then there exists $r''\in R_j$ and a common subword $v$ of $r''$ and $w$ with $\abs{v}\geq \abs{r''}/N$. This is due to \cite[Lemma~4.4]{Zbinden:small-cancellation}; in the terminology of that paper, the word $v$ has the form $y^k_{(i, \ell)}$.
    \label{S:relator-intersections}
\end{enumerate}

Fix a non-decreasing sublinear function $f'(x) := \sqrt{\frac{\mu\sqrt{x}}{1000\cdot 101}}$, and let $f(x) := 101(f'(x))^2 = \frac{\mu\sqrt{x}}{1000}$. Let $T = \{t_1, t_2\}$ be a set of two formal variables that are distinct from those in $S$. For $M\geq 1$, let $\mathcal T_M=\{\eta_1,\eta_2,\dots, \eta_{2^M}\}$ be the set of (ordered) words over $T$ of length $M$. Fix an ordering $\alpha$ on the countable set $(S^\pm\cup T^\pm)^* \times \mathbb N$.

Define $k_0 = 0$, $\eps_1 = 0$, $Q_0 = \emptyset$, and $M_0 = 0$. Assume that we have defined $k_j$, $\eps_{j+1}$, $Q_j$, and $M_j$ for all $0\leq j < i$, and let $G_j=\langle S\cup T \mid \bigcup_{m=0}^j Q_m\rangle$. For $i\geq 1$, we will inductively define $k_i$, $\eps_{i+1}, Q_i$ and $M_i$.

Let $\alpha(i) = ( \chi,t)$. Define an integer $k_i > k_{i-1}$  such that the following hold for all $r\in R_{k_i}$.
\begin{enumerate}[(G1)]
    \item  $\abs{r}/L -1\geq M_i:= \log{(6N\abs{r}^2)} > M_{i-1}$;\label{G:lower-bound-on-M}
    \item the word $\chi$ satisfies $\abs{\chi}+1\leq f'(\abs{r}/L)$;\label{G:bound-on-w}
    \item  $f(\abs{r}) > \eps_{i}$; and \label{cond:f-eps-relationship}
    \item $\abs{r}\geq 10^6\eps_i/\mu +1$.\label{G:minimum}
\end{enumerate}

Let $\mc U$ be the set of proper fractions of any of the three relators $r\in R_{k_i}$. For each proper fraction $u\in \mc U$, let $d_u$ be a geodesic word in $G_{i-1}$ representing the same element as $\chi^{f'(\abs{r}/L)}$ in $G_{i-1}$. Define
\begin{align}\label{eqn:relator-defn}
    b_u^L = \left(\prod_{m=1}^N u \eta_{m_u+m}\right), \quad   b_u^R = \left(\prod_{m=1}^N u \eta_{m_u+m+N}\right), \quad b_u=b_u^Lb_u^R , \quad c_u =  b_u^Rxd_uyb_u^L,
\end{align}
where $\eta_{m_u+1},\dots, \eta_{m_u+2N}$ are the lowest indexed words in $\mathcal T_{M_i}$ that have not yet been used and $x=x(u),y=y(u)\in T^\pm$ are letters so that $yb_ux$ and $xd_uy$ are reduced. Finally, let $a_u$ be a geodesic word in $G_{i-1}$ representing the same element as $c_u$ in $G_{i-1}$.

Define $Q_i$ to be the set of all cyclic conjugates and their inverses of words in $R_{k_i}\cup \left (\bigcup_{u\in \mc U} a_u\right )$. Finally, let $\eps_{i+1}=4\cdot\max\{\abs{r}\mid r\in Q_i\}$.

\begin{prop}\label{prop:expgsc}
   The group $G=\npres$ is a well-defined $(f, \mu, (\eps_n)_n)$--expanding graded small cancellation group.
\end{prop}

\begin{proof}
    We have $\eps_1 = 0$ and $\mu \leq 0.001$. Moreover, $f$ is a non-decreasing sublinear function which satisfies (T\ref{cond:T5}) by definition. Thus the group $G_0 = \langle S\cup T \mid \emptyset \rangle$ satisfies the expanding graded small-cancellation condition. 

    We will use induction on $n$ to prove the claims below. Claim~\ref{claim:expanding-graded} holding for all $n$ implies that $G=\npres$ satisfies the $(f, \mu, (\eps_n)_n)$--expanding graded small-cancellation condition.
    
    Let $n\geq 1$, and assume that the claims below hold for all $j <n$. We will prove that they hold for $n$.

    \begin{claim}\label{claim:intersection-bu}
        Let $e= \pm 1$, let $j\leq n$, let $u_1,  u_2$ be proper fractions of relators $r_1\in R_{k_n}$ and $r_2\in R_{k_j}$ and let $w$ be a common subword of $b_{u_1}^e$ and $b_{u_2}$. If $u_1 = u_2$ and $e = 1$, we require that $w$ is a subword of $b_{u_1}$ in two different ways. We have $\abs{w}\leq 2\abs{r_i}/U$ for $i = 1, 2$.
    \end{claim}
    \begin{claimproof}
        Since all the words $\eta_h$ used in the definitions of $b_{u_1}$ and $b_{u_2}$ are distinct and words over $T$ (as opposed to words over $T^\pm$), we know that $w$ has to be a subword of $(\eta_hu_1\eta_{h+1})^e$ and $\eta_{h'}u_2\eta_{h'+1}$ for some $h, h'$. The statement follows.
    \end{claimproof}

    \begin{claim}\label{claim:b_u-are-good}
        Let $r\in \overline{R_{k_n}}$, and let $u$ be a proper fraction of $r$. The words $yb_ux$ and $r$ are $(n-1)$--good.
    \end{claim}
    \begin{claimproof}
        By the choice of $x, y$, the word $yb_ux$ is reduced. The word $r$ is reduced since it is in $\overline{R_{k_n}}$. Let $j<n$, let $s\in Q_j$, let $w$ be a common subword of $r$ and $s$ and let $v$ be a common subword of $xb_uy$ and $s$. 
        
        \smallskip
        
        Case 1: $s\in \overline{R_{k_j}}$. Since $H$ satisfies the $C'(4/N)$--condition, we have $\abs{w}\leq 4\abs{s}/N\leq \abs{s}/100$. Moreover, $v$ has to be a subword of $u$ and hence also has to satisfy $\abs{v} \leq \abs{s}/100$.
        
        \smallskip
        
        Case 2: $s \in \overline{a_{u'}}=\overline{b_{u'}'v_{u'}}$ for some proper fraction $u'$ of a relator $r'\in R_{k_j}$ and $b_{u'}',v_{u'}$ as in Claim~\ref{claim:a_u-is-almost-b_u}. By Claim~\ref{claim:a_u-is-almost-b_u} for $j$, we know that $w = w_1w_2w_3$ and $v = v_1v_2v_3$, where $w_1, w_3$ are common cyclic subwords of $b_{u'}$ and $ b_u$, $v_1, v_3$ are common cyclic subwords of  $b_{u'}$ and $r$, and $w_2, v_2$ have length at most $\mu\abs{s}/1000$. Using Claim~\ref{claim:intersection-bu}, we obtain $\abs{v_1}\leq 2\abs{r'}/U$, implying $\abs{v}\leq \abs{s}/100$. For $w$, observe that any common subword of $b_{u'}$ and $r$ has to be a common subword of $u'$ and $r$. Hence $\abs{w_1}\leq \abs{s}/N$ and $\abs{w_3}\leq \abs{s}/N$, yielding $\abs{w}\leq \abs{s}/100$.      
    \end{claimproof}

    \begin{claim}\label{claim:a_u-is-almost-b_u}
        Let $u$ be a proper fraction of a relator $r\in R_{k_n}$. The word $a_u$ is a cyclic geodesic in $G_{n-1}$ and satisfies $4N\abs{u}\geq \abs{a_u}\geq \abs{r}$. Moreover, there exists a cyclic shift $a_u' = b_u'v_u$ of $a_u$ such that $b_u'$ is a subword of $b_u$ and $\abs{v_u}\leq f(\abs{u})\leq \mu\abs{a_u}/1000$.
    \end{claim}
    \begin{claimproof}
        Define a loop $q_1\cdots q_6$ in $G_{n-1}$ such that $q_1, q_2, q_3, q_4, q_5$ are labeled by $b_u^Rx, d_u, yb_ux, d_u$ and $yb_u^L$, respectively, and $q_6$ is a geodesic. Note that the label of $q_6$ represents the same element as $a_u^2$.  This fits the setting of Lemma~\ref{lem:IgoodShort} with $\abs{d_u}$ playing the role of $M$, as Claim~\ref{claim:b_u-are-good} implies that $b_u^Rx$, $yb_ux$, and $yb_u^L$ are all $(n-1)$--good and, as each contains $u$ as a subword, have length at least $100\abs{d_u}$. Hence, Lemma~\ref{lem:IgoodShort} implies that $q_6$ goes through the midpoint $m$ of $q_3$. In particular, $\abs{q_6} = d(q_1^-, m) + d(m, q_5^+) = 2\abs{a_u}$.  Thus, $a_u^2$ is geodesic, implying that $a_u$ is a cyclic geodesic. 

        Thus, we can assume that $q_6$ is labeled by $a_u^2$. Lemma~\ref{lem:IgoodShort} then further implies that there exists a word $b_u'\subset b_u$ of length at least $\abs{b_u} - 100\abs{d_u}$ which is a subword of $a_u^2$, implying the existence of a cyclic shift $a_u'$ as desired. Consequently, $\abs{b_u}+2+\abs{d_u}\geq \abs{a_u}\geq \abs{b_u} - 100\abs{d_u}$. The conditions in the construction yield $4N\abs{u}\geq \abs{a_u}\geq \abs{r}$ and  $\abs{v_u} = \abs{a_u} - \abs{b_u'}\leq 101\abs{d_u}+2\leq f(\abs{u})\leq \mu\abs{a_u}/1000$. 
    \end{claimproof}

    \begin{claim}\label{claim:C'(1/3mu)}
        The set $Q_n$ satisfies the $C'(\mu/3)$--condition.
    \end{claim}
    \begin{claimproof}
        Let $s \neq s'\in Q_n$, and let $w$ be a common prefix of $s$ and $s'$. We want to show that $\abs{w}< \mu\abs{s}/3$. 
       
        If $s, s'\in \overline{R_{k_n}}$, this follows from $H$ satisfying the $C'(4/N)$ condition, since $N\geq 12/\mu$. 
        
        If $s\in \overline{R_{k_n}}$ and $s'\in \overline{a_u}$ for some proper fraction $u$ of a relator $r\in R_{k_n}$, then $w$ cannot contain letters from $T^{\pm}$. Thus the longest common subword of $b_u$ and $s$ has length at most $\abs{u}$. By Claim~\ref{claim:a_u-is-almost-b_u}, $w = w_1w_2w_3$, where $w_2$ is a subword of $v_u$ and $w_1, w_3$ are subwords of $b_u$. The length estimates from Claim~\ref{claim:a_u-is-almost-b_u} and the fact that $U\geq 1000/\mu$ yield $\abs{w}\leq 2\abs{u}+\abs{v_u}\leq \mu\abs{s}/3\leq \mu\abs{s'}/3$.

        Lastly, if $s\in \overline{a_u}$ and $s'\in \overline{a_{u'}}$ for some proper fractions $u, u'$ of relators $r, r'\in R_{k_n}$, then by Claim~\ref{claim:a_u-is-almost-b_u},  $w$ can be written as $w = w_1w_2w_3 = w_1'w_2'w_3'$, where $w_2, w_2'$ are subwords of $v_u, v_u'$, $w_1, w_3$ are subwords of $b_u^\pm$ and $w_1',w_3'$ are subwords of $b_{u'}^\pm$. By Claim~\ref{claim:intersection-bu}, any common subword of $w_i, w_j'$ for $i, j \in\{ 1,3\}$ has length at most $3\abs{r}/U$. Hence $\abs{w}\leq 12 \abs{r}/U + \abs{v_u}+\abs{v_u'}$. Thus, by the length estimates in Claim~\ref{claim:a_u-is-almost-b_u}, we have $\abs{w}\leq \mu\abs{a_u}/3$.
    \end{claimproof}

    \begin{claim}\label{claim:complicated-C}
        The group $G_n$ satisfies the $C(\eps_{n}, \mu, 10^6\eps_n/\mu +1)$--condition over $G_{n-1}$.
    \end{claim}
    \begin{claimproof}
        We will use Lemma~\ref{lem:remove-eps-pieces}, so it suffices to prove conditions ($\tilde{C}$\ref{tildeC1})-($\tilde{C}$\ref{tildeC4}). 

        ($\tilde{C}$\ref{tildeC1}): Let $s\in Q_n$. If $s\in \overline{R_{k_n}}$,then $s$ is $(n-1)$--good by Claim~\ref{claim:b_u-are-good} and hence a geodesic in $G_{n-1}$ by Proposition~\ref{prop:good-geodesics}(i). Otherwise, $s\in \overline{a_u}$ for some proper fraction $u$ of a relator $r\in R_{k_n}$, in which case $s$ is geodesic by Claim~\ref{claim:a_u-is-almost-b_u}.
        
        ($\tilde{C}$\ref{tildeC2}): Let $s\in Q_n$. We have $\abs{s}\geq \abs{r}$ for some $r\in R_{k_n}$: if $s\in \overline{R_{k_n}}$ this is immediate, and otherwise this follows from Claim~\ref{claim:a_u-is-almost-b_u}. Hence $\abs{s}\geq 10^6\eps_n/\mu +1$ by (G\ref{G:minimum}).
        
        ($\tilde{C}$\ref{tildeC3}): This holds by Claim~\ref{claim:C'(1/3mu)}.
        
        ($\tilde{C}$\ref{tildeC4}): If $s\in \overline{R_{k_n}}$, then $s$ is $(n-1)$--good by Claim~\ref{claim:b_u-are-good}, so we can set $s_g = s$. Otherwise, $s\in \overline{a_u}$ for some proper fraction $u$, and using the notation from Claim~\ref{claim:a_u-is-almost-b_u}, we choose $s_g = b_u'$. By Claim~\ref{claim:b_u-are-good}, $b_u'$ is $(n-1)$--good, and $\abs{a_u} - \abs{b_u'} = \abs{v_u}\leq f(\abs{u})$ by Claim~\ref{claim:a_u-is-almost-b_u}.
    \end{claimproof}

    The following claim concludes the induction step: 
    \begin{claim}\label{claim:expanding-graded}
        The group $G_n$ satisfies the $(f, \mu, (\eps_n)_n)$--expanding graded small cancellation condition. 
    \end{claim}
    \begin{claimproof}
        The group $G_{n-1}$ satisfies the $(f, \mu, (\eps_n)_n)$--expanding graded small cancellation condition: if $n = 1$, we only only have to show that (T\ref{cond:T4}) holds, which it does by construction, while if $n  >1$, this is true by induction. The relators of the groups $G_n$ and $G_{n-1}$ only differ at level $n$. It thus suffices to show that $\eps_{n+1}\geq 4\abs{q}$ and $f(\abs{q}) > \eps_n$ for all relators $q\in Q_n$, and that (T\ref{T:C-cond}) and (T\ref{cond:T4}) hold for $n$.

        Let $q_{n}\in Q_{n}$. Then $\abs{q_n}\geq \abs{r_n}$ for some $r_n\in R_{k_n}$: this follows from Claim~\ref{claim:a_u-is-almost-b_u} if $q_n\not \in \overline{R_{k_n}}$ and is immediate if $q_n\in \overline{R_{k_n}}$. Hence, (G\ref{cond:f-eps-relationship}) implies $f(\abs{q_n}) > \eps_n$. The inequality $\eps_{n+1}\geq 4\abs{q_n}$ holds by definition of $\eps_{n+1}$.

        (T\ref{cond:T4}): This is the same condition as ($\tilde{C}$\ref{tildeC4}) and proven in Claim~\ref{claim:complicated-C}. 
        
        (T\ref{T:C-cond}): This holds by Claim~\ref{claim:complicated-C}.
    \end{claimproof}
\end{proof}

\begin{prop}
    The group $G=\npres$ is MLTG and contains a Morse element. 
\end{prop}

\begin{proof}
    To show $G$ is MLTG, we will show that it has $\sigma$--compact Morse boundary, which suffices by~\Cref{thm:main1}. For $i\geq 1$, define the sublinear functions 
    \begin{align*}
        \rho_i(t) = i + 2\log(6N^3t^2) + f(t).
    \end{align*}

    Recall that by the construction of $H$, all relators of $R_{k_i}$ have the same length. Thus by Claim~\ref{claim:a_u-is-almost-b_u} we have $4N\abs{r}\geq \abs{q}\geq \abs{r}$ for all $r\in R_{k_i}$ and $q\in Q_i - \overline{R_{k_i}}$. Hence the assumptions of Proposition~\ref{prop:connection-intersecting-relator-intersecting} are satisfied. Therefore, by Proposition~\ref{prop:connection-intersecting-relator-intersecting} and Lemmas~\ref{lem:sublinvsMorse} and \ref{lem:sigma_cpt2}, to prove that $G$ is Morse local-to-global, it suffices to show that any geodesic in $G$ with sublinear relator intersection function is $\rho_i$--relator intersecting for some $i$.
    
    Let $\lambda$ be a geodesic in $G$ with sublinear relator intersection function $\rho_\lambda$. Then there exists $T_0 > 0$ such that $\rho_\lambda(t) <  t/ N$ for all $t\geq T_0$. Let $i_0$ be such that $\abs{r_0} > T_0$ for any (equivalently, all) $r_0\in R_{k_{i_0}}$, and let $T_1 = \abs{r_0}$. We claim that $\lambda$ is $\rho_{4NT_1}$--relator intersecting.

    \begin{claim}\label{claim:constant-intersection}
         If $r\in \overline{R_j}$ and $w$ is a common subword of $r$ and $\lambda$, then $\abs{w}\leq T_1$.
    \end{claim}
    \begin{claimproof}
        Assume $\abs{w} > T_1$. Then $j>k_{i_0}$, and so by (H\ref{S:relator-intersections}), there exists $r'\in R_{k_{i_0}}$ and a common subword $v$ of $w$ and $r'$ or $(r')^{-1}$ such that $\abs{v}\geq \abs{r'}/N = T_1/N$. Hence $\rho_{\lambda}(T_1) \geq \abs{v} \geq T_1/N$, a contradiction to the definition of $T_0$.
    \end{claimproof}

    \smallskip
    
    Let $q \in \overline{a_u}$ for some proper fraction $u$ of a relator $r\in R_{k_j}$ for some $j$, and let $w$ be a common subword of $q$ and $\lambda$. We have $\abs{a_u}\leq 4N\abs{u}$ by Claim~\ref{claim:a_u-is-almost-b_u}. Thus, if $\abs{u}\leq T_1$, then $\abs{w}\leq \rho_{4NT_1}(\abs{q})$.

    Suppose $\abs{u} > T_1$. Claim~\ref{claim:constant-intersection} implies that the longest common subword of $u$ and $\lambda$ has length at most $T_1$, and so $u$ cannot be a subword of $w$. Using Claim~\ref{claim:a_u-is-almost-b_u} and the notation thereof, we obtain that $w$ has to be contained in a word of the form $u\eta_1v_u\eta_2u$, or $u\eta_3u$, where the words $\eta_k$ have length at most $M_j$ and $\abs{v_u}\leq f(\abs{u})$. Consequently, $\abs{w}\leq 2T_1 + 2M_j + f(\abs{u})\leq \rho_{4NT_1}(\abs{a_u})$. Hence, $\lambda$ is $\rho_{4NT_1}$--relator-intersecting, concluding the proof that $G$ is MLTG.

    \smallskip
    
    We next show that any generator $x\in T$ is an infinite order Morse element by showing that $\gamma$, the path starting at the identity and labeled by $x^\infty$ is a Morse geodesic. Let $\rho$ be the relator intersection function of $\gamma$. To prove that $\gamma$ is a geodesic it suffices to show that $\gamma$ is $\infty$--good, that is, that $\rho(t)\leq t/100$ for all $t$. If $\rho$ is also sublinear, then $\gamma$ is Morse, and hence $x$ is a Morse element.

    Let $q\in Q_i$ be a relator. Then either $q\in \overline{R_{k_i}}$, in which case its intersection with $\gamma$ is trivial, or $q \in  \overline{a_u}$ for some proper fraction $u$ of a relator $r\in R_{k_i}$. Using the notation of Claim~\ref{claim:a_u-is-almost-b_u} and the construction of $G$, we obtain that any common subword $w$ of $q$ and $\gamma$ is contained in $\eta_1v_u\eta_2$, where $\eta_1, \eta_2\in \mc T_{M_i}$ and $\abs{v_u}\leq f(\abs{u})\leq \mu\abs{a_u}/1000$. The upper bounds on $M_i$ and $f$ yield the desired upper bounds on $\abs{w}$.
\end{proof}

The proof of the next proposition is similar to the proof of \cite[Lemma~4.6]{AZ}.
\begin{prop}
    The group $G=\npres$ is not acylindrically hyperbolic.
\end{prop}

\begin{proof}
Suppose $G$ acts on a $\delta$--hyperbolic space $X$, and fix a basepoint $x_0\in X$.  By \cite[Theorem~1.2]{O:acylindrical}, it suffices to show that no infinite-order element of $G$ is a loxodromic isometry of $X$.  To that end, let $w\in G$ be any infinite order element.  We will show that $d_X(x_0,w^nx_0)$ is sublinear in $n$, which implies that $w$ is not a loxodromic isometry of $X$.

Let $K$ be the constant from Lemma~\ref{lem:shortsubpaths}, and choose $n>K$.  There exists $r\in R$ such that $f'(\abs{r}/L)>n$ and $\abs{r}>N_0$.  We view $r$ as an embedded cycle in $X$. By Lemma~\ref{lem:shortsubpaths}, the relator $r$ has a proper fraction $u$ such that $d_X(u^-,u^+)\leq g(\abs{r})$. Corresponding to this proper fraction $u$, the group $G$ has a relator $a_u$, which is a geodesic word representing the same element as $c_u=b_u^Rxd_uyb_u^L$, where $d_u$ is a geodesic word representing the same element as $w^{f'(\abs{r}/L)}$ and $b_u^R,b_u^L$ are defined in \eqref{eqn:relator-defn}. Recall that each $\eta_{m_u+m}$ is chosen from $\mc T_{M_i}$, the set of words in $\mc T$ of length $M_i$, and $M_i=\log_2(6N\abs{r}^2)$ by (G\ref{G:lower-bound-on-M}).  Therefore, by the triangle inequality, the endpoints of $w^{f'(\abs{r}/L)}$ in $X$ are at distance at most
\begin{align*}
    \sum_{m=1}^{2N}\abs{\eta_{m_u+m}} + \sum_{i=1}^{2N} d_X(u^-,u^+)+ 2 
    \leq 2N\log(6N\abs{r^2}) + 2Ng(\abs{r}) + 2 :=A(\abs{r}).
\end{align*}
Since $f'$ is superlogarithmic and grows faster than $g$, the function $A(x)$ grows slower than $f'$.  Therefore, $d_X(x_0,w^nx_0)$ grows sublinearly, and so $w$ is not a loxodromic isometry of $X$, concluding the proof.
\end{proof}

\bibliographystyle{alpha}
\bibliography{mybib}

\end{document}